\documentclass [12pt]{article}
\usepackage {a4}
\usepackage {amsfonts}
\usepackage {amssymb,amsmath,amsthm}
\usepackage [utf8]{inputenc}
\usepackage [english]{babel}
\usepackage{enumerate}
\usepackage{graphicx}
\usepackage{fullpage}
\usepackage{bbm}
\usepackage{hyperref}
\usepackage{todonotes}
\usepackage{verbatim}

\newcommand{\id}{\mathbbm{1}}
\newcommand{\norm}[1]{\left\Vert #1 \right\Vert}

\newcommand{\R}{\mathbb{R}}

\newcommand{\Pp}{\mathbb{P}}
\newcommand{\DD}{\mathcal{D}}
\newcommand{\DDD}{\widetilde{\mathcal{D}}}

\newcommand{\MM}{\mathcal{M}}
\newcommand{\PP}{\mathcal{P}}
\newcommand{\EE}{\mathcal{E}}
\newcommand{\E}{\mathbb{E}}

\newcommand{\Oo}{\mathcal{O}}
\newcommand{\Ll}{\mathcal{L}}

\newcommand{\Ss}{\mathcal{S}}

\newcommand{\dd}{\mathrm{d}}

\newcommand{\supp}{\mathrm{supp}}

\newcommand{\sdiff}{\mathrm{SDiff}}

\newtheorem{theorem}{Theorem}[section]
\newtheorem*{theorem*}{Theorem}

\newtheorem{lemma}[theorem]{Lemma}
\newtheorem{corollary}[theorem]{Corollary}
\newtheorem{proposition}[theorem]{Proposition}

\newtheorem{theoremalph}{Theorem}[section]

\theoremstyle{definition}
\newtheorem{remark}[theorem]{Remark}
\newtheorem{assumption}[theorem]{Assumption}

\author{Michał Kotowski \and B\'{a}lint Vir\'{a}g}

\title{Large deviations for the interchange process on the interval and incompressible flows}

\begin{document}
\maketitle

\begin{abstract}
We use the framework of permuton processes to show that large deviations of the interchange process are controlled by the Dirichlet energy. This establishes a rigorous connection between processes of permutations and one-dimensional incompressible Euler equations. While our large deviation upper bound is valid in general, the lower bound applies to processes corresponding to incompressible flows, studied in this context by Brenier. These results imply the Archimedean limit for relaxed sorting networks and allow us to asymptotically count such networks.
\end{abstract}

\tableofcontents

\begin{section}{Introduction}

In this paper we investigate the large deviation principle for a model of random permutations called the one-dimensional \emph{interchange process}. The process can be roughly described as follows. We put $N$ particles, labelled from $1$ to $N$, on a line $\{1,\ldots, N\}$ and at each time step perform the following procedure: an edge is chosen at random and adjacent particles are swapped. By comparing the particles' initial positions with their positions after given time $t$ we obtain a random permutation from the symmetric group $\Ss_N$ on $N$ elements.

The interchange process on the interval (whose discrete time analog is known as the adjacent transposition shuffle) and on more general graphs has attracted considerable attention in probability theory, for example with regard to the analysis of mixing times. It is natural to ask whether, after proper rescaling and as $N \to \infty$, the permutations obtained in the interchange process converge in distribution to an appropriately defined limiting process.

Such limits have been recently studied (\cite{limits}, \cite{mustazee}) under the name of permutons and permuton processes. These notions have been inspired by the theory of graph limits (\cite{lovasz}), where the analogous notion of a graphon as a limit of dense graphs appears. A \emph{permuton} is a Borel probability measure on $[0,1]^2$ with uniform marginals on each coordinate. A sequence of permutations $\sigma^N \in \Ss_N$ is said to converge to a permuton $\mu$ as $N \to \infty$ if the corresponding empirical measures
\[
\frac{1}{N}\sum_{i=1}^N \delta_{\left(\frac{i}{N},\frac{\sigma^N(i)}{N}\right)}
\]
converge weakly to $\mu$. A \emph{permuton process} is a stochastic process $X = (X_{t}, 0 \leq t \leq T)$ taking values in $[0,1]$, with continuous sample paths and having uniform marginals at each time $t \in [0,T]$. A permutation-valued path, such as a sample from the interchange process, is said to converge to $X$ if the trajectory of a randomly chosen particle converges in distribution to $X$.

Depending on the time scale considered, one observes different asymptotic structure in the permutations arising from the interchange process. If the average number of all swaps is greater than $\sim N^3 \log N$, the process will be close to its stationary distribution (\cite{aldous}, \cite{lacoin}), which is the uniform distribution on $\Ss_N$. For $\sim N^3$ swaps each particle has displacement of order $N$ and the whole process converges, in the sense of permuton processes, to a Brownian motion on $[0,1]$ (\cite{mustazee2}).

Here we will be interested in yet shorter time scales, corresponding to $\sim N^{2+\varepsilon}$ swaps for fixed $\varepsilon \in (0,1)$. In this scaling each particle has displacement $\ll N$, so the resulting permutations will be close to the identity permutation. Nevertheless, in the spirit of large deviation theory one can still ask questions about rare events, for example ``what is the probability that starting from the identity permutation we are close to a fixed permuton after time $t$?'' or, more generally, ``what is the probability that the interchange process behaves like a given permuton process $X$?''. We expect such probabilities to decay exponentially in $N^\gamma$ for some $\gamma > 0$, with the decay rate given by a \emph{rate function} on the space of permuton processes.

The large deviation principle we obtain in this paper can be informally summarized as follows: for a class of permuton processes solving a natural energy minimization problem, the probability $\Pp(A)$ that the interchange process is close in distribution to a process $X$ satisfies asymptotically
\begin{equation}\label{eq:informal-main-result}
\frac{1}{N^{\gamma}} \log \Pp(A) \approx - I(X),
\end{equation}
where $\gamma = 2 - \varepsilon$ and $I(X)$ is the \emph{energy} of $X$, defined as the expected Dirichlet energy of a path sampled from $X$. Apart from a purely probabilistic interest, the result is relevant to two other seemingly unrelated subjects, namely the study of Euler equations in fluid dynamics and the study of sorting networks in combinatorics.

Let us first state the energy minimization problem in question, which is as follows -- given a permuton $\mu$, find
\begin{equation}\label{eq:energy-min}
\inf\limits_{(X_0, X_T) \sim \mu} I(X),
\end{equation}
where the infimum is over all permuton processes $X$ such that $(X_0, X_T)$ has distribution $\mu$. As it happens, such energy-minimizing processes have been considered in fluid dynamics in the study of \emph{incompressible Euler equations}, under the name of \emph{generalized incompressible flows}. This connection is discussed in more detail in Section \ref{sec:incompressible-flows}. Very roughly speaking, Euler equations in a domain $D \subseteq \R^d$ describe motion of fluid particles whose trajectories satisfy the equation
\begin{equation}\label{eq:euler-informal}
x''(t) = - \nabla p (t, x)
\end{equation}
for some function $p$ called the \emph{pressure}. The incompressibility constraint means that the flow defined by the equation has to be volume-preserving. Classical, smooth solutions to Euler equations correspond to flows which are diffeomorphisms of $D$. Generalized incompressible flows are a stochastic variant of such solutions in which each particle can choose its initial velocity independently from a given probability distribution.

It turns out that, under additional regularity assumptions, such generalized solutions to Euler equations \eqref{eq:euler-informal} for $D = [0,1]$ correspond exactly to permuton processes solving the energy minimization problem \eqref{eq:energy-min} for some permuton $\mu$. Our large deviation result \eqref{eq:informal-main-result} is valid precisely for such energy-minimizing permuton processes (again, under certain regularity assumptions).

As it happens, the original motivation for our work came from a different direction, namely from the study of \emph{sorting networks} in combinatorics. This connection is explained in more detail below. Using our large deviation principle \eqref{eq:informal-main-result}, we are able to prove novel results on a variant of the model we call \emph{relaxed sorting networks}. Thus the large deviation principle presented in this paper provides a rather unexpected link between problems in combinatorics (sorting networks) and fluid dynamics (incompressible Euler equations), along with a quite general framework for analyzing permuton processes which we hope will find further applications.

\paragraph{Main results.}

Let us now state our main results more formally, still with complete definitions and discussion of assumptions deferred until Sections \ref{sec:stationarity} and \ref{sec:ode-part}. Let $\DD = \DD([0,T], [0,1])$ be the space of c\`adl\`ag paths from $[0,T]$ to $[0,1]$ and let $\MM(\DD)$ be the space of Borel probability measures on $\DD$. Let $\PP \subseteq \MM(\DD)$ denote the space of permuton processes and their approximations by permutation-valued processes. For $\pi \in \MM(\DD)$ by $I(\pi)$ we will denote the expected Dirichlet energy of the process $X$ whose distribution is $\pi$.

Let $\eta^N$ denote the interchange process in continuous time on the interval $\{1, \ldots, N\}$, speeded up by $N^{\alpha}$ for some $\alpha \in (1,2)$. Let $\gamma = 3 - \alpha$. We have the following large deviation principle

\begin{theoremalph}[Large deviation lower bound]\label{th:theorem-main-lower}
Let $\Pp^{N}$ be the law of the interchange process $\eta^N$ and let $\mu^{\eta^{N}} \in \MM(\DD)$ be the empirical distribution of its trajectories. Let $\pi$ be a permuton process which is a generalized solution to Euler equations \eqref{eq:gen-euler}. Provided $\pi$ satisfies Assumptions \eqref{as:main-assumptions}, for any open set $\Oo \subseteq \PP$ such that $\pi \in \Oo$ we have
\[
\liminf_{N \to \infty} N^{-\gamma} \log \Pp^{N}\left(\mu^{\eta^{N}} \in \Oo\right) \geq - I(\pi).
\]
\end{theoremalph}

\begin{theoremalph}[Large deviation upper bound]\label{th:theorem-main-upper}
Let $\Pp^{N}$ be the law of the interchange process $\eta^{N}$ and let $\mu^{\eta^{N}} \in \MM(\DD)$ be the empirical distribution of its trajectories. For any closed set $\mathcal{C} \subseteq \PP$ we have
\[
\limsup_{N \to \infty} N^{-\gamma} \log \Pp^{N}\left(\mu^{\eta^{N}} \in \mathcal{C} \right) \leq - \inf\limits_{\pi \in \mathcal{C}} I(\pi).
\]
\end{theoremalph}

The results are referred to as respectively Theorem \ref{thm:lower-bound-for-minimizers} and Theorem \ref{th:upper-bound} in the following sections. Here the large deviation upper bound is valid for \emph{all} permuton processes, without any additional assumptions. On the other hand, in the proof of the lower bound we exploit rather heavily the special structure possessed by generalized solutions to Euler equations. We expect the lower bound to hold for arbitrary permuton processes as well, since one can locally approximate any permuton process by energy minimizers. However, for our techniques to apply one would need to understand in more detail regularity of the associated velocity distributions and pressure functions, which falls outside the scope of our work.

The reader may notice that the rate function, which is the energy $I(\pi)$, is similar to the one appearing in the analysis of large deviations for independent random walks. In fact, the crux of our proofs lies in proving that particles in the interchange process and its perturbations are in a certain sense almost independent. The main techniques used here come from the field of interacting particle systems. A comprehensive introduction to the subject can be found in \cite{kipnis}. The novelty in our approach is in applying tools usually used to study hydrodynamic limits to a setting which is in some respects more involved, since the limiting objects we consider, permuton processes, are stochastic processes instead of deterministic objects like solutions of PDEs apearing, for example, for exclusion processes.

\paragraph{Sorting networks and the sine curve process.} The large deviation bounds can be applied to obtain results on a model related to sorting networks. A \emph{sorting network} on $N$ elements is a sequence of $M = \binom{N}{2}$ transpositions $(\tau_{1}$, $\tau_{2}$, $\ldots$, $\tau_{M})$ such that each $\tau_{i}$ is a transposition of adjacent elements and $\tau_{M} \circ \ldots \circ \tau_{1} = \mathrm{rev}_{N}$, where $\mathrm{rev}_{N} = (N \,\ldots \,2 \, 1)$ denotes the \emph{reverse permutation}. It is easy to see that any sequence of adjacent transpositions giving the reverse permutation must have length at least $\binom{N}{2}$, hence sorting networks can be thought of as shortest paths joining the identity permutation and the reverse permutation in the Cayley graph of $\Ss_{N}$ generated by adjacent transpositions.

A \emph{random sorting network} is obtained by sampling a sorting network uniformly at random among all sorting networks on $N$ elements. Let us work in continuous time, assuming each transposition $\tau_i$ happens at time $\frac{i}{M+1}$. It was conjectured in \cite{sorting} and recently proved in \cite{duncan} that the trajectory of a randomly chosen particle in a random sorting network has a remarkable limiting behavior as $N \to \infty$, namely it converges in the sense of permuton processes to a deterministic limit, which is the sine curve process described below.

Here it will be more natural to consider the square $[-1,1]^2$ and processes with values in $[-1,1]$ instead of $[0,1]$ (with the obvious changes in the notion of a permuton and a permuton process which we leave implicit). The \emph{Archimedean law} is the measure on $[-1,1]^2$ obtained by projecting the normalized surface area of a $2$-dimensional half-sphere to the plane or, equivalently, the measure supported inside the unit disk $\{x^2 + y^2 \leq 1\}$ whose density is given by $1 /(2\pi \sqrt{1 - x^2 - y^2}) \, dx \, dy$. Observe that thanks to the well-known plank property each strip $[a,b] \times [-1,1]$ has measure proportional to $b-a$, hence the Archimedean law defines a permuton.

The \emph{sine curve process} is the permuton process $\mathcal{A} = (\mathcal{A}_{t}, 0 \leq t \leq 1)$ with the following distribution -- we sample $(X,Y)$ from the Archimedean law and then follow the path
\[
\mathcal{A}_t = X \cos \pi t + Y \sin \pi t.
\]
One can directly check that $\mathcal{A}_t$ has uniform distribution on $[-1,1]$ at each time $t$, hence $\mathcal{A}_t$ indeed defines a permuton process. Observe that $(\mathcal{A}_{0}, \mathcal{A}_{0}) = (X,X)$ and $(\mathcal{A}_{0}, \mathcal{A}_{1}) = (X, -X)$, thus the sine curve process defines a path between the identity permuton and the \emph{reverse permuton}.

An equivalent way of describing the sine curve process consists of choosing a pair $(R, \theta)$ at random, where the angle $\theta$ is uniform on $[0,2\pi]$ and $R$ has density $r / 2 \pi \sqrt{1-r^2} \, dr$ on $[0,1]$, and following the path $\mathcal{A}_t = R \cos(\pi t + \theta)$. Thus the trajectories of this process are sine curves with random initial phase and amplitude -- the path of a random particle is determined by its initial position $X$ and velocity $V$, given by $(X,V) = (R \cos \theta, -\pi R \sin \theta)$.

Recall now the energy minimization problem \eqref{eq:minimize-permuton}. The sine curve process is the unique minimizer of energy among all permuton processes joining the identity to the reverse permuton (\cite{brenier}, see also \cite{mustazee}), with the minimal energy equal to $I(\mathcal{A}) = \frac{\pi^2}{6}$. It is one of the few examples where the solution to the problem \eqref{eq:minimize-permuton} can be explicitly calculated for a target permuton $\mu$. It also seems to play a special role in constructing generalized incompressible flows which are non-unique solutions to the energy minimization problem in dimensions greater than one, see, e.g., \cite{bernot-figalli-santambrogio}.

The sine curve process is a generalized solution to Euler equations with the pressure function $p(x) = \frac{x^2}{2}$, which unsurprisingly leads to each particle satisyfing the harmonic oscillator equation $x'' = -x$. The reader may check that the sine curve process satisfies the Assumptions \eqref{as:main-assumptions} (with the velocity distribution being time-independent), thus providing a non-trivial and explicit example for which our large deviation bounds hold. To the best of our knowledge the connection between sorting networks on the one hand and Euler equations on the other hand was first observed in the literature in \cite{duncan}.

Let us now describe the results on relaxed sorting networks. Fix $\delta > 0$ and $N \geq 1$. We define a $\delta$-\emph{relaxed sorting network} of length $M$ on $N$ elements to be a sequence of $M$ adjacent transpositions $(\tau_1, \ldots, \tau_M)$ such that the permutation $\sigma_M = \tau_M \circ \ldots \circ \tau_1$ is $\delta$-close to the reverse permutation $\mathrm{rev} = (N \, \ldots \, 2 \, 1)$ in the Wasserstein distance on the space $\MM([0,1]^2)$ of Borel probability measures on $[0,1]^2$ (see Section \ref{sec:stationarity} for the definition). For fixed $\kappa \in (0,1)$ we define a \emph{random} $\delta$-\emph{relaxed sorting network} on $N$ elements by choosing $M$ from a Poisson distribution with mean $\lfloor \frac{1}{2} N^{1 + \kappa}(N-1) \rfloor$ and then sampling a $\delta$-relaxed sorting network of length $M$ uniformly at random.

Our first result is that the analog of the sorting network conjecture holds for relaxed sorting networks, that is, in a random relaxed sorting network the trajectory of a random particle is with high probability close in distribution to the sine curve process. Precisely, we have the following

\begin{theorem}\label{th:lln-for-relaxed-networks}
Fix $\kappa \in (0,1)$ and let $\pi^{N}_{\delta}$ denote the empirical distribution of the permutation process (as defined in \eqref{eq:empirical-eta}) associated to a random $\delta$-relaxed sorting network on $N$ elements. Let $\pi_{\mathcal{A}}$ denote the distribution of the sine curve process. Given any $\varepsilon > 0$ we have for all sufficiently small $\delta > 0$
\[
\lim\limits_{N \to \infty} \Pp^N \left( \pi^{N}_{\delta} \in B(\pi_{\mathcal{A}}, \varepsilon) \right) = 1,
\]
where $B(\pi_{\mathcal{A}}, \varepsilon)$ is the $\varepsilon$-ball in the Wasserstein distance on $\PP$.
\end{theorem}

Here for consistency of notation we assume that the sine curve process is rescaled so that it is supported on $[0,1]$ rather than $[-1,1]$.

The second result is more combinatorial and concerns the problem of enumerating sorting networks. A remarkable formula due to Stanley (\cite{stanley}) says that the number of all sorting networks on $N$ elements is equal to
\[
\frac{\binom{N}{2}!}{1^{N-1} 3^{N-2} \ldots (2N-3)^{1}},
\]
which is asymptotic to $\exp\left\{\frac{N^2}{2}\log N + (\frac{1}{4} - \log 2)N^2 + O(N \log N) \right\}$.

For relaxed sorting networks we have the following asymptotic estimate

\begin{theorem}\label{th:asymptotics-relaxed}
For any $\kappa \in (0,1)$ let $\mathcal{S}^{N}_{\kappa, \delta}$ be the number of $\delta$-relaxed sorting networks on $N$ elements of length $M = \lfloor\frac{1}{2} N^{1+\kappa}(N-1)\rfloor$. We have
\[
\mathcal{S}^{N}_{\kappa,\delta} = \exp \left\{ \frac{1}{2} N^{1 + \kappa} (N-1)\log (N-1) - \left(\frac{\pi^2}{6} + \varepsilon^{N}_{\delta}\right)N^{2 - \kappa}\right\},
\]
where $\varepsilon^{N}_{\delta}$ satisfies $\lim\limits_{\delta \to 0} \lim\limits_{N \to \infty} \varepsilon^{N}_{\delta} = 0$.
\end{theorem}

The asymptotics is analogous to that of Stanley's formula -- the first term in the exponent corresponds simply to the number of all paths of required length, and, crucially, the factor $\frac{\pi^2}{6}$ corresponds to the energy of the sine curve process.

The proofs of Theorem \ref{th:lln-for-relaxed-networks} and Theorem \ref{th:asymptotics-relaxed} are given in Section \ref{sec:asymptotics}. It would be an interesting problem to obtain analogous results for relaxed sorting networks reaching \emph{exactly} the reverse permutation, not only being $\delta$-close in the permuton topology. This case is not covered by the results of this paper, since the set of permuton processes reaching exactly the reverse permuton is not open, hence the lower bound of Theorem \ref{th:theorem-main-lower} does not apply.

\paragraph{Acknowledgments.} We would like to thank Maxim Arnold, Duncan Dauvergne, Boris Khesin and Mustazee Rahman for interesting discussions. MK was supported by the National Science Centre, Poland, grant no. 2019/32/C/ST1/00525. BV was supported by the Canada Research Chair program and the NSERC Discovery Accelerator grant.
\end{section}

\begin{section}{Preliminaries}\label{sec:preliminaries}
\begin{subsection}{Permutons and stochastic processes}\label{sec:stationarity}

\paragraph{Permutons.}

Consider the space $\MM([0,1]^2)$ of all Borel probability measures on the unit square $[0,1]^2$, endowed with the weak topology. A \emph{permuton} is a probability measure $\mu \in \MM([0,1]^2)$ with uniform marginals. In other words, $\mu$ is the joint distribution of a pair of random variables $(X,Y)$, with $X$, $Y$ taking values in $[0,1]$ and having marginal distribution $X, Y \sim \mathcal{U}[0,1]$. We will sometimes call the pair $(X,Y)$ itself a permuton if there is no risk of ambiguity. A few simple examples of permutons are the \emph{identity permuton} $(X,X)$, the \emph{uniform permuton} (the distribution of two independent copies of $X$, which is the uniform measure on the square) or the \emph{reverse permuton} $(X, 1 - X)$.

Permutons can be thought of as continuous limits of permutations in the following sense. Let $\Ss_{N}$ be the symmetric group on $N$ elements and
let $\sigma \in \Ss_{N}$. We associate to $\sigma$ its \emph{empirical measure}
\begin{equation}\label{eq:permutation-empirical}
\mu_{\sigma} = \frac{1}{N} \sum\limits_{i=1}^{N} \delta_{\left(\frac{i}{N}, \, \frac{\sigma(i)}{N}\right)},
\end{equation}
which is an element of $\MM([0,1]^2)$. By a slight abuse of terminology we will sometimes identify $\sigma$ with $\mu_\sigma$. Since every such measure has uniform marginals on $\left\{\frac{1}{N}, \frac{2}{N}, \ldots, 1  \right\}$, it is not difficult to see that if a sequence of empirical measures converges weakly, the limiting measure will be a permuton. Conversely, every permuton can be realized as a limit of finite permutations, in the sense of weak convergence of empirical measures (see \cite{limits}). We will consider $\MM([0,1]^2)$ endowed with the Wasserstein distance corresponding to the Euclidean metric on $[0,1]^2$, under which the distance of measures $\mu$ and $\nu$ is given by
\[
d_{\mathcal{W}}(\mu, \nu) = \inf\limits_{\left\{(X,Y), (X',Y')\right\}} \E \left[ \sqrt{(X - X')^2 + (Y-Y')^2} \right],
\]
where the infimum is over all couplings of $(X,Y)$ and $(X',Y')$ such that $(X,Y) \sim \mu$, $(X',Y') \sim \nu$.
\paragraph{The path space $\DD$ and stochastic processes.}

A natural setting for analyzing trajectories of particles in random permutation sequences is to consider $\DD = \DD([0,T], [0,1])$, the space of all c\`adl\`ag paths from $[0,T]$ to $[0,1]$. We endow it with the standard Skorokhod topology, metrized by a metric $\rho$ under which $\DD$ is separable and complete. By $\MM(\DD)$ we will denote the space of all Borel probability measures on $\DD$, endowed with the weak topology. It will be convenient to metrize $\MM(\DD)$ by the Wasserstein distance, under which the distance between measures $\mu$ and $\nu$ is given by
\[
d_{\mathcal{W}}(\mu, \nu) = \inf\limits_{(X,Y)} \E \left[ \rho(X,Y) \right],
\]
where the infimum is over all couplings $(X,Y)$ such that $X \sim \mu$, $Y \sim \nu$. We will also make use of the Wasserstein distance associated to the supremum norm, given by
\[
d_{\mathcal{W}}^{sup}(\mu, \nu) = \inf\limits_{(X,Y)} \E \left[ \norm{X - Y}_{sup} \right],
\]
where $\norm{\cdot}_{sup}$ is the supremum norm on $\DD$ and again the infimum is over all couplings $(X,Y)$ as above.

Given two times $0 \leq s \leq t \leq T$ and a stochastic process $X = (X_{t}, 0 \leq t \leq T)$ with distribution $\mu \in \MM(\DD)$, by $\mu_{s, t} \in \MM([0,1]^2)$ we will denote the distribution of the marginal $(X_{s}, X_{t})$. Note that the projection $\mu \mapsto \mu_{s, t}$ is continuous as a map from $\MM(\DD)$ to $\MM([0,1]^2)$ as long as paths $X \sim \mu$ sampled from $\mu$ have almost surely no jumps at times $s$ and $t$. We will sometimes implicitly identify the stochastic process with its distribution when there is no risk of misunderstanding.

\paragraph{Permutation processes and permuton processes.}

Consider a permutation-valued path $\eta^{N} = (\eta^{N}_{t}, 0 \leq t \leq T)$, with $\eta^{N}_{t}$ taking values in the symmetric group $\mathcal{S}_{N}$. We will always assume that $\eta^N$ is c\`{a}dl\`{a}g as a map from $[0,T]$ to $\Ss_N$. Let $\eta^{N}(i) = \left( \eta^{N}_{t}(i), 0 \leq t \leq T \right)$ be the trajectory of $i$ under $\eta^{N}$ and let $X^{\eta^{N}}(i) = \frac{1}{N}\eta^{N}(i)$ be the rescaled trajectory. We define the empirical measure
\begin{equation}\label{eq:empirical-eta}
\mu^{\eta^{N}} = \frac{1}{N} \sum\limits_{i=1}^{N} \delta_{X^{\eta^{N}}(i)},
\end{equation}
where $\delta_{X^{\eta^{N}}(i)}$ is the delta measure concentrated on the trajectory $X^{\eta^{N}}(i)$.

The associated {\it permutation process} $X^{\eta^{N}} = (X^{\eta^{N}}_{t}, 0 \leq t \leq T)$ is obtained by choosing $i = 1, \ldots, N$ uniformly at random and following the path $X^{\eta^{N}}(i)$. In other words, $X^{\eta^{N}}$ is a random path with values in $[0,1]$ whose distribution is $\mu^{\eta^{N}} \in \MM(\DD)$. If $\eta^N$ is fixed, the only randomness here comes from the random choice of the particle $i$. Note that at each time $t$ the marginal distribution of $X_{t}^{\eta^{N}}$ is uniform on $\left\{ \frac{1}{N}, \frac{2}{N}, \ldots, 1 \right\}$.

A \emph{permuton process} is a stochastic process $X = (X_{t}, 0 \leq t \leq T)$ taking values in $[0,1]$, with continuous sample paths and such that for every $t \in [0,T]$ the marginal $X_{t}$ is uniformly distributed on $[0,1]$. The name is justified by observing that if $\pi$ is the distribution of $X$, then for any fixed $s,t \in [0,T]$ the joint distribution $\pi_{s,t} \in \MM([0,1]^2)$ of $(X_{s}, X_{t})$ defines a permuton. As explained in the next subsection, permuton processes arise naturally as limits of permutation processes defined above.

Since every permutation process has marginals uniform on $\left\{ \frac{1}{N}, \frac{2}{N}, \ldots, 1 \right\}$, we will call it an \emph{approximate permuton process}. By $\PP$ we will denote the space of all permuton processes and approximate permuton processes, treated as a subspace of $\MM(\DD)$ (with the same topology and the metric $d_{\mathcal{W}}$).

\paragraph{Random permutation and permuton processes.}

A \emph{random permuton process} is a permuton process chosen from some probability distribution on the space of all permuton processes, i.e., a random variable $X$, defined for a probability space $\Omega$, such that $X(\omega)$ is a permuton process for $\omega \in \Omega$. By identifying the random variable with its distribution we can also think of a random permuton process as a random element of $\MM(\PP)$. In this setting, with weak topology on $\MM(\PP)$, one can consider convergence in distribution of random permuton processes $X_{n}$ to a (possibly also random) permuton process $X$.

One can prove (see \cite{mustazee}) that if a sequence of random permutation processes $X^{\eta^N}$ converges in distribution, then the limit is a permuton process (in general also random). Of particular interest will be sequences of random permutation-valued paths $\eta^{N}$ (coming for example from the interchange process) such that the corresponding permutation processes $X^{\eta^N}$ converge in distribution to a \emph{deterministic} permuton process (for example the sine curve process described below).

For any random permuton process $X$ we define its associated \emph{random particle process} $\bar{X} = \E_{\omega} X(\omega)$, which is a process with a deterministic distribution, obtained by first sampling a permuton process $X(\omega)$ and then sampling a random path according to $X(\omega)$.

To elucidate the difference between random and deterministic permuton processes, consider a random permuton process $X$ and its associated random particle process $\bar{X}$. If we sample an outcome $X(\omega)$ and then a path from $X(\omega)$, then obviously the distribution of paths will be the same as for $\bar{X}$. However, consider now sampling an outcome $X(\omega)$ and then sampling independently two paths from $X(\omega)$. The distribution of a pair of paths obtained in this way will not in general be the same as the distribution of two independent copies sampled from $\bar{X}$, since the paths might be correlated within the outcome $X(\omega)$. The following general lemma will be useful later for showing that limits of certain random permutation processes are in fact deterministic (\cite[Lemma 3]{mustazee2}):

\begin{lemma}\label{lm:deterministic}
Let  $K$ be a compact metric space and let $\mu$ be a random probability measure on $K$, i.e., a random variable with values in $\MM(K)$. Let $X$ and $Y$ be two independent samples from an outcome of $\mu$ and let $Z$ be a sample from an outcome of an independent copy of $\mu$. If $(X,Y)$, as a $K^2$-valued random variable, has the same distribution as $(X,Z)$, then $\mu$ is in fact deterministic, i.e., there exists $\nu \in \MM(K)$ such that $\mu = \nu$ almost surely.
\end{lemma}

\paragraph{Energy.}

Here we introduce several related notions of energy for paths, permutations, permutons and permuton processes.

Given a path $\gamma : [0,T] \to [0,1]$ and a finite partition $\Pi = \{ 0 = t_{0} < t_{1} < \ldots < t_{k} = T \}$ we define the \emph{energy of $\gamma$ with respect to $\Pi$} as
\begin{equation}\label{eq:energy-path-def}
\EE^{\Pi}(\gamma) = \frac{1}{2} \sum\limits_{i=1}^{k} \frac{| \gamma(t_{i}) - \gamma(t_{i-1}) |^2}{t_{i} - t_{i-1}},
\end{equation}
and the \emph{energy of $\gamma$} as
\begin{equation}\label{eq:energy-path-full}
\EE(\gamma) = \sup_{\Pi} \EE^{\Pi}(\gamma),
\end{equation}
where the supremum is over all finite partitions $\Pi = \{ 0 = t_{0} < t_{1} < \ldots < t_{k} = T \}$. For a path which is not absolutely continuous the supremum is equal to $+\infty$. If a path $\gamma$ is differentiable, its energy is equal to
\[
\frac{1}{2} \int\limits_{0}^{T} \dot{\gamma}(s)^2 \, ds.
\]
For a permutation $\sigma \in \mathcal{S}_N$ we define its energy as
\begin{equation}\label{eq:permutation-energy}
I(\sigma) = \frac{1}{2} \left( \frac{1}{N} \sum\limits_{i=1}^{N} \left( \frac{\sigma(i) - i}{N} \right)^2 \right).
\end{equation}

Likewise, for a permuton $\mu \in \MM([0,1]^2)$ its energy is defined by
\begin{equation}\label{eq:permuton-energy-def}
I(\mu) = \frac{1}{2} \E |X-Y|^2,
\end{equation}
where the pair $(X, Y)$ has distribution $\mu$. If $\mu = \mu_{\sigma}$ is the empirical measure of a permutation $\sigma \in \mathcal{S}_N$, defined by \eqref{eq:permutation-empirical}, then we have $I(\mu_{\sigma}) = I(\sigma)$. Note also that $I = I(\mu)$ is a continuous function of $\mu$ in the weak topology on $\MM([0,1]^2)$.

Finally, we define the energy of a permuton process $\pi$ as
\begin{equation}\label{eq:process-energy}
I(\pi) = \E_{\gamma \sim \pi} \EE(\gamma),
\end{equation}
where the expectation is over paths $\gamma$ sampled from $\pi$. We can extend this definition to any process $\pi \in \MM(\DD)$ by adopting the convention that $I(\pi) = + \infty$ if paths sampled from $\pi$ are not absolutely continuous almost surely. The function $I$ will turn out to correspond to the rate function in large deviation bounds for random permuton process. It can be checked that $I$ is lower semicontinuous (in the weak topology on $\PP$) and its level sets $\{ \pi \in \PP : I(\pi) \leq C\}$ are compact.

We will also use the notation
\begin{equation}\label{eq:def-energy-fin-dim}
I^{\Pi}(\pi) = \E_{\gamma \sim \pi} \EE^{\Pi}(\gamma)
\end{equation}
to denote the approximation of energy of $\pi$ associated to the finite partition $\Pi$. The following lemma will be useful in characterizing the large deviation rate function in terms of these approximations
\begin{lemma}\label{lm:approximate-energy}
For any process $\pi \in \MM(\DD)$ we have
\[
I(\pi) = \sup\limits_{\Pi} I^{\Pi}(\pi),
\]
where the supremum is taken over all finite partitions $\Pi = \{ 0 = t_{0} < t_{1} < \ldots < t_{k} = T \}$.
\end{lemma}

\begin{proof}
Let $\Pi_n = \left\{0 < \frac{1}{2^n} < \frac{2}{2^n} < \ldots < 1\right\}$, $n=0,1,2,\ldots$, be the sequence of dyadic partitions of $[0,1]$. It is elementary to show that if a path $\gamma$ is continuous, then $\EE(\gamma) = \lim\limits_{n \to \infty} \EE^{\Pi_n}(\gamma)$.  
Note that if $\Pi'$ is a refinement of $\Pi$, then we have $\EE^{\Pi}(\gamma) \leq \EE^{\Pi'}(\gamma)$, thus $\EE^{\Pi_n}(\gamma) \to \EE(\gamma)$ monotonically as $n \to \infty$. Now we apply the monotone convergence theorem to get the same same convergence for the expectations $\E_{\gamma \sim \pi}\EE^{\Pi_n}(\gamma)$.
\end{proof}

\paragraph{The interchange process.}

The \emph{interchange process} on the interval $\{1, \ldots, N\}$ is a Markov process in continuous time defined in the following way. Consider particles labelled from $1$ to $N$ on a line with $N$ vertices. Each edge has an independent exponential clock that rings at rate $1$. Whenever a clock rings, the particles at the endpoints of the corresponding edge swap places. By comparing the initial position of each particle with its position after time $t$ we obtain a random permutation of $\{ 1, \ldots,N\}$.

Formally, we define the state space of the process as consisting of permutations $\eta \in \mathcal{S}_N$, with the notation $ \eta = (x_1, \ldots, x_N)$ indicating that the particle with label $i$ is at the position $x_{i}$, or in other words, $x_i = \eta(i)$. The dynamics is given by the generator
\begin{equation}\label{eq:unbiased-generator}
(\Ll f)(\eta) = \frac{1}{2} N^{\alpha} \sum\limits_{x=1}^{N-1} \left( f(\eta^{x,x+1}) - f(\eta) \right),
\end{equation}
where $\eta^{x, x+1}$ is the configuration $\eta$ with particles at locations $x$ and $x+1$ swapped and $\alpha \in (1,2)$ is a fixed parameter (introduced so that we will be able to consider the limit $N \to \infty$). Since we will also be considering variants of this process with modified rates, we will often refer to the process with generator $\Ll$ as the \emph{unbiased interchange process}.

The interchange process defines a probability distribution on permutation-valued paths $\eta^N = (\eta^N_t, 0 \leq t \leq T)$ for any $T \geq 0$. Consider now the permutation process $X^{\eta^N}$ associated to $\eta^N$, that is, sample $\eta^N$ according to the interchange process, pick a particle uniformly at random and follow its trajectory in $\eta^N$. The distribution $\mu^{\eta^N}$ of $X^{\eta^N}$, defined by \eqref{eq:empirical-eta}, is then a random element of $\MM(\DD)$.

The position of a random particle in the interchange process will be distributed as the stationary simple random walk (in continuous time) on the line $\{1, \ldots, N\}$. If we look at timescales much shorter than $N^2$, typically each particle will have distance $o(N)$ from its origin, so the permutation obtained at time $t$ such that $tN^{\alpha} \ll N^{2}$ will be close (in the sense of permutons) to the identity permutation. As mentioned in the introduction, we will be interested in large deviation bounds for rare events such as seeing a nontrivial permutation after a short time.

\end{subsection}

\begin{subsection}{Euler equations and generalized incompressible flows}\label{sec:incompressible-flows}

Let us now discuss the connection to fluid dynamics and incompressible flows (the discussion here follows \cite{ambrosio-figalli} and \cite{bernot-figalli-santambrogio}). The Euler equations describe the motion of an incompressible fluid in a domain $D \subseteq \R^d$ in terms of its \emph{velocity field} $u(t,x)$, which is assumed to be divergence-free. The evolution of $u$ is given in terms of the \emph{pressure field} $p$
\[
\begin{cases}
\partial_t u + (u \cdot \nabla)u = - \nabla p \quad & \mbox{in $[0,T] \times D$}, \\
\mathrm{div} \, u = 0 \quad & \mbox{in $[0,T] \times D$},\\
u \cdot n = 0 \quad & \mbox{on $[0,T] \times \partial D$},
\end{cases}
\]
where the second equation encodes the incompressiblity constraint and the third equation means that $u$ is parallel to the boundary $\partial D$.

Assuming $u$ is smooth, the trajectory $g(t,x)$ of a fluid particle initially at position $x$ is obtained by solving the equation
\[
\begin{cases}
\dot{g}(t,x) = u(t, g(t,x)), \\
g(0,x) = x.\\
\end{cases}
\]
Since $u$ is assumed to be divergence-free, the flow map $\Phi^{t}_{g} : D \to D$ given by $\Phi^{t}_{g}(x) = g(t,x)$ is a measure-preserving diffeomorphism of $D$ for each $t \in [0,T]$. This means that $(\Phi^{t}_{g})_{\ast} \mu_D = \mu_D$, where from now on by $f_{\ast}$ we denote the pushforward map on measures, associated to $f$, and $\mu_D$ is the Lebesgue measure inside $D$. Denoting by $\sdiff(D)$ the space of all measure-preserving diffeomorphisms of $D$, we can rewrite the Euler equations in terms of $g$
\begin{equation}\label{eq:euler-g}
\begin{cases}
\ddot{g}(t,x) = - \nabla p (t, g(t,x)) \quad & \mbox{in $[0,T] \times D$}, \\
g(0,x) = x & \mbox{in $D$},\\
g(t, \cdot) \in \sdiff(D) & \mbox{for each $t \in [0,T]$}.
\end{cases}
\end{equation}
Arnold proposed an interpretation according to which the equation above can be viewed as a geodesic equation on $\sdiff(D)$. Thus one can look for solutions to \eqref{eq:euler-g} by considering the variational problem
\begin{equation}\label{eq:arnold}
\mbox{minimize} \quad \frac{1}{2}\int\limits_{0}^{T} \int\limits_{D} |\dot{g}(t,x)|^2 \, d\mu_{D}(x) \, dt
\end{equation}
among all paths $g(t, \cdot) : [0,T] \to \sdiff(D)$ such that $g(0, \cdot) = f$, $g(T, \cdot) = h$ for some prescribed $f, h \in \sdiff(D)$ (by right invariance without loss of generality $f$ can be assumed to be the identity). The pressure $p$ then arises as a Lagrange multiplier coming from the incompressibility constraint.

Shnirelman proved (\cite{shnirelman}) that in dimensions $d \geq 3$ the infimum in this minimization problem is not attained in general and in dimension $d=2$ there exist diffeomorphisms $h = g(T, \cdot)$ which cannot be connected to the identity map by a path with finite action. This motivated Brenier (\cite{brenier}) to consider the following relaxation of this problem. With $C(D)$ denoting the space of continuous paths from $[0,T]$ to $D$ and $\MM(C(D))$ the set of probability measures on $C(D)$, the variational problem is
\begin{equation}\label{eq:minimize-brenier}
\mbox{minimize} \quad \int\limits_{C(D)} \left(\frac{1}{2}\int\limits_{0}^{T} |\dot{\gamma}(t)|^2 \, dt \right)  d\pi(\gamma)
\end{equation}
over all $\pi \in \MM(C(D))$ satisfying the constraints
\begin{equation}\label{eq:constraints}
\begin{cases}
 \pi_{0,T} = (id, h)_{\ast} \mu_D, \\
 \pi_t = \mu_D \quad \mbox{for each $t \in [0,T]$},
\end{cases}
\end{equation}
where $\pi_{0,T}$, $\pi_t$ denote the marginals of $\pi$ at times respectively $0,T$ and at time $t$.

Following Brenier, a probability measure $\pi \in \MM(C(D))$ satisfying constraints \eqref{eq:constraints} is called a \emph{generalized incompressible flow} between the identity $id$ and $h$. To see that indeed \eqref{eq:minimize-brenier} is a relaxation of \eqref{eq:arnold}, note that any sufficiently regular path $g(t, \cdot) : [0,T] \to \sdiff(D)$, for example corresponding to a solution of \eqref{eq:euler-g}, induces a generalized incompressible flow given by $\pi = (\Phi_g)_{\ast} \mu_D$, where as before $\Phi_g(x) = g(\cdot, x)$. As evidenced by the sine curve process mentioned in the introduction, the converse is false -- trajectories of particles sampled from a generalized flow can cross each other or split at a later time when starting from the same position, which is not possible for classical, smooth flows. We refer the reader to \cite{brenier-physica} for an interesting discussion of physical relevance of this phenomenon.

The problem admits a natural further relaxation in which the target map is ``non-deterministic'', in the sense that we have $\pi_{0,T} = \mu$ with $\mu$ being an arbitrary probability measure supported on $D \times D$ and having uniform marginals on each coordinate, not necessarily of the form $\mu = (id, h)_{\ast}\mu_{D}$ for some map $h$. From now on whenever we refer to problem \eqref{eq:minimize-brenier} or generalized incompressible flows we will be always considering this more general variant.

The connection between the generalized problem \eqref{eq:minimize-brenier} and the original Euler equations \eqref{eq:euler-g} is provided by a theorem due to Ambrosio and Figalli (\cite{ambrosio-figalli}), with earlier weaker results by Brenier (\cite{brenier-distribution}). Roughly speaking, they showed that given a measure $\mu$ with uniform marginals there exists a pressure function $p(t,x)$ such that the following holds -- one can replace the problem of minimizing the functional \eqref{eq:minimize-brenier} over incompressible flows satisfying $\pi_{0,T} = \mu$ by an easier problem in which the incompressibility constraint is dropped, provided one adds to the functional a Lagrange multiplier given by $p$. We refer the reader to \cite[Section 6]{ambrosio-figalli} for a precise formulation and further results on regularity of $p$.

In particular, if $\pi$ is optimal for \eqref{eq:minimize-brenier} and the corresponding pressure $p$ is smooth enough, their result implies that almost every path $\gamma$ sampled from $\pi$ minimizes the functional
\begin{equation}\label{eq:functional-with-p}
\gamma \mapsto \int\limits_{0}^{T} \left( \frac{1}{2} |\dot{\gamma}(t)|^2 - p(t, \gamma(t)) \right) dt.
\end{equation}
In that case the equation $\ddot{g}(t,x) = - \nabla p(t, g(t,x))$ from \eqref{eq:euler-g} is nothing but the Euler-Lagrange equation for extremal points of the functional \eqref{eq:functional-with-p}. We can therefore, at least under some regularity assumptions on $p$, think of generalized incompressible flows as solutions to \eqref{eq:euler-g} in which instead of having a diffeomorphism we assume random initial conditions for each particle.

From now on let us restrict the discussion to $D = [0,1]$, which will be most directly relevant to the results of this paper. In this case the original problem \eqref{eq:arnold} is somewhat uninteresting, since the only measure-preserving diffeomorphisms of $[0,1]$ are $f(x) = x$ and $f(x) = 1 - x$. However, the relaxed problem \eqref{eq:minimize-brenier} is non-trivial and indeed for the target map $h(x) = 1 - x$ and $T = 1$ the unique optimal solution is given by the sine curve process.

In this setting, the reader may recognize that generalized incompressible flows are in fact the same objects as permuton processes. The term \emph{measure-preserving plans} is used in \cite{ambrosio-figalli} for what we call permutons. The functional minimized in \eqref{eq:minimize-brenier} is the energy $I(\pi)$ of a permuton process, defined in \eqref{eq:process-energy}. In this language the optimization problem we are interested in can be rephrased as follows:
\begin{equation}\label{eq:minimize-permuton}
\mbox{find} \inf\limits_{\substack{\pi \in \PP \\ \pi_{0,T} = \mu}} I(\pi),
\end{equation}
where the infimum is over all permuton processes $\pi \in \PP$ satisfying $\pi_{0,T} = \mu$ for a given permuton $\mu \in \MM([0,1]^2)$.

\paragraph{Generalized solutions to Euler equations.}

We will say that a permuton process $\pi$ is a \emph{generalized solution to Euler equations} if there exists a function $p : [0,T] \times [0,1] \to \R$, differentiable in the second variable, such that almost every path $x : [0,T] \to [0,1]$ sampled from $\pi$ satisfies the equation
\begin{equation}\label{eq:gen-euler}
\begin{cases}
		x'(t) = v(t) \\
		v'(t) = - \partial_x p(t, x(t)) \\
\end{cases}
\end{equation}
for $t \in [0,T]$. This is of course equivalent to $x''(t) = - \partial_x p (t, x(t))$.

By the remarks above, if $\pi$ minimizes the energy in \eqref{eq:minimize-permuton} and the associated pressure $p$ is smooth enough, then $\pi$ is always a generalized solution to Euler equations. However, this is only a necessary condition -- for a discussion of corresponding sufficient conditions see \cite{bernot-figalli-santambrogio}.
\end{subsection}

\begin{subsection}{Proof outline and structure of the paper}\label{sec:main-results}

Let us now give a brief outline of the proof strategy for Theorem \ref{th:theorem-main-lower} and Theorem \ref{th:theorem-main-upper}. For the lower bound, given a process $X$ we construct a perturbation of the interchange process (defined by introducing asymmetric jump rates based on \eqref{eq:gen-euler}) for which a law of large numbers holds, namely, the distribution of the path of a random particle converges to a deterministic limit (which is the distribution of $X$). The large deviation principle is then proved by estimating the Radon-Nikodym derivative between the biased process and the original one.

The key property which makes this construction possible is that the process $X$ satisfies a second order ODE given by \eqref{eq:gen-euler}, so its trajectories are fully specified by the particle's position and velocity (the latter chosen initially from a mean zero distribution). The biased process is then constructed by endowing each particle with an additional parameter keeping track of its velocity, but we perform an additional change variables, working instead of velocity with a variable we call \emph{color}. The advantage of this is that the uniform distribution of colors is stationary when the jump rates are properly chosen, which will greatly facilitate the analysis. An additional technical difficulty arises if the velocity distribution of $X$ is time-dependent or not regular enough near the boundary, in which case we first approximate $X$ by a process with a sufficiently regular and piecewise time-homogeneous velocity distribution.

To prove the law of large numbers we need to show that in the biased interchange process particles' trajectories behave approximately like independent samples from $X$. This requires proving that their velocities remain uncorrelated when averaged over time and is accomplished by means of a local mixing result called the \emph{one block estimate}. It is here that we rely on stationarity of the uniform distribution of colors in the biased process and the fact that $X$ has velocity zero on average.

The strategy for proving the upper bound is somewhat simpler. We consider a family of exponential martingales similar to the one employed in analyzing independent random walks and use the one block estimate to show that the particles' velocities are typically nonnegatively correlated. This enables us to prove the large deviation upper bound for compact sets and the extension to closed sets is done by proving exponential tightness. 

\paragraph{Structure of the paper.}

The rest of the paper is structured as follows. In Section \ref{sec:ode-part} we introduce the change of variables needed to define the process with colors and prove the approximation result for $X$ mentioned above (Proposition \ref{prop:approximation-epsilon-delta}). In Section \ref{sec:interchange} we define the biased interchange process and derive the conditions on its rates which guarantee stationarity. Section \ref{sec:lln} contains the proof of the law of large numbers for the biased interchange process (Theorem \ref{th:lln}). In Section \ref{sec:one-block} we prove two variants of the one block estimate -- one needed for the large deviation upper bound (Lemma \ref{lm:one-block-superexponential-probability}) and a more involved one needed for the proof of the law of large numbers (Lemma \ref{lm:one-block-superexponential-biased}). In Section \ref{sec:lower-bound} these pieces are then used to prove the large deviation lower bound (Theorem \ref{thm:lower-bound-for-minimizers}). Section \ref{sec:upper-bound} is devoted to the proof of the large deviation upper bound (Theorem \ref{th:upper-bound}) and is independent of the previous sections (apart from the use of Lemma \ref{lm:one-block-superexponential-probability}). Finally, in Section \ref{sec:asymptotics} we prove Theorem \ref{th:lln-for-relaxed-networks} and Theorem \ref{th:asymptotics-relaxed} on relaxed sorting networks.
\end{subsection}
\end{section}

\begin{section}{ODEs and generalized solutions to Euler equations}\label{sec:ode-part}

\paragraph{Regularity assumptions and properties of generalized solutions.}

Suppose $\pi$ is a generalized solution to Euler equations \eqref{eq:gen-euler} and let $X$ be a process with distribution $\pi$. For the proof of the large deviation lower bound we will need to impose additional regularity assumptions on $\pi$. For $t \in [0,T]$ let $\mu_t$ denote the joint distribution of $(x(t), x'(t))$ when $x$ is sampled according to $\pi$. In particular, $\mu_0$ is the joint distribution of the initial conditions of the ODE \eqref{eq:gen-euler}. If $\Phi^{t,s}(x,v)$ denotes the solution $x(s)$ of \eqref{eq:gen-euler} satisfying $(x(t), v(t)) = (x, v)$, then $\mu_t = \Phi^{0,t}_{\ast}\mu_0$.

We will assume that each $\mu_t$ has a density $\rho_t(x,v)$ with respect to the Lebesgue measure on $[0,1] \times \R$. For $x \in [0,1]$ and $t \in [0,T]$ let $\mu_{t,x}$ denote the conditional distribution of $v$, given $x$, at time $t$. In addition we assume that for $x=0$ or $1$ the distribution $\mu_{t,x}$ is a delta mass at $0$, as otherwise the process $X$ cannot stay confined to $[0,1]$ and have mean velocity zero everywhere (see the discussion of incompressiblity below).

Let $F_{t,x}$ denote the cumulative distribution function of $\mu_{t,x}$ and let $V_{t}(x, \cdot) : [0,1] \to \R$ be the quantile function of $\mu_{t,x}$, defined for $x \in [0,1]$ and $\phi \in (0,1]$ by
\[
V_{t}(x,\phi) = \inf\left\{ v \in \R \, | \, F_{t,x}(v) \geq \phi \right\}
\]
and $V_t(x, 0) = \inf\left\{ v \in \R \, | \, F_{t,x}(v) > 0 \right\}.$ In particular for $x=0,1$ we have $V_{t}(x,\phi) = 0$.

\begin{assumption}\label{as:main-assumptions}
Throughout the paper, we will assume that for a generalized solution to Euler equations $\pi$ the following properties are satisifed
\begin{enumerate}[(1)]
\item the pressure function $(t,x) \mapsto p(t,x)$ in \eqref{eq:gen-euler} is measurable in $t$ and differentiable in $x$, with the derivative $\partial_x p(t,x)$ Lipschitz continuous in $x$ (with the Lipschitz constant uniform in $t$)
\item there exists a compact set $K \subseteq [0,1] \times \R$ such that for each $t \in [0,T]$ the density $\rho_t$ is supported in $K$
\item for $t \in [0,T], x \in [0,1]$ the support of $\mu_{t,x}$ is a connected interval in $\R$
\item the density $\rho_t$ is continuously differentiable in $t$, $x$ and $v$ for each $t \in [0,T]$ and $x,v$ in the interior of the support of $\rho_t$
\end{enumerate}
\end{assumption}

Let us comment on the relevance of these assumptions. Assumption (1) will guarantee uniqueness of solutions to \eqref{eq:gen-euler}. Assumption (2) implies that the velocity of a particle moving along a path sampled from $\pi$ stays uniformly bounded in time. Assumption (3) implies that for any $x \in (0,1)$ and $\phi \in [0,1]$ we have $F_{t,x}(V_{t}(x, \phi)) = \phi$, i.e., $V_t(x, \cdot)$ is the inverse function of $F_{t,x}$. Assumptions (3) and (4) imply that $V_{t}(x,\phi)$ is a continuous function of $t$, $x$, $\phi$ and it is continuously differentiable in all variables for $x \in (0,1)$.

Note that for $V_t(x,\phi)$ to be differentiable at $\phi = 0,1$, the distribution function $F_{t,x}$ necessarily has to be non-differentiable at corresponding $v$ such that $F_{t,x}(v) = \phi$. This is why we can require the density $\rho_t$ to be smooth only in the interior of its support and not at the boundary.

From now on we assume that $\pi$ is a fixed generalized solution to Euler equations, satisfying Assumptions \eqref{as:main-assumptions}. Almost every path $x : [0,T] \to [0,1]$ sampled from $\pi$ satisfies the ODE
\begin{equation}\label{eq:true-process}
\begin{cases}
		x'(t) = v(t) \\
		v'(t) = - \partial_x p(t, x(t)). \\
\end{cases}
\end{equation}

Note that since $\pi$ is a permuton process, each measure $\mu_t$ satisfies the \emph{incompressibility} condition, meaning that its projection onto the first coordinate is equal to the uniform measure on $[0,1]$. This is equivalent to the property that for any test function $f : [0,1] \to \R$ we have
\[
\int\limits f(x) \, d\mu_t(x,v) = \int\limits_{0}^{1} f(x) \, dx.
\]
An important consequence of the incompressibility assumption is that under $\mu_t$ the velocity has mean zero at each $x$, that is, we have the following
\begin{lemma}\label{lm:velocity-mean-zero}
For any $t \in [0,T]$ and $x \in [0,1]$ we have
\[
\int v \, d\mu_{t,x}(v) = 0.
\]
 \end{lemma}
\begin{proof}
Consider any test function $f : [0,1] \to \R$ and write
\[
\int\limits f(x) \, d\mu_{t+s}(x,v) = \int\limits f(x) \, d(\Phi^{t,t+s}_{\ast}\mu_t) (x,v) = \int\limits f(\Phi^{t,t+s}(x,v)) \, d\mu_{t}(x,v).
\]
By incompressibility the integral above is always equal to $\int\limits_{0}^{1} f(x) \, dx $, in particular does not depend on time. On the other hand its derivative with respect to $s$ is
\[
\frac{d}{ds} \int\limits f(x) \, d\mu_{t+s}(x,v) = \frac{d}{ds} \int\limits f(\Phi^{t,t+s}(x,v)) \, d\mu_{t}(x,v) = \int\limits f'(\Phi^{t,t+s}(x,v)) \frac{d\Phi^{t,t+s}}{ds}(x,v) \, d\mu_{t}(x,v).
\]
Since $\Phi^{t,t+s}(x,v)\vert_{s=0} = x$ and $\frac{d\Phi^{t,t+s}}{ds}(x,v)\vert_{s=0} = v$, by evaluating the derivative at $s = 0$ we arrive at $\int\limits f'(x) v \, d\mu_{t}(x,v) = 0$. Since $\int g(x,v) \, d\mu_t(x,v) = \int g(x,v) \, d\mu_{t,x}(v) dx$ for any measurable $g$ and $f$ was an arbitrary test function, the claim of the lemma holds for almost every $x$. Since we have assumed that $\mu_{t}$ has a continuous density, the claim in fact holds for all $x$, which ends the proof.
\end{proof}

We will also make use of an explicit evolution equation that the densities $\rho_t$ have to satisfy. This is the content of the following lemma.
\begin{lemma}\label{lm:evolution-equation}
For any $t \in [0,T]$ and $x, v$ in the interior of the support of $\rho_t$ we have
\[
\frac{\partial \rho_t}{\partial t} (x,v) = - v \frac{\partial \rho_t}{\partial x} (x,v) + \partial_x p(t, x) \frac{\partial \rho_t}{\partial v} (x,v).
\]
\end{lemma}
\begin{proof}
Let $f : [0,1] \times \R \to \R$ be any test function and consider the integral
\[
I_{t+s} = \int f(x,v) \, d\mu_{t+s}(x,v).
\]
On the one hand, its derivative with respect to $s$ is equal to
\begin{align*}
 \frac{d}{ds}I_{t+s} & = \frac{d}{ds} \int f(x,v) \, d\mu_{t+s}(x,v) =\frac{d}{ds} \int f\left(\Phi^{t,t+s}(x,v),\frac{d\Phi^{t,t+s}}{ds}(x,v)\right) \rho_t (x,v) \, dx \, dv = \\
& = \int  \bigg[ \frac{\partial f}{\partial x}\left(\Phi^{t,t+s}(x,v),\frac{d\Phi^{t,t+s}}{ds}(x,v)\right)\frac{d\Phi^{t,t+s}}{ds}(x,v)  + \\
& +  \frac{\partial f}{\partial v}\left(\Phi^{t,t+s}(x,v),\frac{d\Phi^{t,t+s}}{ds}(x,v)\right)\frac{d^2\Phi^{t,t+s}}{ds^2}(x,v) \bigg] \rho_t (x,v) \, dx \, dv.
\end{align*}
Since $\Phi^{t,t+s}(x,v)$ is a solution to \eqref{eq:true-process}, we have $\frac{d\Phi^{t,t+s}}{ds}(x,v)\big\vert_{s=0} = v$ and $\frac{d^2\Phi^{t,t+s}}{ds^2}(x,v)\big\vert_{s=0} = -\partial_x p(t, x)$, which gives us
\[
\frac{d}{ds}I_{t+s}\Big\vert_{s=0} = \int \left( \frac{\partial f}{\partial x}(x,v)v - \frac{\partial f}{\partial v}(x,v)\partial_x p(t, x) \right) \rho_t (x,v) \, dx \, dv.
\]
Performing integration by parts with respect to $x$ for the first term and with respect to $v$ for the second term gives (noting that $f$ has compact support so the boundary terms vanish)
\[
\frac{d}{ds}I_{t+s}\Big\vert_{s=0} = -\int f(x,v)v \frac{\partial \rho_t}{\partial x} (x,v) \, dx \, dv + \int f(x,v)\partial_x p(t, x) \frac{\partial \rho_t}{\partial v} (x,v) \, dx \, dv.
\]
On the other hand, we have
\[
\frac{d}{ds}I_{t+s} = \frac{d}{ds} \int f(x,v) \, d\mu_{t+s}(x,v) = \frac{d}{ds} \int f(x,v) \rho_{t+s} (x,v) \, dx \, dv = \int f(x,v) \frac{\partial \rho_{t+s}}{\partial s} (x,v) \, dx \, dv,
\]
so
\[
\frac{d}{ds}I_{t+s}\Big\vert_{s=0} = \int f(x,v) \frac{\partial \rho_{t}}{\partial t} (x,v) \, dx \, dv
\]
and thus
\[
\int f(x,v) \left(- v \frac{\partial \rho_t}{\partial x} (x,v) + \partial_x p(t, x) \frac{\partial \rho_t}{\partial v} (x,v) - \frac{\partial \rho_{t}}{\partial t} (x,v) \right) dx \, dv.
\]
Since the test function $f$ was arbitrary, the equation from the statement of the lemma must hold for every $t$, $x$, $v$ as assumed.
\end{proof}

\paragraph{The colored trajectory process.}

Let $X = (X_t, 0 \leq t \leq T)$ be the permuton process with distribution $\pi$. For the large deviation lower bound we will need to construct a suitable interacting particle system in which the behavior of a random particle mimics that of the permuton process $X$. A crucial ingredient will be a property analogous to Lemma \ref{lm:velocity-mean-zero}, i.e., having velocity distribution whose mean is locally zero. Instead of working with velocity $v$, whose distribution $\rho_t(x,v)$ at a given site $x$ may change in time, it will be more convenient to perform a change variables and use another variable $\phi$, which we call \emph{color}, whose distribution will be invariant in time.

Recall that under Assumptions \eqref{as:main-assumptions} the distribution function $F_{t,x}(\cdot)$ and the quantile function $V_{t}(x, \cdot)$ are related by
\begin{equation}\label{eq:cdf}
\begin{cases}
F_{t,x}(V_{t}(x, \phi)) = \phi \\
V_t(x, F_{t,x}(v)) = v
\end{cases}
\end{equation}
for any $t \in [0,T]$, $x \in (0,1)$, $\phi \in [0,1]$, $v \in \supp \, \mu_{t,x}$.

The reason for introducing the variable $\phi$ is the following elementary property -- if $\phi$ is sampled from the uniform distribution on $[0,1]$, then $V_t(x, \phi)$ is distributed according to $\mu_{t,x}$. Thus instead of working with $(x,v)$ variables in the ODE \eqref{eq:true-process}, where the distribution of $v$ evolves in time, we can set up an ODE for $x$ and $\phi$ such that the joint distribution of $(x, \phi)$ will be uniform on $[0,1]^2$ at each time. The velocity $v$ and its distribution can then be recovered via the equation $v = V_t(x, \phi)$.

Let $(x(t), v(t))$ be a solution to \eqref{eq:true-process} such that $x(t) \neq 0,1$ and let
\[
\phi(t) = F_{t, x(t)}(v(t)).
\]
Let us derive the ODE that $(x(t), \phi(t))$ satisifes. Since $(x(t), v(t))$ is a solution of \eqref{eq:true-process}, we have
\begin{align*}
\phi'(t) = & \frac{\partial F_{t, x(t)}}{\partial t}(v(t)) + \frac{\partial F_{t, x(t)}}{\partial x}(v(t))x'(t) + \frac{\partial F_{t, x(t)}}{\partial v}(v(t)) v'(t) = \\
 = &\frac{\partial F_{t, x(t)}}{\partial t}(v(t)) + \frac{\partial F_{t, x(t)}}{\partial x}(v(t))v(t) + \rho_t(x(t), v(t)) \left[ -\partial_x p (t, x(t)) \right].
\end{align*}
Lemma \ref{lm:evolution-equation} implies that
\[
\frac{\partial F_{t, x(t)}}{\partial t} (x(t), v(t)) = - \int\limits_{-\infty}^{v(t)} w \frac{\partial \rho_t}{\partial x} (x(t),w) + \left[ \partial_x p(t, x(t)) \right] \rho_t (x(t),v(t)),
\]
which gives
\[
\phi'(t) = \frac{\partial F_{t, x(t)}}{\partial x}(v(t))v(t) - \int\limits_{-\infty}^{v(t)} w \frac{\partial \rho_t}{\partial x} (x(t),w)
\]
and upon integrating by parts in the last integral we obtain
\begin{equation}\label{eq:phi-prime}
\phi'(t) = \int\limits_{-\infty}^{v(t)} \frac{\partial F_{t,x(t)}}{\partial x}(x(t),w) \, dw.
\end{equation}
Now, differentiating \eqref{eq:cdf} with respect to $x$ and $\phi$ gives
\[
\begin{cases}
\frac{\partial F_{t,x}}{\partial x}(V_t(x, \phi)) + \rho_t(x,\phi) \frac{\partial V_t}{\partial x}(x, \phi) = 0 \\
\rho_t(x, \phi) \frac{\partial V_t}{\partial \phi}(x, \phi) = 1.
\end{cases}
\]
Also by \eqref{eq:cdf} we have $v(t) = V_t(x(t), \phi(t))$, so a change of variables $w = V_t(x(t), \psi)$ in \eqref{eq:phi-prime} yields
\[
\phi'(t) = R_t(x(t), \phi(t)),
\]
where $R_t(x,\phi) = - \int\limits_{0}^{\phi} \frac{\partial V_t}{\partial x}(x, \psi) \, d\psi$.

Thus we have shown that $(x(t), \phi(t))$ satisfies the ODE
\begin{equation}\label{eq:process-with-color}
\begin{cases}
		x'(t) = V_t(x(t), \phi(t)) \\
		\phi'(t) = R_t(x(t), \phi(t)). \\
\end{cases}
\end{equation}
If $x(t) \neq 0,1$, this equation is equivalent to \eqref{eq:true-process}, i.e., $(x(t), \phi(t))$ is a solution of \eqref{eq:process-with-color} with initial conditions $(x(0), \phi(0)) = (x_0, \phi_0)$ if and only if $(x(t),v(t))$ is a solution of \eqref{eq:true-process} with initial conditions $(x(0), v(0))= (x_0, V_0(x_0, \phi_0))$. We also note that Lemma \ref{lm:velocity-mean-zero} expressed in terms of $(x, \phi)$ variables states that for each $t \in [0,T]$ and $x \in [0,1]$ we have
\begin{equation}\label{eq:mean-zero-color}
\int\limits_{0}^{1} V_t(x, \psi) \, d\psi = 0.
\end{equation}

From now on we work exclusively with \eqref{eq:process-with-color}. We will need to make two approximations necessary for the interacting particle system analysis later on. One is necessitated by the fact that the function $V_t(x, \phi)$ might not be smooth with respect to $x$ at the boundaries $x = 0,1$ (this happens, for example, for the sine curve process). We will therefore replace the function by its smooth approximation in a $\beta$-neighborhood of the boundary and in the end take $\beta \to 0$. The other approximation consists in dividing the time interval $[0,T]$ into intervals of length $\delta$ and approximating $V_t(x, \phi)$ for given $x, \phi$ with a piecewise-constant function of $t$. This will enable us to give a simple stationarity condition for the corresponding interacting particle system and in the end take $\delta \to 0$ a well.

Let $\beta \in (0,\frac{1}{4})$ and let $V_{t}^{\beta}(x, \phi)$ be a function with the following properties
\begin{enumerate}[(a)]
\item $V_{t}^{\beta}(x, \phi)$ is continuously differentiable for every $t \in [0,T]$, $x \in [0,1]$, $\phi \in [0,1]$,
\item $V_{t}^{\beta}(x, \phi) = V_{t}(x, \phi)$ for $x \in \left[\beta, 1 - \beta\right]$ and $V_{t}^{\beta}(0,\phi) = V_{t}^{\beta}(1,\phi) = 0$,
\item for each $x \in [0,1]$ we have $\int\limits_{0}^{1} V_{t}^{\beta}(x, \psi) \, d\psi = 0$,
\item $|V_{t}^{\beta}(x, \phi)| \leq |V_{t}(x, \phi)| + 1$,
\item we have $\lim\limits_{\beta \to 0} \int\limits_{0}^{1}\int\limits_{0}^{1} |V_{t}^{\beta}(x, \phi)|^2 \, dx \, d\phi = \int\limits_{0}^{1}\int\limits_{0}^{1} |V_{t}(x, \phi)|^2 \, dx \, d\phi$.
\end{enumerate}

The existence of such a function $V_{t}^{\beta}$ is proved at the end of this section. By $(x^{\beta}(t), \phi^{\beta}(t))$ we will denote the solution to the ODE
\begin{equation}\label{eq:process-with-delta}
\begin{cases}
		x'(t) = V^{\beta}_{t}(x(t), \phi(t)) \\
		\phi'(t) = R^{\beta}_{t}(x(t), \phi(t)). \\
\end{cases}
\end{equation}

Take any $\delta > 0$ (to simplify notation we will assume that $T$ is an integer multiple of $\delta$, this will not influence the argument in any substantial way) and consider a partition $0 = t_0 < t_1 < \ldots < t_M = T$ of $[0,T]$ into $M = \frac{T}{\delta}$ intervals of length $\delta$, with $t_k = k \delta$. Let $V^{\beta, \delta}(t, x,\phi)$ be the piecewise-constant in time approximation of $V^{\beta}_{t}(x, \phi)$, defined by

\begin{equation}\label{eq:def-of-piecewise-s}
V^{\beta, \delta}(t, x,\phi) = V^{\beta}_{t_k}(x,\phi) \quad \mbox{for $t \in [t_k, t_{k+1})$}, \, \, \mbox{$k=0,1,\ldots,M - 1$}.
\end{equation}

We can now define the piecewise-stationary process which will be our main tool in subsequent arguments. Consider the ODE
\begin{equation}\label{eq:invariant-process}
\begin{cases}
		y'(t) = V^{\beta, \delta}(t, y(t),\phi(t)) \\
		\phi'(t) = R^{\beta, \delta}(t, y(t), \phi(t)), \\
\end{cases}
\end{equation}
where
\[
R^{\beta, \delta}(t, y, \phi) = - \int\limits_{0}^{\phi} \frac{\partial V^{\beta, \delta}}{\partial y}(t, y, \psi) \, d\psi.
\]
Solutions to \eqref{eq:invariant-process} exist and are unique as usual for any initial conditions, provided we interpret $(y'(t), \phi'(t))$ above as right-handed derivatives at $t = 0, t_1, t_2, \ldots, t_{M-1}$ (we adopt this convention from now on).

Let $P^{\beta, \delta} = \left( (X^{\beta, \delta}_{t}, \Phi^{\beta, \delta}_{t}), 0 \leq t \leq T \right)$ be the stochastic process with values in $[0,1]^2$ with the following distribution: choose $(X^{\beta, \delta}_{0}, \Phi^{\beta, \delta}_{0})$ uniformly at random from $[0,1]^2$ and then take $(X^{\beta, \delta}_{t}, \Phi^{\beta, \delta}_{t}) = (y(t), \phi(t))$, where $(y, \phi)$ is the solution of the system \eqref{eq:invariant-process} with initial conditions given by $(y(0), \phi(0)) = (X^{\beta, \delta}_{0}, \Phi^{\beta, \delta}_{0})$. We will call this process the \emph{colored trajectory process} associated to \eqref{eq:invariant-process}.

We also define the process $P^{\beta} = \left( (X^{\beta}_{t}, \Phi^{\beta}_{t}), 0 \leq t \leq T \right)$, which is obtained in the same way as $P^{\beta, \delta}$ except that we follow solutions to \eqref{eq:process-with-delta} instead of \eqref{eq:invariant-process}, i.e., make no piecewise approximation in time of $V_{t}^{\beta}$. 

The key property of the process $P^{\beta, \delta}$ is the following
\begin{lemma}\label{lm:p-is-stationary}
For each $t \in [0,T]$ the distribution of $(X^{\beta, \delta}_t, \Phi^{\beta, \delta}_t)$ is uniform on $[0,1]^2$.
\end{lemma}

\begin{proof}
First we show that the process stays confined to $[0,1]^2$. Because of uniqueness of solutions to \eqref{eq:invariant-process} it is enough to show that if a solution starts in the interior of $[0,1]^2$, it never reaches the boundary, or, equivalently, that if a solution is at the boundary at some $t$, it is actually at the boundary for all $s \in [0,T]$. If $y(t) = 0$ or $1$ for any $t$, then $y'(t) = 0$, since $V^{\beta, \delta}(t, 0,\phi) = V^{\beta, \delta}(t, 1,\phi) = 0$ for any $\phi$. By uniqueness of solutions we then have $y(t) \equiv 0$ or $1$. If $\phi(t) = 0$ for any $t$, then $R^{\beta, \delta}(t, y, 0) = 0$ regardless of $y$, so as before $\phi'(t) = 0$ and $\phi(t) \equiv 0$. Finally, if $\phi(t) = 1$, then using the property (c) of the function $S^{\beta}_{t}(x, \phi)$ we have 
\[
R^{\beta, \delta}(t,y,1) = - \int\limits_{0}^{1} \frac{\partial V^{\beta, \delta}}{\partial y}(t, y, \psi) \, d\psi = - \frac{\partial}{\partial y} \int\limits_{0}^{1} V^{\beta, \delta} (t, y, \psi) \, d\psi = 0,
\]
so as before $\phi'(t) = 0$ and $\phi(t) \equiv 1$.

Now we observe that the form of $V^{\beta, \delta}$ and $R^{\beta, \delta}$ in \eqref{eq:invariant-process} implies that the vector field $(V^{\beta, \delta}(t,\cdot,\cdot),R^{\beta, \delta}(t,\cdot,\cdot))$ is divergence-free at each $t$, so by Liouville's theorem the uniform measure on $[0,1]^2$ is invariant for the corresponding flow map.
\end{proof}

In particular, the process $X^{\beta, \delta} = (X^{\beta, \delta}_t, 0 \leq t \leq T)$ is a permuton process. Crucially, we can couple it to the process $X$ in a natural way. Consider $(x_0, \phi_0)$ chosen uniformly at random from $[0,1]^2$ and take $(x(0), v(0)) = (x_0, V_{0}(x_0, \phi_0))$, resp. $(y(0), \phi(0)) = (x_0, \phi_0)$, as initial conditions for \eqref{eq:process-with-color}, resp. \eqref{eq:invariant-process}. By definition of $V_0(x, \phi)$, the pair $(x(0), v(0))$ has distribution given by $\mu_0$, so indeed the pair of solutions $(x(t), y(t))$ corresponding to the initial conditions above defines a coupling of $X$ and $X^{\beta, \delta}$. From now on $X$ and $X^{\beta, \delta}$ are always assumed to be coupled in this way.

It is readily seen that the statements above also hold for $P^{\beta}$ instead of $P^{\beta, \delta}$, hence with a slight abuse of notation we can allow $\delta = 0$ and write $P^{\beta, 0} = P^{\beta}$, $X^{\beta,0} = X^{\beta}$ etc.

Our goal in the remainder of this section is to show that, as $\beta, \delta \to 0$, the processes $X$ and $X^{\beta, \delta}$ typically stay close to each other and have approximately the same Dirichlet energy, so in the probabilistic part of the arguments it will be enough to work with the process $(X^{\beta, \delta}, \Phi^{\beta, \delta})$, which is more convenient thanks to piecewise stationarity.

First we prove a simple lemma, showing that $X^{\beta}$ is unlikely to ever be close to the boundary (so that approximation of $X$ with $X^{\beta}$ is meaningful as $\beta \to 0$).
\begin{lemma}\label{lm:no-visits-to-boundary}
Let $\Pp$ denote the law of the process $X^{\beta}$. Let
\[
B^{\beta} = \left\{ \exists t \in [0,T] \, X^{\beta}_{t} \notin [\beta, 1 - \beta] \right\}.
\]
We have
\[
\Pp\left( B^{\beta}  \right) \xrightarrow{\beta \to 0} 0.
\]
\end{lemma}

\begin{proof}
We will prove that $X^{\beta}_{t} \notin [0, \beta]$ with high probability as $\beta \to 0$ (the proof for $[1-\beta, 1]$ is analogous). Suppose that $y$ is a solution of \eqref{eq:invariant-process} with initial condition $y(0) \notin [0, 2\beta]$ and that $y(t) \in [0, \beta]$ for some $t \in [0,T]$. Then there exists a time interval $[s,s']$ such that $y(s) = 2\beta$, $y(s') = \beta$ and $y(u) \in [\beta, 2 \beta]$ for every $u \in [s,s']$. Without loss of generality we can assume that $[s,s'] \subseteq [t_k, t_{k+1})$ for some $k$ (the other case is easily dealt with by further subdividing $[\beta, 2 \beta]$ into two equal subintervals and repeating the argument for each of them). By the mean value theorem
\[
|y(s) - y(s')| = (s' - s) y'(w)
\]
for some $w \in [s,s']$. For $x \in [\beta, 2\beta]$ we have $V^{\beta}(w,x,\phi) = V_{t_k}(x,\phi)$, so $y'(w) = V_{t_k}(w, y(w), \phi(w))$. Since $|y(w)| \leq 2\beta$ and $V_{t_k}(x,\phi)$ is continuous at $x=0$, we have $|y'(w)| \leq f(\beta)$ for some function $f$ (depending only on $V$) satisfying $\lim\limits_{\beta \to 0} f(\beta) = 0$. As $|y(s) - y(s')| = \beta$, altogether this implies that $s'- s \geq \frac{\beta}{f(\beta)}$, i.e., if the process $X^{\beta}$ starts outside $[0, 2\beta]$, it has to spend time at least $\frac{\beta}{f(\beta)}$ before it reaches $[0, \beta]$. Thus
\[
\int\limits_{0}^{T} \id_{\{ X^{\beta}_{s} \in [0, \beta] \}} \, ds \geq \frac{\beta}{f(\beta)} \id_{\{\exists t \in [0,T] \, X^{\beta}_{t} \in [0,\beta]\}} \id_{\{X^{\beta}_{0} \notin [0,2\beta]\}}.
\]
Taking expectation yields
\[
\E \int\limits_{0}^{T} \id_{\{ X^{\beta}_{s} \in [0, \beta] \}} \, ds \geq \frac{\beta}{f(\beta)} \Pp \left( \{\exists t \in [0,T] \, X^{\beta}_{t} \in [0,\beta]\} \cap \{X^{\beta}_{0} \notin [0,2\beta]\} \right).
\]
Since $X^{\beta}$ is a permuton process, $X^{\beta}_s$ has uniform distribution for each $s$, which gives
\[
\E \int\limits_{0}^{T} \id_{\{ X^{\beta}_{s} \in [0, \beta] \}} \, ds = \int\limits_{0}^{T} \E \id_{\{ X^{\beta}_{s} \in [0, \beta] \}} \, ds = \int\limits_{0}^{T} \Pp \left( X^{\beta}_{s} \in [0, \beta] \right) ds = T \beta.
\]
Together with the inequality above this implies
\[
T \beta \geq \frac{\beta}{f(\beta)} \left( \Pp \left( \exists t \in [0,T] \, X^{\beta}_{t} \in [0,\beta] \right) - \Pp(X^{\beta}_{0} \in [0, 2\beta])\right).
\]
Since $X^{\beta}_{0}$ has uniform distribution, we have $\Pp(X^{\beta}_{0} \in [0, 2\beta]) = 2\beta$. Thus
\[
\Pp \left( \exists t \in [0,T] \, X^{\beta}_{t} \in [0,\beta] \right) \leq 2 \beta + T f(\beta).
\]
Since $f(\beta) \to 0$ as $\beta \to 0$, the claim is proved.
\end{proof}

\begin{proposition}\label{prop:x-and-y-are-close}
Fix $\beta \in (0, \frac{1}{4})$ and $(x_0, \phi_0) \in [0,1]^2$. Let $(x^{\beta}(t), \phi^{\beta}(t))$, resp. $(x^{\beta, \delta}(t), \phi^{\beta, \delta}(t))$, be the solution to \eqref{eq:process-with-delta}, resp. \eqref{eq:invariant-process}, with initial conditions $(x_0, \phi_0)$. We have
\begin{align*}
& \sup\limits_{t \in [0,T]} | x^{\beta, \delta}(t) - x^{\beta}(t)| \xrightarrow{\delta \to 0} 0, \\
& \sup\limits_{t \in [0,T]} | \phi^{\beta, \delta}(t) - \phi^{\beta}(t)| \xrightarrow{\delta \to 0} 0.
\end{align*}
\end{proposition}

\begin{proof}
The statement follows from continuous dependence of solutions to an ODE on parameters, see e.g., \cite[Theorem 4.2]{ode-book}. Denoting $V^{\beta, 0}(t, y, \phi) = V^{\beta}_t(y, \phi)$, $R^{\beta, 0}(t, y, \phi) = R^{\beta}_t(y, \phi)$, we only need to check that for $f(t, y, \phi, \delta) = V^{\beta, \delta}(t,y,\phi)$, $g(t, y, \phi, \delta) = R^{\beta, \delta}(t, y, \phi)$ we have
\begin{enumerate}[(1)]
\item $f(\cdot, y, \phi, \delta)$ and $g(\cdot, y, \phi, \delta)$ are measurable on $[0,T]$,
\item for any fixed $t \in [0,T]$ and $\delta > 0$ $f(t,\cdot,\cdot,\delta)$ and $g(t,\cdot,\cdot,\delta)$ are continuous in $(y, \phi)$,
\item for any fixed $t \in [0,T]$ $f(t,\cdot,\cdot,\cdot)$ and $g(t,\cdot,\cdot,\cdot)$ are continuous in $(y,\phi,\delta)$ at $\delta = 0$,
\item $f(t,y,\phi,\delta)$, $g(t,y,\phi,\delta)$ are uniformly bounded.
\end{enumerate}
Properties 1), 2) and 4) follow directly from our regularity assumptions about $V^{\beta, \delta}(t,y,\phi)$ (in case of $R^{\beta, \delta}(t,y,\phi)$ we use continuity of $\frac{\partial V^{\beta, \delta}}{\partial y}(t, y, \phi)$). Property 3) follows from pointwise convergence $f(t,y,\phi,\delta) \xrightarrow{\delta \to 0} f(t,y,\phi,0)$ and equicontinuity of $\{f(t,y,\phi,\delta)\}_{\delta \geq 0}$ in $(y, \phi)$, which in turn follows from uniform continuity of $V^{\beta}_{t}(y, \phi)$ in $t, y$ and $\phi$. The argument for $g(t,y,\phi, \delta)$ is analogous (again, using uniform continuity of $\frac{\partial V^{\beta, \delta}}{\partial y}(t, y, \phi)$).
\end{proof}

Now we can prove the main result of this section, which states that the trajectories of the process $X$ and its energy can be approximated by those of the process $X^{\beta, \delta}$.

\begin{proposition}\label{prop:approximation-epsilon-delta}
Let $\pi \in \MM(\DD)$ be the distribution of the process $X$ and let $\pi^{\beta, \delta} \in \MM(\DD)$ be the distribution of the process $X^{\beta, \delta}$. Then we have
\[
\lim\limits_{\beta \to 0} \lim\limits_{\delta \to 0} d_{\mathcal{W}}^{sup}(\pi, \pi^{\beta, \delta}) = 0,
\]
where $d_{\mathcal{W}}^{sup}$ is the Wasserstein distance associated to the supremum norm on $\DD$.

Furthermore,
\[
\lim\limits_{\beta \to 0} \lim\limits_{\delta \to 0} I(\pi^{\beta, \delta}) = I(\pi),
\]
where $I(\mu)$ is the energy of the process $\mu$ defined in \eqref{eq:process-energy}.
\end{proposition}

\begin{proof}
For the first convergence it is enough to show that $\E \norm{X - X^{\beta, \delta}}_{sup} \to 0$ in the coupling between $X$ and $X^{\beta, \delta}$ considered before. We have
\[
\E \norm{X - X^{\beta, \delta}}_{sup} \leq \E \norm{X - X^{\beta}}_{sup} + \E \norm{X^{\beta} - X^{\beta, \delta}}_{sup}.
\]
Let $B^{\beta}$ be the event from the statement of Lemma \ref{lm:no-visits-to-boundary}. Since the supremum norm is bounded by $1$, we have
\[
\E \norm{X - X^{\beta}}_{sup} \leq \Pp\left( B^{\beta} \right) + \E \left[\norm{X - X^{\beta}}_{sup} \id_{(B^{\beta})^c} \right].
\]
By Lemma \ref{lm:no-visits-to-boundary} the first term is $o(1)$ as $\beta \to 0$. Since $V^{\beta}_{t}(x, \phi) = V_{t}(x, \phi)$ if $x \in [\beta, 1-\beta]$, on the event $(B^{\beta})^c$ we have $X^{\beta} = X$, so the second term above is equal to $0$. As for $\E\norm{X^{\beta} - X^{\beta, \delta}}_{sup}$, by Proposition \ref{prop:x-and-y-are-close} for fixed $\beta > 0$ we have with probability one $\norm{X^{\beta} - X^{\beta, \delta}}_{sup} \to 0$ as $\delta \to 0$, which together with the estimate on $\E \norm{X - X^\beta}_{sup}$ proves the first claim of the theorem.

As for the energy, let $\pi^\beta$ denote the distribution of the process $X^\beta$, with $X$, $X^\beta$ and $X^{\beta, \delta}$ coupled as before. Since
\[
|I(\pi) - I(\pi^{\beta, \delta})| \leq |I(\pi) - I(\pi^{\beta})| + |I(\pi^{\beta}) - I(\pi^{\beta, \delta})|
\]
it is enough to show that $\lim\limits_{\delta \to 0} I(\pi^{\beta, \delta}) = I(\pi^{\beta})$ and $\lim\limits_{\beta \to 0} I(\pi^{\beta}) = I(\pi)$. We have
\[
I(\pi^{\beta, \delta}) = \E \int\limits_{0}^{T} |\dot{X}^{\beta, \delta}(t)|^2 \, dt = \E \int\limits_{0}^{T} V^{\beta, \delta}(t, X^{\beta, \delta}(t), \Phi^{\beta, \delta}(t))^2 \, dt.
\]
For fixed $t \in [0,T]$ by Lemma \ref{lm:p-is-stationary} $(X^{\beta, \delta}(t), \Phi^{\beta, \delta}(t))$ has uniform distribution on $[0,1] \times [0,1]$ and moving the expectation inside the integral we obtain
\[
I(\pi^{\beta, \delta}) = \int\limits_{0}^{T} \E \left[ V^{\beta, \delta}(t, X^{\beta, \delta}(t), \Phi^{\beta, \delta}(t))^2 \right] dt = \int\limits_{0}^{T} \left( \int\limits_{0}^{1} \int\limits_{0}^{1} V^{\beta, \delta}(t, x, \phi)^2  \, dx \, d\phi\right) dt.
\]
The analogous formula is valid for $I(\pi)$ as well. Now, for fixed $\beta > 0$ we have $V^{\beta, \delta}(t,x,\phi) \xrightarrow{\delta \to 0} V^{\beta}_t(x,\phi)$ and $V^{\beta, \delta}(t,x,\phi)$ is uniformly bounded in $t, x$ and $\phi$, independently of $\delta$, which by dominated convergence implies the convergence of the integrals above as well. Thus $\lim\limits_{\delta \to 0} I(\pi^{\beta, \delta}) = I(\pi^{\beta})$. The convergence $\lim\limits_{\beta \to 0} I(\pi^{\beta}) = I(\pi)$ follows directly from properties (d) and (e) of $V^{\beta}_{t}(x, \phi)$ and dominated convergence.
\end{proof}

\paragraph{Construction of $V_{t}^{\beta}$.} We will construct the desired modification of $V_t(x, \phi)$ for $x \in [0,\beta]$, the construction for $x \in [1-\beta,1]$ is analogous. Fix $t \in [0,T]$. Let
\[
L_t(x, \phi) = \frac{\partial V_t}{\partial x}(\beta, \phi) (x - \beta) + V_{t}(\beta, \phi).
\]
Consider $\beta' < \beta / 2$ to be fixed later and let $f$ be a smooth approximation of a step function which has values in $[0,1]$, is equal to $0$ on $[0, \beta - 2 \beta']$, equal to $1$ on $[\beta' - \beta, \beta]$ and is increasing on $[\beta - 2\beta', \beta - \beta']$. In particular we have $f(0) = 0$, $f(\beta) = 1$ and $f'(\beta) = 0$.

Let us now take
\[
\widetilde{V}_t(x, \phi) = f(x)L_t(x, \phi)
\]
and
\[
V^{\beta}_{t}(x, \phi) =
\begin{cases}
\widetilde{V}_t(x, \phi) \quad \mbox{for $x \in [0, \beta]$}, \\
V_{t}(x, \phi) \quad \mbox{otherwise}.
\end{cases}
\]
We will check that $V^{\beta}_{t}(x, \phi)$ indeed satisfies the desired properties.

Let us first check that the property (c) is satisfied for $x \in [0, \beta]$. We have
\begin{align*}
& \int\limits_{0}^{1} \widetilde{V}_t(x, \psi) \, d\psi = \int\limits_{0}^{1} f(x)L_t(x, \psi) \, d\psi = f(x) \int\limits_{0}^{1} \left( \frac{\partial V_t}{\partial x}(\beta, \psi) (x - \beta) + V_{t}(\beta, \psi) \right) d\psi = \\
& = f(x)(x - \beta) \int\limits_{0}^{1} \frac{\partial V_t}{\partial x}(\beta, \psi) \, d\psi + f(x) \int\limits_{0}^{1} V_{t}(\beta, \psi) \, d\psi = \\
& = f(x)(x - \beta) \frac{d}{d x}\Big\vert_{x=\beta} \left( \int\limits_{0}^{1} V_t(x, \psi) \, d\psi \right) + f(x) \int\limits_{0}^{1} V_{t}(\beta, \psi) \, d\psi = 0,
\end{align*}
thanks to \eqref{eq:mean-zero-color}.

Property (b) follows directly from $f(0) = 0$. As for (a), for $x \in [0,\beta)$ continuous differentiability of $V^{\beta}_t(x, \phi)$ follows from continuous differentiability of $f(x)$. At $x = \beta$ we have
\[
\widetilde{V}_t(\beta, \phi) = f(\beta) L_t (\beta, \phi) = f(\beta) V_t(\beta, \phi)
\]
and $f(\beta) = 1$, so $\widetilde{V}_t(x, \phi)$ is continuous at $x=\beta$. Likewise,
\[
\frac{\partial \widetilde{V}_t}{\partial x}(x, \phi) = f'(x) L_t (x, \phi) + f(x) \frac{\partial L_t}{\partial x}(x, \phi) = f'(x) L_t (x, \phi) + f(x) \frac{\partial V_t}{\partial x} (x, \phi).
\]
Since $f(\beta) = 1$ and $f'(\beta) = 0$, we have $\frac{\partial \widetilde{V}_t}{\partial x}(\beta, \phi) = \frac{\partial V_t}{\partial x} (\beta, \phi)$. As the functions in the formula above are continuously differentiable at $x=\beta$, $V^{\beta}_t(x, \phi)$ is continuously differentiable at $x=\beta$ as well.

To see that property (d) is satisfied, we note that by continuity of $V_t(x, \phi)$ and $\frac{\partial V_t}{\partial x}(x, \phi)$ for $x \neq 0,1$ we can take $\beta'$ in the definition of $f(x)$ above to be arbitrarily small (depending on $V_t$, $\frac{\partial V_t}{\partial x}$ and $\beta$) so that on $[\beta - 2 \beta', \beta]$ the function $\widetilde{V}_t(x, \phi)$ is less than $|V_t(\beta, \phi)| + 1$ in absolute value. Since on $[0, \beta - 2\beta']$ we have $V_{t}^{\beta}(x, \phi) = 0$, the desired bound on $|V_{t}^{\beta}(x, \phi)|$ follows.

Finally, to prove that property (e) holds it is enough to show that
\[
\int\limits_{0}^{1} \int\limits_{0}^{\beta} |V_{t}^{\beta}(x, \phi)|^2 \, dx \, d\phi \to 0
\]
as $\beta \to 0$, since $V_{t}^{\beta}(x, \phi) = V_{t}(x, \phi)$ for $x \notin [0,\beta]$. The claim follows immediately from property (d), since the integrand is bounded independently of $\beta$.
\end{section}

\begin{section}{The biased interchange process and stationarity}\label{sec:interchange}

\paragraph{The biased interchange process.}

For the sake of proving a large deviation lower bound, we will need to perturb the interchange process to obtain dynamics which typically exhibits (otherwise rare) behavior of a fixed permuton process. Let us introduce the \emph{biased interchange process}. Its configuration space $E$ consists of sequences $\eta = \left( (x_{i}, \phi_{i})\right)_{i=1}^{N}$, where as before $(x_1, \ldots, x_N)$ is a permutation of $\{1, \ldots, N\}$ and $\phi_{i}$ has $N$ possible values, $1, \ldots, N$. Here $x_{i}$ will be the \emph{position} of the particle with label $i$ and $\phi_{i}$ will be its \emph{color}.

By a slight abuse of notation we will write $\eta^{-1}(x)$ to denote the label (number) of the particle at position $x$ in configuration $\eta$ (so that $\eta^{-1}(x_{i}) = i$). For a position $x$ we will often write $\phi_{x}$ as a shorthand for $\phi_{\eta^{-1}(x)}$ (the positions will be always denoted by $x$ or $y$ and labels by $i$, so there is no risk of ambiguity). In this way we can treat any configuration $\eta$ as a function which assigns to each site $x$ a pair $(\eta^{-1}(x), \phi_x)$, the label and the color of the particle present at $x$

The configuration at time $t$ will be denoted by $\eta^{N}_{t}$ (or simply $\eta_{t}$), and likewise by $x_{i}(\eta^{N}_{t})$ and $\phi_{i}(\eta^{N}_{t})$ we denote the position and the color of the particle number $i$ at time $t$. We will use notation $X_{i}(\eta^{N}_{t}) = \frac{1}{N}x_{i}(\eta^{N}_{t})$, $\Phi_{i}(\eta^{N}_{t}) = \frac{1}{N} \phi_{i}(\eta^{N}_{t})$ for the rescaled positions and colors. By the same convention as above $\Phi_x(\eta^{N}_{t})$ will denote the rescaled color of the particle at site $x$ at time $t$.

Let $\varepsilon = N^{1-\alpha}$, with the same $\alpha \in (1,2)$ as in \eqref{eq:unbiased-generator}. Suppose we are given functions $v, r : [0,T] \times \{1, \ldots, N\} \times \{1, \ldots, N\}$. The dynamics of the corresponding biased interchange process is defined by the (time-inhomogeneous) generator
\begin{align}\label{eq:biased-generator}
& (\widetilde{\Ll}_t f)(\eta) = \frac{1}{2} N^{\alpha} \sum\limits_{x=1}^{N-1} \big( 1 + \varepsilon \left[v(t, x, \phi_{x}(\eta)) - v(t,x+1, \phi_{x+1}(\eta))\right] \big) (f(\eta^{x, x+1}) - f(\eta)) + \\
& + \frac{1}{2} N^{\alpha} \sum\limits_{x=1}^{N} \Big[ \big( 1 + \varepsilon r(t,x, \phi_{x}(\eta)) \big)  (f(\eta^{x,+}) - f(\eta)) + \big( 1 - \varepsilon r (t,x, \phi_{x}(\eta)) \big) (f(\eta^{x,-}) - f(\eta)) \Big]. \nonumber
\end{align}
Here $\eta^{x, x+1}$ is the configuration $\eta$ with particles at locations $x$ and $x+1$ swapped, and $\eta^{y, \pm}$ is the configuration $\eta$ with $\phi_{y}$ changed by $\pm 1$ (with the convention that $\eta^{y,+} = \eta^{y}$ if $\phi_y = N$ and likewise $\eta^{y,-} = \eta^{y}$ if $\phi_y = 1$). We will often use the abbreviated notation $v_{x}(t,\eta) = v(t,x, \phi_{x}(\eta))$ (with the convention $v_{0}(t,\eta) = v_{N+1}(t,\eta) = 0$).

In other words, at each time neighboring particles make a swap at rate close to $1$, with bias proportional to the difference of their \emph{velocities} $v(t,x,\phi_{x})$, and each particle independently changes its color by $\pm 1$, also at rate close to $1$ with bias proportional to $\pm r(t,x, \phi_{x})$. The parameter $\varepsilon$ has been chosen so that we expect particles to have displacement of order $N$ at macroscopic times.

Since the interchange process is a pure jump Markov process, for each particle its rescaled position $X_{i}(\eta^{N})$ and color $\Phi_{i}(\eta^{N})$ will be c\`adl\`ag paths from $[0,T]$ to $[0,1]$ and thus elements of $\DD$. In the same way we can consider the joint trajectory $P_{i}(\eta^{N}) =  (X_{i}(\eta^{N}), \Phi_{i}(\eta^{N}))$ as an element of $\DDD = \DD([0,T], [0,1]^2)$, the space of c\'adl\'ag paths from $[0,T]$ to $[0,1]^2$ (equipped with the Skorokhod topology). By $\MM(\DDD)$ we will denote the space of Borel probability measures on $\DDD$, endowed with the weak topology, and by a slight abuse of notation the corresponding Wasserstein distance will be denoted by $d_{\mathcal{W}}$, as for $\MM(\DD)$.

If $\eta^{N}$ is the trajectory of the biased interchange process, then by analogy with the permutation process $X^{\eta^{N}}$ we can define the \emph{colored permutation process} $P^{\eta^{N}} = (X^{\eta^{N}}, \Phi^{\eta^{N}})$, obtained by choosing a particle $i$ at random and following the path $(X_{i}(\eta^{N}_{t}), \Phi_{i}(\eta^{N}_{t}))$. Thus we keep track both of the position and the color of a random particle. Since $\eta^N$ is random, the distribution $\nu^{\eta^N}$ of $P^{\eta^{N}}$, given by
\[
\nu^{\eta^N} = \frac{1}{N} \sum\limits_{i=1}^{N} \delta_{P^{\eta^{N}}_{i}},
\]
is a random element of $\MM(\DDD)$.

\paragraph{Stationarity conditions.}

Let us now connect the discussion of the interchange process with deterministic permuton processes and generalized solutions to Euler equations considered in Section \ref{sec:ode-part}. Recall the colored trajectory process $P^{\beta, \delta} = (X^{\beta, \delta}, \Phi^{\beta, \delta})$ defined in Section \ref{sec:ode-part}. From now on we consider $\beta \in (0, \frac{1}{4})$ and $\delta > 0$ to be fixed and we suppress them in the notation, writing $X = X^{\beta, \delta}$, $\Phi = \Phi^{\beta, \delta}$, $V(t, x, \phi) = V^{\beta, \delta}(t, x, \phi)$, $R(t, x, \phi) = R^{\beta, \delta}(t, x, \phi)$. Note that this should not be confused with the actual generalized solution to Euler equations, which was also denoted by $X$, but does not appear in this and the following sections except in Theorem \ref{thm:lower-bound-for-minimizers}.

Our goal is to set up a biased interchange process so that typically trajectories of particles will behave like trajectories of the process $X$. We would also like to preserve the stationarity of the uniform distribution of colors, which will greatly facilitate parts of the argument. To find the correct rates $v(t, x, \phi)$ and $r(t, x, \phi)$ in \eqref{eq:biased-generator}, recall that by definition the trajectories of the colored trajectory process $P = (X, \Phi)$ satisfy the equation
\begin{equation}\label{eq:ode-general}
	\begin{cases}
		\frac{dX}{dt}(t) = V(t, X(t), \Phi(t))  \\
		\frac{d\Phi}{dt}(t) = R(t, X(t), \Phi(t)), \\
	\end{cases}
\end{equation}
with the functions $V$ and $R$ satisfying
\begin{align}\label{eq:ode-rates}
\begin{cases}
 V(t,X, \Phi) = \frac{\partial F}{\partial \Phi}(t,X, \Phi) \\
 R(t,X, \Phi) = -\frac{\partial F}{\partial X}(t,X, \Phi)
\end{cases}
\end{align}
for $F(t,X,\Phi) = \int\limits_{0}^{\Phi} V(t,X,\psi) \, d\psi$. Note that $F(t,X,0) = 0$ and $F(t,X,1)=0$, where the latter equality follows from property (c) of $V^{\beta}_{t}(x,\phi)$ (and thus of $V = V^{\beta,\delta}$).

It is clear that $v$ and $r$ should be chosen so that approximately we have $v(t,x, \phi) \approx V\left(t,\frac{x}{N} , \frac{\phi}{N}\right)$, $r(t,x, \phi) \approx R\left(t,\frac{x}{N} , \frac{\phi}{N}\right)$. To analyze the stationarity condition, consider the uniform distribution on configurations of the biased interchange process, i.e., a distribution in which the labelling of particles is a uniformly random permutation and each particle has a uniformly random color, chosen indepedently from $\{1, \ldots, N\}$ for each of them. We want to find a condition on rates $v(t,x, \phi)$ and $r(t,x, \phi)$ such that this measure will be invariant for the dynamics of $\widetilde{\Ll}_t$.

Note that since $V(t,X,\Phi)$, $R(t,X,\Phi)$ are piecewise-constant as functions of $t$, the dynamics induced by $\widetilde{\Ll}_t$ is time-homogeneous on each interval $[t_k, t_{k+1})$ from the definition \eqref{eq:def-of-piecewise-s} of $V$. Thus the stationarity condition for the uniform measure is that for each state (i.e., each configuration $\eta$) the sums of outgoing and incoming jump rates have to be equal. We write down this condition as follows. For any given configuration $\eta$, with particle at location $x$ having color $\phi_{x} = \phi_{x}(\eta)$, there are the following possible outgoing jumps:
\begin{itemize}
\item for some $x \in \{ 1, \ldots, N - 1\}$ the particles at locations $x$ and $x+1$ swap, at rate $1 + \varepsilon \left[ v (t,x, \phi_{x}) - v (t,x+1, \phi_{x+1}) \right]$
\item for some $x \in \{1, \ldots, N\}$ the particle at $x$ changes its color from $\phi_x$ to $\phi_x \pm 1$, at rate $1 \pm \varepsilon r (t,x, \phi_{x})$
\end{itemize}
and incoming jumps:
\begin{itemize}
\item for some $x \in \{ 1, \ldots, N - 1\}$ the particles at locations $x$ and $x+1$ swap, at rate $1 + \varepsilon \left[ v (t,x, \phi_{x+1}) - v (t,x+1, \phi_{x}) \right]$
\item for some $x \in \{1, \ldots, N\}$ the particle at $x$ changes its color from $\phi_x \pm 1$ to $\phi_x$, at rate $1 \mp \varepsilon r (t,x, \phi_{x} \pm 1)$
\end{itemize}

Thus the condition on the sums of jump rates is
\begin{align*}
& \sum\limits_{x=1}^{N-1} \left( v (t,x, \phi_{x}) - v (t,x+1, \phi_{x+1}) \right) = \\
& = \sum\limits_{x=1}^{N-1} \left( v (t,x, \phi_{x+1}) - v (t,x+1, \phi_{x}) \right) +  \sum\limits_{x=1}^{N} \left( r (t,x, \phi_{x} - 1) - r (t,x, \phi_{x} + 1) \right),
\end{align*}
where we adopt the convention $r(t,x,0) = r(t,x,N+1) = 0$. This implies
\begin{align*}
& \sum\limits_{x=2}^{N-1} \big( v (t,x-1, \phi_{x}) - v (t,x+1, \phi_{x}) +  r(t,x,\phi_{x} - 1) - r(t,x,\phi_{x} + 1) \big)  + \\
& + v(t,N-1, \phi_{N}) + v(t,N, \phi_{N}) - v(t,1, \phi_{1}) - v(t,2, \phi_{1}) + \\
& +  \left[ r(t,1,\phi_{1} - 1) - r(t,1,\phi_{1} + 1) \right] + \left[ r(t,N,\phi_{N} - 1) - r(t,N,\phi_{N} + 1) \right] = 0.
\end{align*}
Since we would like this equation to be satisfied for any configuration, regardless of the choice of $\phi_x$ for each $x$, we want each term in the sum and each of the boundary terms to vanish. This gives us a set of equations
\begin{align}\label{eq:stationarity-eqs}
\begin{cases}
v(t,1, \phi) + v(t,2, \phi) = r(t,1,\phi - 1) - r(t,1,\phi + 1)  \\
v (t,x+1, \phi) - v (t,x-1, \phi)  = r(t,x,\phi - 1) - r(t,x,\phi + 1), \, \, \, x = 2, \ldots, N-1 \\
v(t,N-1, \phi) + v(t,N, \phi) =  r(t,N,\phi + 1) - r(t,N,\phi - 1)
\end{cases}
\end{align}
which have to be satisfed for every $\phi = 1, \ldots, N$.

Let us consider the function $f(t,x, \phi)$ defined for $x \in \{0, \ldots, N+1\}$, $\phi \in \{1, \ldots, N\}$ by
\begin{align*}
f (t,x, \phi) = 
\begin{cases}
F\left( t,\frac{x}{N}, \frac{\phi}{N+1} \right), \quad & x = 2, \ldots, N-1, \\
0, \quad & x = 0, 1, N, N + 1,
\end{cases}
\end{align*}
where $F$ is the function appearing in \eqref{eq:ode-rates}. It is straightforward to check that the rates given by
\begin{equation}\label{eq:rates-s-r}
\begin{cases}
v(t,x, \phi) =  \frac{N}{2} \left( f(t,x, \phi - 1) - f(t,x, \phi + 1) \right) \\
r(t,x, \phi) =  \frac{N}{2} \left( f(t,x+1, \phi) - f(x-1,\phi) \right)
\end{cases}
\end{equation}
solve the equations for stationarity, given by \eqref{eq:stationarity-eqs}, for any $x, \phi \in \{1, \ldots, N\}$.

Note that with this choice of rates we have for any $x, \phi \in \{1, \ldots, N\}$
\begin{equation}\label{eq:s-r-approx}
\begin{cases}
v(t,x, \phi) = V \left(t, \frac{x}{N}, \frac{\phi}{N} \right) + O\left( \frac{1}{N} \right) \\
r(t,x, \phi) = R \left(t, \frac{x}{N}, \frac{\phi}{N} \right) + O\left( \frac{1}{N} \right)
\end{cases}
\end{equation}
uniformly in $x$, $\phi$ and $t$, because of smoothness of $F(t, X, \Phi)$ in $X$ and $\Phi$ variables. In particular the rates $v$ and $r$ are uniformly bounded for all $N$.

From now on we will always assume that the biased interchange process has rates $v(t,x, \phi)$ and $r(t,x, \phi)$ given by \eqref{eq:rates-s-r} and is started from the uniform distribution (which by the discussion above is stationary). The properties of $v$ and $r$ which will be relevant to our analysis is that they are bounded, approximately equal to some smooth functions $V$, $R$, that the corresponding dynamics has the uniform measure as the stationary distribution and, crucially, that in stationarity the velocities are independent and mean zero. This last property, which should be thought of as the particle system analog of Lemma \ref{lm:velocity-mean-zero}, is conveniently summarized in the following proposition.

\begin{proposition}\label{prop:speeds}
Let $\phi_{x}$, $x=1, \ldots, N$, be independent and uniformly distributed on $\{1, \ldots, N\}$. Then for each $t \in [0,T]$ the random variables $v(t,x, \phi_x)$, $x=1, \ldots, N$, are independent and for each $x$ we have
\[
\E \, v(t,x, \phi_x) = 0.
\]
\end{proposition}

\begin{proof}
Under the uniform distribution of $\phi_x$ we have
\[
\E \, v(t,x, \phi_x) = \frac{1}{N}\sum\limits_{\phi=1}^{N} v(t,x,\phi),
\]
which by definition of $v$ is equal to
\begin{align*}
\frac{1}{N}\sum\limits_{\phi=1}^{N} \frac{N}{2} \left( f(t,x, \phi - 1) - f(t,x, \phi + 1) \right) = \frac{1}{2} \left( F\left(t, \frac{x}{N}, 0\right) - F\left(t,\frac{x}{N}, 1\right)\right).
\end{align*}
Recalling the definition of $F$ below \eqref{eq:ode-rates}, the right-hand side is equal to $0$.
\end{proof}
\end{section}
\begin{section}{Law of large numbers}\label{sec:lln}

Throughout this section $\widetilde{\Pp}^{N}$ will denote the probability law of the biased interchange process on $N$ particles, started in stationarity, associated to the equation \eqref{eq:ode-general} (with all the assumptions from the previous section). To simplify notation we will usually write $\eta = \eta^{N}$. Whenever we use $o(\cdot)$ or $O(\cdot)$ asymptotic notation the implicit constants will depend only on the rates $v$, $r$ and possibly on $T$.

Let $P = (X, \Phi)$ be the colored trajectory process associated to the equation \eqref{eq:ode-general} and let $P^{\eta^{N}}$ be the colored permutation process defined in Section \ref{sec:interchange}. Let us denote the distributions of $P$ and $P^{\eta^{N}}$ respectively by $\nu$ and $\nu^{\eta^{N}}$, with $\nu, \nu^{\eta^{N}} \in \MM(\DDD)$.
We will prove the following theorem
\begin{theorem}\label{th:lln}
Let $\eta^{N}$ be the trajectory of the biased interchange process. The measures $\nu^{\eta^{N}}$ converge in distribution, as random elements of $\MM(\DDD)$, to the deterministic measure $\nu$ as $N \to \infty$.
\end{theorem}

In other words, the random processes $P^{\eta^{N}}$ converge in distribution to the process $P$ whose distribution is deterministic. The theorem above can be thought of as a law of large numbers for random permuton processes and it will be useful for establishing the large deviation lower bound.

\begin{remark}\label{rm:lln}
Since the limiting measure $\nu$ is deterministic and supported on continuous trajectories, Theorem \ref{th:lln} implies that the convergence $\nu^{\eta^N} \to \nu$ in fact holds in a stronger sense, namely in probability when $\MM(\DDD)$ is endowed with the Wasserstein distance $d_{\mathcal{W}}^{sup}$ associated to the supremum norm on $\DDD$.
\end{remark}
To prove Theorem \ref{th:lln}, we will show that typically trajectories of most particles approximately follow the same ODE \eqref{eq:ode-general} as trajectories of the limiting process. In other words, if a given particle is at site $x$, it should locally move according to its velocity $v(t, x, \phi_{x})$. However, because of swaps between particles the actual jump rates of the particle will be influenced by velocities of its neighbors. Nevertheless, since velocity at each site has mean $0$ in stationarity, we will be able to show that the contribution from velocities of the particle's neighbors cancels out when averaged over time -- this will be the content of the \emph{one block estimate} proved in the next section.

Note that to prove that the random processes converge indeed to a deterministic process, it is not enough to look only at single path distributions, as explained in Section \ref{sec:stationarity}. Nevertheless, we will show that in the interchange process typically any two particles (in fact almost all of them) behave like independent random walks, which by Lemma \ref{lm:deterministic} will be enough to establish a deterministic limit.

Throughout this and the following sections we will make extensive use of martingales associated to Markov processes (see \cite{kipnis} for a comprehensive treatment of such techniques applied to interacting particle systems). For any Markov process with generator $\Ll$ and a bounded function $F : E \to \R$, where $E$ is the configuration space of the process, the following processes are mean zero martingales (\cite[Lemma A1.5.1]{kipnis})
\begin{align}
& M_t = F(\eta_t) - F(\eta_0) - \int\limits_{0}^{t} \Ll F(\eta_s) \, ds, \label{eq:martingale1}\\
& N_t = M_{t}^{2} - \int\limits_{0}^{t} \left(\Ll F(\eta_s)^2 - 2F(\eta_s)\Ll F(\eta_s)\right) ds \label{eq:martingale2}.
\end{align}
Furthermore, for any $F$ as above the following process is a mean one positive martingale (see discussion following \cite[Lemma A1.7.1]{kipnis})
\begin{equation}\label{eq:martingale-exp}
\mathbb{M}_t = \exp\left\{ F(\eta_t) - F(\eta_0) - \int\limits_{0}^{t} e^{-F(\eta_s)} \Ll e^{F(\eta_s)} \, ds \right\}.
\end{equation}
In the following sections we will also consider the case when $F$ is not necessarily bounded, in which case $M_t$, $N_t$, $\mathbb{M}_t$ are only local martingales.

Our first goal is to prove that with high probability almost all particles move according to their local velocity $v(t,x_{i}, \phi_i)$. Recall that
\[
X_{i}(\eta_{t}) = \frac{1}{N} x_{i}(\eta_{t}), \quad \Phi_{i}(\eta_{t}) = \frac{1}{N} \phi_{i}(\eta_{t})
\]
are respectively the rescaled position and color of the particle with label $i$. Our first goal is to prove the following

\begin{proposition}\label{prop:second-moment}
For any fixed $t \in [0,T]$ and $\varepsilon > 0$ we have in the biased interchange process
\begin{align*}
& \widetilde{\Pp}^{N} \left( \frac{1}{N} \sum\limits_{i=1}^{N}\left| X_{i}(\eta_{t}) - X_{i}(\eta_{0}) - \int\limits_{0}^{t} v(s,x_{i}(\eta_{s}), \phi_{i}(\eta_{s})) \, ds \right| > \varepsilon \right)\to 0  \\
& \widetilde{\Pp}^{N} \left( \frac{1}{N} \sum\limits_{i=1}^{N} \left| \Phi_{i}(\eta_{t}) - \Phi_{i}(\eta_{0}) - \int\limits_{0}^{t} r(s,x_{i}(\eta_{s}), \phi_{i}(\eta_{s})) \, ds \right| > \varepsilon \right	)\to 0
\end{align*}
as $N \to \infty$.
\end{proposition}

As a starting point let us rewrite $X_i(\eta_t)$ in a more useful form. Recall from \eqref{eq:biased-generator} that $\widetilde{\Ll}$ denotes the generator of the biased interchange process. By the formula \eqref{eq:martingale1} applied to $F(\eta_s) = X_{i}(\eta_s)$ we have
\[
X_{i}(\eta_{t}) - X_{i}(\eta_{0}) = M_{t}^{i} + \int\limits_{0}^{t} \widetilde{\Ll} X_{i}(\eta_{s}) \, ds,
\]
where $M_{t}^{i}$ is a mean zero martingale with respect to $\widetilde{\Pp}^{N}$. Recall that $v_x(t,\eta) = v(t,x, \phi_{x}(\eta))$ denotes the velocity of the particle at site $x$ in configuration $\eta$ at time $t$. For simplicity we will also write $v_{x_{i}}(t,\eta) = v(t,x_{i}(\eta), \phi_{i}(\eta))$ for the velocity of the particle with label $i$. We have
\begin{align*}
& \widetilde{\Ll} X_{i}(\eta_{s}) = \frac{1}{N} \widetilde{\Ll} (x_{i}(\eta_{s})) = \frac{1}{2} N^{\alpha-1} \sum\limits_{x=1}^{N-1} \left( 1 + \varepsilon \left[v_{x}(s,\eta_{s}) - v_{x+1}(s,\eta_{s})\right] \right) (x_{i}(\eta^{x, x+1}_{s}) - x_{i}(\eta_{s})) = \\
& = \frac{1}{2} N^{\alpha-1}\varepsilon \Big[ - \left[ v_{x_{i}-1}(s,\eta_{s}) - v_{x_{i}}(s,\eta_{s}) \right] + \left[ v_{x_{i}}(s,\eta_{s}) - v_{x_{i}+1}(s,\eta_{s}) \right]\Big] = \\
& = \frac{1}{2} \left(2 v_{x_{i}}(s,\eta_{s}) - v_{x_{i}-1}(s,\eta_{s}) - v_{x_{i}+1}(s,\eta_{s})\right),
\end{align*}
since the position of the particle $i$ changes by $\pm 1$ depending on whether it makes a swap with its left or right neighbor.

Thus we obtain
\begin{align*}
& X_{i}(\eta_{t}) - X_{i}(\eta_{0}) = M_{t}^{i} + \int\limits_{0}^{t} v_{x_{i}}(s,\eta_{s}) \, ds  + \frac{1}{2} \int\limits_{0}^{t} \left( v_{x_{i}-1}(s,\eta_{s}) + v_{x_{i}+1}(s,\eta_{s}) \right) ds,
\end{align*}
or in other words
\begin{align}\label{eq:xt-x0}
& X_{i}(\eta_{t}) - X_{i}(\eta_{0}) - \int\limits_{0}^{t} v_{x_{i}}(s,\eta_{s}) \, ds =  M_{t}^{i} + \frac{1}{2} \int\limits_{0}^{t} \left( v_{x_{i}-1}(s,\eta_{s}) + v_{x_{i}+1}(s,\eta_{s}) \right) \, ds.
\end{align}
For the sake of proving the first part of Proposition \ref{prop:second-moment} it will be enough to show that
\begin{equation}\label{eq:expectation-to-zero}
\frac{1}{N} \sum\limits_{i=1}^{N} \E \Big( X_{i}(\eta_{t}) - X_{i}(\eta_{0}) - \int\limits_{0}^{t} v_{x_{i}}(\eta_{s}) \, ds \Big)^2 \to 0
\end{equation}
as $N \to \infty$. First we prove that for most particles the martingale term $M_{t}^{i}$ will be small with high probability. Let us define
\[
Q_{s}^{i} = \widetilde{\Ll} X_{i}(\eta_{s})^2 - 2 X_{i}(\eta_{s}) \widetilde{\Ll} X_{i}(\eta_{s}).
\]
By the martingale formula \eqref{eq:martingale2} we have that
\begin{equation}\label{eq:quadratic}
N_{t}^{i} = (M_{t}^{i})^2 - \int\limits_{0}^{t} Q_{s}^{i} \, ds
\end{equation}
is a mean zero martingale. A quick calculation gives
\begin{align*}
& \widetilde{\Ll} X_{i}(\eta_{s})^2  = \\
& = \frac{1}{2} \Big[ \left(v_{x_{i}-1}(s,\eta_{s}) - v_{x_{i}}(s,\eta_{s})\right) \left( \frac{-2 x_{i}(\eta_{s})+1}{N} \right) + \left(v_{x_{i}}(s,\eta_{s}) - v_{x_{i}+1}(s,\eta_{s})\right) \left( \frac{2 x_{i}(\eta_{s})+1}{N} \right)\Big] +  N^{\alpha - 2}
\end{align*}
and
\begin{align*}
& 2 X_{i}(\eta_{s}) \widetilde{\Ll} X_{i}(\eta_{s}) = \frac{x_{i}(\eta_{s})}{N} \left( 2 v_{x_{i}}(s,\eta_{s}) - v_{x_{i}-1}(s,\eta_{s}) - v_{x_{i}+1}(s,\eta_{s}) \right),
\end{align*}
so these two quantities are the same up to terms of order $o(1)$. Thus $Q_{s}^{i} = o(1)$ (uniformly in $s$ and $i$) and, since $\E N_{t}^{i} = 0$, we obtain from \eqref{eq:quadratic} that $\E (M_{t}^{i})^2 = o(1)$ as well.

Incidentally, a similar calculation (only simpler, since it does not involve correlations between adjacent particles) and the martingale argument gives us that for $\Phi_{i}(\eta_{t}) = \frac{1}{N}\phi_{i}(\eta_{t})$ we have
\[
\Phi_{i}(\eta_{t}) - \Phi_{i}(\eta_{0}) - \int\limits_{0}^{t} r(s,x_{i}(\eta_{s}), \phi_{i}(\eta_{s})) \, ds = o(1)
\]
for any fixed particle $i$. This proves the second part of Proposition \ref{prop:second-moment}.

Recalling \eqref{eq:xt-x0} and \eqref{eq:expectation-to-zero}, to finish the proof of the first part of Proposition \ref{prop:second-moment} we only need to show that
\[
\frac{1}{N} \sum\limits_{i=1}^{N} \E \left(Y_{i}^{t}\right)^2 \to 0
\]
as $N \to \infty$, where
\[
Y_{i}^{t} = \int\limits_{0}^{t} \left( v_{x_{i}-1}(s,\eta_{s}) + v_{x_{i}+1}(s,\eta_{s}) \right) ds.
\]
Recall from \eqref{eq:def-of-piecewise-s} that $V^{\beta,\delta}(s,x,\phi)$ was defined in terms of a partition $0 = t_0 < t_1 < \ldots < t_M = T$. We would like to take advantage of the fact that on each interval the dynamics of the biased interchange process is time-homogeneous. Suppose that $t \in [t_l, t_{l+1})$ for some $l \leq M - 1$ and let us write
\[
Y_{i}^{t} = \sum\limits_{k=0}^{l-1} \int\limits_{t_k}^{t_{k+1}} \left( v_{x_{i}-1}(s,\eta_{s}) + v_{x_{i}+1}(s,\eta_{s}) \right) ds + \int\limits_{t_l}^{t} \left( v_{x_{i}-1}(s,\eta_{s}) + v_{x_{i}+1}(s,\eta_{s}) \right) ds.
\]
For any $t \geq 0$ let
\[
Y_{i}^{t,k} = \int\limits_{t_k}^{t_k + t} \left( v_{x_{i}-1}(s,\eta_{s}) + v_{x_{i}+1}(s,\eta_{s}) \right) ds.
\]
Since $M$ is fixed, it is enough to show that for any fixed $k \leq M - 1$ and $t \in [0, t_{k+1} - t_{k}]$ we have
\[
\frac{1}{N} \sum\limits_{i=1}^{N} \E \left(Y_{i}^{t,k}\right)^2 \to 0
\]
as $N \to \infty$.

To keep the notation simple we will prove the desired statement just for $k=0$, with the general case being exactly analogous. Recall that $t_0 = 0$. By definition of the piecewise-constant in time approximation of $V^{\beta,\delta}$, for $s \in [0, t_{1})$ we have $v_{x}(s, \eta_s) = v_{x}(0, \eta_s)$. Let us define $v_{x}(\eta) = v_{x}(0, \eta)$. Fix any $t \in [0, t_{1}] $ and let us look at
\[
\left(Y_{i}^{t,0}\right)^2 = \left( \int\limits_{0}^{t} \left( v_{x_{i}-1}(\eta_{s}) + v_{x_{i}+1}(\eta_{s}) \right) ds \right)^2.
\]
We will have four cross-terms here, it is enough to show that each of them is small in expectation. The argument will be similar in all cases, so we will only present the proof for one of them. Let us focus on
\begin{align*}
& \E \left[\left( \int\limits_{0}^{t} v_{x_{i}-1}(\eta_{s})  \, ds \right) \left( \int\limits_{0}^{t} v_{x_{i}-1}(\eta_{s})  \, ds \right) \right] = \E \int\limits_{0}^{t} \int\limits_{0}^{t} v_{x_{i}-1}(\eta_{u_{1}}) v_{x_{i}-1}(\eta_{u_{2}}) \, du_{1} \, du_{2}.
\end{align*}
For each particle $i$ we are looking at the correlation of the velocity of its left neighbor at time $u_1$ with the velocity of its left neighbor at time $u_{2}$. By averaging over particles $i = 1, \ldots, N$ and using the symmetry between $u_1$ and $u_2$ we can write the contribution to the second moment of $Y_{i}^{t,0}$ as
\begin{align*}
& \frac{2}{N} \sum\limits_{i=1}^{N} \E \int\limits_{0}^{t} \int\limits_{u_{1}}^{t} v_{x_{i}-1}(\eta_{u_{1}}) v_{x_{i}-1}(\eta_{u_{2}}) \, du_{2} \, du_{1} = 2 \int\limits_{0}^{t} \, du_{1} \left( \frac{1}{N} \sum\limits_{i=1}^{N} \E  \int\limits_{u_{1}}^{t} v_{x_{i}-1}(\eta_{u_{1}}) v_{x_{i}-1}(\eta_{u_{2}}) \, du_{2} \right).
\end{align*}
Since the rates $v$ are bounded, it is enough to show that for each fixed $u_{1} \in [0,t]$ the expression inside the bracket is close to $0$ as $N \to \infty$. Let us look at
\[
\frac{1}{N} \sum\limits_{i=1}^{N} \E \int\limits_{u_{1}}^{t} v_{x_{i}-1}(\eta_{u_{1}}) v_{x_{i}-1}(\eta_{u_{2}}) \, du_{2}.
\]
Since the average here depends only on the configuration at time $u_{1}$ and its evolution from that point on (and not otherwise on the trajectory of the process before time $u_{1}$), by stationarity of the biased interchange process it will be the same as
\begin{equation}\label{eq:decorrelation}
\frac{1}{N} \sum\limits_{i=1}^{N} \E \int\limits_{0}^{t - u_1} v_{x_{i}-1}(\eta_{0}) v_{x_{i}-1}(\eta_{s}) \, ds ,
\end{equation}
since the dynamics of the process is time-homogeneous on $[0, t_{1})$.

Thus we have to prove that for a random particle the velocity of its initial left neighbor is uncorrelated (when averaged over time) with the velocity of its current left neighbor. Let us introduce the following setup -- we can rewrite the average above in terms of a sum over sites (for $y = x_{i}(\eta_{s})$) instead over particles
\begin{equation}\label{eq:average-sites}
\frac{1}{N} \sum\limits_{y=1}^{N} \E \int\limits_{0}^{t - u_1} v_{x_{\eta^{-1}_{s}(y)}(\eta_{0})-1}(\eta_{0}) v_{y-1}(\eta_{s}) \, ds
\end{equation}
To analyze this average we introduce the following extension of the biased interchange process. Consider the extended configuration space $\widetilde{E}$ consisting of sequences $\left( (x_i, \phi_i, L_i) \right)_{i=1}^{N}$, with $L_i \in \{1, \ldots, N\}$. Here each particle, in addition to its color $\phi_{i}$, also has an additional color $L_{i}$ in which we keep information about the velocity of its left neighbor at time $0$, that is
\[
L_{i} = v_{x_{i}(\eta_{0})-1}(\eta_{0}).
\]
The dynamics is given by the same generator \eqref{eq:biased-generator} as before, i.e., labels (together with their corresponding colors $\phi_i$ and $L_i$) are exchanged by swaps of adjacent particles, each $\phi_i$ has its own evolution and $L_i$ does not evolve. For a site $x$ let $L_{x}(\eta)$ be the additional color at site $x$ in configuration $\eta$, i.e., $L_x(\eta) = L_{\eta^{-1}(x)}$. We can now treat $\eta$ as a function which assigns to each site $x$ a triple $(\eta^{-1}(x), \phi_{x}, L_{x})$ or simply a pair $(\phi_{x}, L_{x})$, since we are not interested in particles' labels at this point, only in the distribution of colors.

In this setup the average \eqref{eq:average-sites} can be written as
\begin{equation}\label{eq:time-average-with-f}
\frac{1}{N} \sum\limits_{y=1}^{N} \E \int\limits_{0}^{t - u_1} f_{y}(\eta_s) \, ds,
\end{equation}
where $f_{y}(\eta) = L_{y}(\eta) v_{y-1}(\eta)$. Let
\[
\Lambda_{x, l} = \{x-l, x-l + 1, \ldots, x+l\},
\]
denote a box of size $l$ around $x$ (with the convention that the box is truncated if the endpoints $x-l$ or $x+l$ exceed $1$ or $N$, but this will not influence the argument in any substantial way) and let $\widehat{\mu}_{x,l}^{\eta}$ be the empirical distribution of colors in $\Lambda_{x,l}$ in configuration $\eta$, given for any $(L, \phi)$ by
\[
\widehat{\mu}_{x, l}^{\eta} \left(L, \phi \right) = \frac{1}{|\Lambda_{x,l}|} \# \{ z \in \Lambda_{x,l} \, | \, \left(L_z(\eta), \phi_z(\eta) \right) = (L, \phi )\}.
\]
Consider the associated i.i.d. distribution on configurations restricted to $\Lambda_{x,l}$, given by
\[
\mu_{x, l}^{\eta} \left((L_y, \phi_y)_{y=x-l}^{x+l} \right) = \prod\limits_{y=x-l}^{x+l} \widehat{\mu}_{x, l}^{\eta} \left(L_y, \phi_y \right).
\]
In other words, under the measure $\mu_{x, l}^{\eta}$ the probability of seeing a color pair $(L, \phi)$ at site $y \in \Lambda_{x,l}$ is proportional to the number of sites in $\Lambda_{x,l}$ with the color pair $(L, \phi)$, independently for each site.

The \emph{superexponential one block estimate} says that on an event of high probability we can replace $f_{y}(\eta_{s})$ in the time average \eqref{eq:time-average-with-f} by its average $\E_{\mu_{y,l}^{\eta_{s}}} (f)$ with respect to the local i.i.d. distribution over a sufficiently large box. In other words, due to local mixing the distribution of colors in a microscopic box can be approximated by an i.i.d. distribution for large $l$.

\begin{lemma}\label{lm:one-block-superexponential-biased}
Let $U_{x,l}(\eta) = | f_{x}(\eta) - \E_{\mu_{x,l}^{\eta}} (f) |$. For any $t \in [0,t_{1}]$ and $\delta > 0$ we have
\[
\limsup\limits_{l \to \infty} \limsup\limits_{N \to \infty} N^{-\gamma} \log \widetilde{\Pp}^{N} \left( \int\limits_{0}^{t} \frac{1}{N} \sum\limits_{x=1}^{N}  U_{x,l}(\eta_{s}) \, ds  > \delta \right) = - \infty,
\]
where $\gamma = 3-\alpha$.
\end{lemma}

The lemma is proved in the next section. Let us see how it enables us to finish the proof of Proposition \ref{prop:second-moment}. By the one block estimate, in \eqref{eq:time-average-with-f} we can replace
\[
\frac{1}{N} \sum\limits_{y=1}^{N} \int\limits_{0}^{t - u_1} f_{y} (\eta_s) \, ds
\]
by
\begin{equation}\label{eq:time-average-after-one-block}
\frac{1}{N} \sum\limits_{y=1}^{N}  \int\limits_{0}^{t - u_1} \E_{\mu_{y,l}^{\eta_{s}}} f_{y} (\eta_s) \, ds,
\end{equation}
with the difference going to $0$ in expectation as first $N \to \infty$ and then $l \to \infty$, so we only need to show that the latter expression goes to $0$ in the same limit.

Observe that in $f_{y}(\eta) = L_y(\eta)v_{y-1}(\eta) = L_y(\eta)v(y-1, \phi_{y-1}(\eta))$ the colors $\phi_{y-1}$ and $L_y$ depend on different sites, so they are independent under $\mu_{y,l}^{\eta_{s}}$, since the measure is product. Thus in the average above we can simply write
\[
\E_{\mu_{y,l}^{\eta_{s}}} f_{y} (\eta_s) = \E_{\mu_{y,l}^{\tilde{\eta}_{s}}} \left[ L_y(\eta)v_{y-1}(\eta) \right] = \left( \E_{\sigma \sim \mu_{y,l}^{\eta_{s}}}  L_{y}(\sigma) \right) \left( \E_{\sigma \sim \mu_{y,l}^{\eta_{s}}} v_{y-1}(\sigma) \right),
\]
where by a slight abuse of notation we have denoted by $\sigma$ the local configuration of colors in a box $\Lambda_{y,l}$ and considered $L_{y}$, $v_{y-1}$ as functions of $\sigma$. The average \eqref{eq:time-average-after-one-block} now becomes
\[
\frac{1}{N} \sum\limits_{y=1}^{N} \int\limits_{0}^{t - u_1} \left( \E_{\sigma \sim \mu_{y,l}^{\eta_{s}}}  L_{y}(\sigma) \right) \left( \E_{\sigma \sim \mu_{y,l}^{\eta_{s}}} v_{y-1}(\sigma) \right) ds.
\]
Since the distribution of $\eta_{s}$ in the biased interchange process process without the additional colors $L_i$ is stationary, the distribution of the average $\E_{\sigma \sim \mu_{y,l}^{\eta_{s}}} v_{y-1}(\sigma)$ does not depend on $s$. So we only need to show that $\E_{\sigma \sim \mu_{y,l}^{\eta_{0}}} v_{y-1}(\sigma)$ is small, since $L_{y}$ is bounded.

Recall that in stationarity $\phi_{y}$ has uniform distribution, so for any $y$ the expectation of $v_{y-1}(\sigma) = v(0, y-1, \phi_{y-1}(\sigma))$ with respect to $\mu_{y,l}^{\eta_{0}}$ is simply equal to
\[
\frac{1}{2l+1} \sum\limits_{j=1}^{2l+1} v(0, y-1, \phi_{j}),
\]
where $\phi_{j}$ are independent and uniformly distributed on $\{1, \ldots, N\}$. As for each $x$ the random variables $v(0, x, \phi_{j})$ are independent, bounded and have mean $0$ (see Proposition \ref{prop:speeds}), an easy application of Hoeffding's inequality gives that for fixed $y$ the sum above goes to $0$ in probability as $l \to \infty$. This finishes the proof of Proposition \ref{prop:second-moment}.

We can now prove the law of large numbers.

\begin{proof}[Proof of Theorem \ref{th:lln}]

Consider the random particle process $\bar{P}^{N} = (\bar{X}^{N}, \bar{\Phi}^{N})$, obtained by first sampling $\eta = \eta^N$ and then following the trajectory $P_{i}(\eta_t) = (X_i(\eta_t), \Phi_{i}(\eta_t))$ of a randomly chosen particle $i$. We will first show that the (deterministic) distribution $\bar{\nu}^{N}$ converges to $\nu$, the distribution of $P$ (in the metric $d_{\mathcal{W}}^{sup}$).

Let us start by proving that the estimate from Proposition \ref{prop:second-moment} holds not only at each time $t$, but also with the supremum over all times $t \leq T$ under the sum over particles. Consider the process $(A^{N}, B^{N})$ defined as
\begin{align*}
& A_{t}^{N} = X_{i}(\eta_{t}) - X_{i}(\eta_{0}) - \int\limits_{0}^{t} v(s, x_{i}(\eta_{s}), \phi_{i}(\eta_{s})) \, ds, \\
& B_{t}^{N} = \Phi_{i}(\eta_{t}) - \Phi_{i}(\eta_{0}) - \int\limits_{0}^{t} r(s, x_{i}(\eta_{s}), \phi_{i}(\eta_{s})) \, ds,
\end{align*}
where $i$ is a random particle and $\eta = \eta^{N}$ comes from the biased interchange process. Proposition \ref{prop:second-moment} implies that all finite-dimensional marginals of $(A^{N},B^{N})$ converge to $0$. To obtain convergence to $0$ for the whole process in the supremum norm we only need to check tightness in the Skorokhod topology (which will imply convergence in the supremum norm, since the limiting process is continuous). We will use the following stopping time criterion (\cite[Proposition 4.1.6]{kipnis}). Let $Y^{N}$ be a family of stochastic processes with sample paths in $\DDD$ such that for each time $t \in [0,T]$ the marginal distribution of $Y^{N}_{t}$ is tight. If for every $\varepsilon > 0$ we have
\begin{equation}\label{eq:tightness-aldous}
\lim_{\gamma \to 0} \limsup_{N \to \infty} \, \sup_{\substack{\tau \\ \theta \leq \gamma}} \Pp\left( \norm{Y^{N}_{\tau + \theta} - Y^{N}_{\tau}} > \varepsilon\right) = 0, 
\end{equation}
where the supremum is over all stopping times $\tau$ bounded by $T$, then the family $Y^{N}$ is tight. Here $\norm{\cdot}$ denotes the Euclidean distance on $[0,1]^2$ and for simplicity we write $\tau + \theta$ instead of $(\tau + \theta) \wedge T$. Let $\tau$ be any stopping time bounded by $T$. We have from formula \ref{eq:xt-x0}
\[
A_{\tau + \theta}^{N} - A_{\tau}^{N} =  M^{i}_{\tau + \theta} - M^{i}_{\tau} - \frac{1}{2} \int\limits_{\tau}^{\tau + \theta} \left[ v(s, x_{i}(\eta_{s})-1, \phi_{x_{i}(\eta_{s}) - 1}(\eta_{s})) + v(s, x_{i}(\eta_{s})+1, \phi_{x_{i}(\eta_{s}) + 1}(\eta_{s})) \right] ds.
\]
Since $v(\cdot, \cdot,\cdot)$ is bounded, the integral is bounded by $C \theta$ for some constant $C > 0$, regardless of $\tau$, so goes to $0$ as $\theta \to 0$ (deterministically and for every $i$). Thus it only remains to bound the martingale term. As $\tau$ is a stopping time, by formula \eqref{eq:quadratic} we have for each $i$
\[
\E \left[ \left(M^{i}_{\tau + \theta}\right)^2 - \left(M^{i}_{\tau}\right)^2\right] = \E \int\limits_{\tau}^{\tau + \theta} Q_{s} \, ds.
\]
As in the calculation of $\E (M^{i}_{t})^2$ following \eqref{eq:quadratic} we have that for fixed $\theta$ the right hand side is $o(1)$ as $N \to \infty$. Since $M^{i}_{t}$ is bounded, we obtain $ \E \left| M^{i}_{\tau + \theta} - M^{i}_{\tau}\right| \to 0$ as $N \to \infty$, for any $\theta$ and $i$ (independently of $\tau$). The calculation for $B^{N}$ is analogous.

This shows that the family $(A^{N}, B^{N})$ satisfies the tightness criterion \eqref{eq:tightness-aldous}. In particular it converges to $0$ in the supremum norm as $N \to \infty$. Thus for any $\varepsilon > 0$ we have
\begin{align}
& \widetilde{\Pp}^{N} \left( \frac{1}{N} \sum\limits_{i=1}^{N} \sup_{0 \leq t \leq T} \left| X_{i}(\eta_{t}) - X_{i}(\eta_{0}) - \int\limits_{0}^{t} v(s, x_{i}(\eta_{s}), \phi_{i}(\eta_{s})) \, ds \right| > \varepsilon \right)\to 0,\label{eq:sup-convergence1}  \\
& \widetilde{\Pp}^{N} \left( \frac{1}{N} \sum\limits_{i=1}^{N} \sup_{0 \leq t \leq T} \left| \Phi_{i}(\eta_{t}) - \Phi_{i}(\eta_{0}) - \int\limits_{0}^{t} r(s,x_{i}(\eta_{s}), \phi_{i}(\eta_{s})) \, ds \right| > \varepsilon \right	)\to 0 \label{eq:sup-convergence2}
\end{align}
as $N \to \infty$.

Now we can prove that $\bar{\nu}^N$ converges to $\nu$. Recalling the definition of the Wasserstein distance $d_{\mathcal{W}}^{sup}$, it is enough to construct for each $N$ a coupling $(\bar{P}^{N}, P)$ such that
\[
\E \norm{ \bar{P}^{N} - P }_{sup} \to 0
\]
as $N \to \infty$.

Let us couple these two processes in the following way: first we let $\bar{P}^{N} = \left( \left(\bar{X}^{N}_{t}, \bar{\Phi}^{N}_{t}\right), 0 \leq t \leq T\right)$ be a path sampled according to $\bar{\nu}^{\eta^N}$, starting at $(\bar{X}^{N}_{0}, \bar{\Phi}^{N}_{0})$ (whose distribution is uniform on $\left\{\frac{1}{N}, \ldots, 1\right\} \times \left\{\frac{1}{N}, \ldots, 1\right\}$). We then take $P(t) = (X(t), \Phi(t))$ to be the solution of the ODE \eqref{eq:ode-general} started from an initial condition $(X(0), \Phi(0))$ chosen uniformly at random from $\left[\bar{X}^{N}_{0} - \frac{1}{N}, \bar{X}^{N}_{0} \right] \times \left[\bar{\Phi}^{N}_{0} - \frac{1}{N}, \bar{\Phi}^{N}_{0} \right]$ (so the two processes start close to each other). Because the initial condition is distributed uniformly on $[0,1]^2$, the path $P = \left( P(t), 0 \leq t \leq T \right)$ will be distributed according to $\nu$.

Since $P(t) = (X(t), \Phi(t))$ is the solution of \eqref{eq:ode-general}, we have at each time $t \leq T$
\begin{align*}
& X(t) - X(0) = \int\limits_{0}^{t} V(s,X(s), \Phi(s)) \, ds,  \\
& \Phi(t) - \Phi(0) = \int\limits_{0}^{t} R(s,X(s), \Phi(s)) \, ds. 
\end{align*}
Bounds \eqref{eq:sup-convergence1}, \eqref{eq:sup-convergence2} imply that for all times $t \leq T$ we have
\begin{align*}
& \bar{X}^{N}(t) - \bar{X}^{N}(0) = \int\limits_{0}^{t} v\left(s, N \bar{X}^{N}(s), N \bar{\Phi}^{N}(s) \right) \, ds + \varepsilon_{t}^{1},  \\
& \bar{\Phi}^{N}(t) - \bar{\Phi}^{N}(0) = \int\limits_{0}^{t} r\left(s, N \bar{X}^{N}(s), N \bar{\Phi}^{N}(s) \right) \, ds + \varepsilon_{t}^{2},
\end{align*}
with $\varepsilon_{t}^{1}$, $\varepsilon_{t}^{2}$ satisfying $\sup\limits_{ 0 \leq t \leq T} |\varepsilon_{t}^{i}| \to 0$ in probability as $N \to \infty$. Recalling from \eqref{eq:s-r-approx} that $v(\cdot, \cdot, \cdot), r(\cdot, \cdot, \cdot)$ are approximately equal to $V(\cdot, \cdot, \cdot), R(\cdot, \cdot, \cdot)$ after rescaling of the arguments, we obtain
\begin{align*}
& \bar{X}^{N}(t) - \bar{X}^{N}(0) = \int\limits_{0}^{t} V\left( s, \bar{X}^{N}(s), \bar{\Phi}^{N}(s) \right) ds + o(1),  \\
& \bar{\Phi}^{N}(t) - \bar{\Phi}^{N}(0) = \int\limits_{0}^{t} R\left( s, \bar{X}^{N}(s), \bar{\Phi}^{N}(s) \right) ds + o(1),
\end{align*}
with the $o(1)$ terms going to $0$ in probability (in the supremum norm over $t$) as $N \to \infty$.

Thus $(\bar{X}^{N}, \bar{\Phi}^{N})$ approximately satisfies the same ODE as $(X, \Phi)$ and an application of Gr\"onwall's inequality gives that for any $\varepsilon > 0$ with probability approaching $1$ as $N \to \infty$ we have
\[
\norm{ \bar{P}^{N} - P}_{sup} \leq C \max\{ |\bar{X}^{N}(0) - X(0)| + \varepsilon, |\bar{\Phi}^{N}(0) - \Phi(0)| + \varepsilon \}  e^{KT} 
\]
for some $C>0$, where $K>0$ depends only on the Lipschitz constants of $V$ an $R$.

By definition of the processes $\bar{P}^{N}$ and $P$ the initial conditions $\bar{X}^{N}(0)$, $X(0)$ and $\bar{\Phi}^{N}(0)$, $\Phi(0)$ differ by at most $\frac{1}{N}$, which implies that $\E \norm{\bar{P}^{N} - P}_{sup} \to 0$ as $N \to \infty$. Thus the distribution $\bar{\nu}^{\eta^N}$ of the random particle process $\bar{P}^{N}$ converges to $\nu$ in the $d_{\mathcal{W}}^{sup}$ metric as desired.

Now we can show that the random measures $\nu^{\eta^{N}}$ converge in distribution to the deterministic measure $\nu$. By the characterization of tightness for random measures (see, e.g., \cite[Theorem 23.15]{kallenberg}) the family $\nu^{\eta^{N}}$ will be tight, as a family of $\MM(\DDD)$-valued random variables, if for any $\varepsilon > 0$ there exists a compact set $K \subseteq \DDD$ such that $\limsup\limits_{N \to \infty} \E \left( \nu^{\eta^{N}}(K) \right) \geq 1 - \varepsilon$, or, more simply put, $\limsup\limits_{N \to \infty} \widetilde{\Pp}^{N} \left( P^{\eta^{N}} \in K \right) \geq 1 - \varepsilon$. Exactly the same calculation as for the processes $(A^N, B^N)$ before shows the processes $P^{\eta^{N}}$ satisfy the tightness criterion \eqref{eq:tightness-aldous}, which guarantess the existence of desired compact sets $K$ and in turn tightness of $\nu^{\eta^{N}}$.

Now to finish the proof we only need to show uniqueness of subsequential limits for the family $\nu^{\eta^{N}}$. Since any such (possibly random) limit must have the associated random particle process distributed according to $\nu$, it is enough to show that the limit is deterministic.

Consider an outcome of $\nu^{\eta^{N}}$, which is a measure from $\MM(\DDD)$, and sample independently two paths $P_{1}^{N}, P_{2}^{N}$ from it. This corresponds to sampling $\eta^{N}$ according to the biased interchange process, then choosing uniformly at random a pair of particles $i,j$ (possibly with $i=j$, but this event has vanishing probability) and following their trajectories in $\eta^{N}$. By the already established convergence $\bar{\nu}^{N} \to \nu$ in $\MM(\DDD)$, each path $P_{1}^{N}$ and $P_{2}^{N}$ separately has distribution converging to $\nu$. Moreover, due to stationarity of $\eta^{N}$ the initial colors $\phi_i(\eta^{N}_{0})$, $\phi_j(\eta^{N}_{0})$ of any two particles $i,j$ are chosen uniformly at random, in particular they are independent for $i \neq j$. Thus the joint distribution of $(P_{1}^{N},P_{2}^{N})$ converges to the distribution of two independent paths sampled from $\nu$, as a path $P$ sampled from $\nu$ is uniquely determined by its initial conditions. Since we already have tightness, applying Lemma \ref{lm:deterministic} gives that any limit of a subsequence has to be deterministic, which finishes the proof.
\end{proof}

\end{section}

\begin{section}{One block estimate}\label{sec:one-block}

In this section we prove the one block estimate of Lemma \ref{lm:one-block-superexponential-biased}, needed for the proof of Theorem \ref{th:lln}. Since another, simpler variant of this estimate will also be needed for the proof of the large deviation upper bound (Lemma \ref{lm:upper-bound-one-block}), we prove the result in generality suited for both of these applications.

Let us fix a continuous function $w : [0,1]^2 \to \R$ and let $I^N_{w} = \left\{ w\left( \frac{i}{N}, \frac{j}{N} \right) \right\}_{i,j=1}^{N}$. Let $I^N = \left\{ \frac{1}{N}, \ldots, 1 \right\}$. Consider the interchange process on an extended configuration space $E'$ in which each particle in addition to its label $i$ has two colors $(a_i, \phi_i)$, with $a_i \in I^N_{w}$, $\phi_i \in I^N$. The dynamics is given by the usual generator $\Ll$ -- adjacent particles are making swaps at rate $\frac{1}{2}N^{\alpha}$ and the colors $a_i, \phi_i$ of the particle $i$ do not evolve in time. Since the one block estimate concerns only the distribution of colors, from now on we ignore the labels of the particles altogether. Similarly as before we use the notation $a_x = a_x (\eta), \phi_x = \phi_x (\eta)$ to denote the colors of the particle at site $x$ in configuration $\eta$. The configuration at time $s$ is denoted by $\eta_s$.

Consider a continuous function $g : [0,1] \to [-1,1]$ and for $\eta \in E'$ let $h_x(\eta) = a_x(\eta) b_{x-1}(\eta)$, where $b_x(\eta) = g(\phi_{x}(\eta))$ or $b_x (\eta) = a_x (\eta)$. As in the previous section let $\Lambda_{x, l} = \{x-l, x-l + 1, \ldots, x+l\}$ denote the box of size $l$ around $x$ (with an appropriate truncation if the endpoints $x-l$ or $x+l$ exceed $1$ or $N$, which we neglect in the notation from now on) and let $\widehat{\mu}_{x, l}^{\eta}$ be the empirical distribution of colors in $\Lambda_{x,l}$ in configuration $\eta$, given for any $(\alpha, \varphi) \in I^N_{w} \times I^N$ by
\[
\widehat{\mu}_{x, l}^{\eta} \left(\alpha, \varphi \right) = \frac{1}{|\Lambda_{x,l}|} \# \{ z \in \Lambda_{x,l} \, | \, \left(a_z(\eta), \phi_z(\eta) \right) = (\alpha, \varphi )\}.
\]
Consider the associated i.i.d. distribution on configurations restricted to $\Lambda_{x,l}$, given for $(\alpha_y, \varphi_y)_{y=x-l}^{x+l} \in \left(I^N_{w} \times I^N \right)^{2l+1}$ by
\[
\mu_{x, l}^{\eta} \left((\alpha_y, \varphi_y)_{y=x-l}^{x+l} \right) = \prod\limits_{y=x-l}^{x+l} \widehat{\mu}_{x, l}^{\eta} \left(\alpha_y, \varphi_y \right).
\]
Since $h_x$ depends on $\eta$ only through the colors at $x$ and $x-1$, we will slightly abuse notation by writing $\E_{\mu_{x,l}^{\eta}} (h_x)$ for the expectation of $h_x$ with respect to $\mu_{x,l}^{\eta}$.

Let $\psi : [0,1] \to \R$ be a continuous function and let $U^{N}_{x,l}(\eta) = \psi(x) \left( h_{x}(\eta) - \E_{\mu_{x,l}^{\eta}} (h_x) \right)$. We define
\[
U^{N}_{l}(\eta) = \frac{1}{N} \sum\limits_{x=1}^{N} |U^{N}_{x,l}(\eta)|.
\]
Let $\mu$ denote the uniform distribution on $E'$. Note that the dynamics given by $\Ll$ is reversible with respect to $\mu$ and the associated Dirichlet form is given by
\[
D^{N}(f) = \frac{1}{4} N^{\alpha} \int \sum\limits_{x=1}^{N-1} \left(\sqrt{f(\eta^{x, x+1})} - \sqrt{f(\eta)}\right)^2 \, d\mu(\eta)
\]
for any $f : E' \to [0, \infty)$.
\begin{lemma}\label{lm:one-block-superexponential-general}
With $\mu$ denoting the uniform distribution on $E'$, we have for any $C_0 > 0$
\[
\limsup\limits_{l \to \infty} \limsup\limits_{N \to \infty} \sup\limits_{\substack{f \\  D^{N}(f) \leq C_0 N^\gamma}} \int U^{N}_{l}(\eta)f(\eta) \, d\mu(\eta) = 0,
\]
where $\gamma = 3-\alpha$ and the supremum is over all densities $f$ with respect to $\mu$ such that $D^{N}(f) \leq C_0 N^\gamma$.
\end{lemma}

\begin{proof}
Let us decompose $a_x = a_x(\eta)$ and $b_x = b_x(\eta)$ into their positive and negative parts, $a_x = a_x^{+} - a_x^{-}$, $b_x = b_x^{+} - b_x^{-}$. Since
\[
h_x = a_x b_{x-1} = a_x^{+} b_{x-1}^{+} - a_x^{+} b_{x-1}^{-} - a_x^{-} b_{x-1}^{+} + a_x^{-} b_{x-1}^{-},
\]
by the triangle inequality it is enough to prove the lemma with $h_x$ replaced by one of the terms in the sum above, say, $a_x^{+} b_{x-1}^{+}$. Let $K = \max \{ 1, \norm{w}_{\infty} \}$ and let us write
\begin{align*}
& a_x^{+}(\eta) = \int\limits_{0}^{K} \id_{\{ a_x(\eta) > \lambda \}} \, d\lambda, \\
& b_x^{+}(\eta) = \int\limits_{0}^{K} \id_{\{ b_x(\eta) > \theta \}} \, d\theta.
\end{align*}
We have
\begin{align*}
& \frac{1}{N} \sum\limits_{x=1}^{N}\left| \psi(x) \left( a_x^{+}(\eta) b_{x-1}^{+}(\eta) - \E_{\mu_{x,l}^{\eta}} \left[ a_x^{+}(\eta) b_{x-1}^{+}(\eta) \right] \right) \right| \leq \\
& \leq \int\limits_{0}^{K} \int\limits_{0}^{K} \frac{1}{N} \sum\limits_{x=1}^{N} \left| \psi(x) \left( \id_{\{ a_x(\eta) > \lambda \}} \id_{\{ b_{x-1}(\eta) > \theta \}} - \E_{\mu_{x,l}^{\eta}} \left[ \id_{\{ a_x(\eta) > \lambda \}} \id_{\{ b_{x-1}(\eta) > \theta \}} \right] \right)\right| d\lambda \, d\theta,
\end{align*}
where the inequality comes from pulling the integrals over $\lambda$ and $\theta$ outside the absolute value. Let us denote the expression under the integrals on the right hand side by $U^{N}_{l, \lambda, \theta}$. Since it is nonnegative and bounded, we can write
\[
\sup\limits_{\substack{f}} \int \left( \int\limits_{0}^{K} \int\limits_{0}^{K} U^{N}_{l, \lambda, \theta}(\eta) \, d\lambda \, d\theta \right) f(\eta) \, d\mu(\eta) \leq \int\limits_{0}^{K} \int\limits_{0}^{K} \left(\sup\limits_{\substack{f \\  }} \int U^{N}_{l, \lambda, \theta}(\eta)   f(\eta) \, d\mu(\eta) \right)\, d\lambda \, d\theta,
\]
where the supremum is over all densities $f$ satisfying $D^{N}(f) \leq C_0 N^\gamma$. By the same token, when taking the $\limsup$ first over $N$ and then over $l$, we can bound the resulting limit from above by one with the integral over $\lambda$ and $\theta$ outside the $\limsup$. Thus we see that it is enough to prove for fixed $\lambda, \theta \in [0, K]$
\[
\limsup\limits_{l \to \infty} \limsup\limits_{N \to \infty} \sup\limits_{\substack{f \\  D^{N}(f) \leq C_0 N^\gamma}} \int U^{N}_{l,\lambda, \theta}(\eta)f(\eta) \, d\mu(\eta) = 0.
\]
Since
\[
U^{N}_{l,\lambda, \theta}(\eta) = \frac{1}{N} \sum\limits_{x=1}^{N} \left| \psi(x) \left( \id_{\{ a_x(\eta) > \lambda \}} \id_{\{ b_{x-1}(\eta) > \theta \}} - \E_{\mu_{x,l}^{\eta}} \left[ \id_{\{ a_x(\eta) > \lambda \}} \id_{\{ b_{x-1}(\eta) > \theta \}} \right] \right)\right|,
\]
we have reduced the problem to proving the one block estimate for the interchange process in which each particle has only four possible colors, corresponding to the possible values of the pair $(\id_{\{ a_i(\eta) > \lambda \}}, \id_{\{ b_i(\eta) > \theta \}})$. This in turn follows by essentially the same argument as for the simple exclusion process, which can be thought of as interchange process with just two colors (see e.g., \cite[Lemma 5.3.1]{kipnis}). Since the argument is by now standard and used in several places in the literature (see e.g., \cite{Fritz2004} for the case of three possible colors), let us only explain that the bound on the Dirichlet form under the supremum is of the right order. The argument for the simple exclusion process goes through (see the remark following the proof of \cite[Lemma 5.4.2]{kipnis}) if we assume that the Dirichlet form corresponding to the generator without time scaling is $o(N)$ and the process is speeded up by $N^2$. In our case the generator $\Ll$ has a scaling factor of $N^{\alpha}$, so if $N^{-\alpha} D^{N}(f)$ is the Dirichlet form corresponding to the process without time scaling, then our bound on this Dirichlet form is $\leq C_0 N^{\gamma - \alpha} = C_{0} N^{3 - 2\alpha}$. Since $\alpha \in (1,2)$, this is $o(N)$, which agrees with the assumptions for the simple exclusion process. 
\end{proof}

\begin{lemma}\label{lm:one-block-superexponential-probability}
Let $\Pp^N$ denote the law of the interchange process on $E'$ with an arbitrary initial distribution. With the notation as above we have for any $t \geq 0$ and $\delta > 0$
\[
\limsup\limits_{l \to \infty} \limsup\limits_{N \to \infty} N^{-\gamma} \log \Pp^{N} \left( \int\limits_{0}^{t} U^{N}_{l}(\eta_s) \, ds  > \delta \right) = - \infty.
\]
\end{lemma}

\begin{proof}
Let $\mu_{0}$ be an arbitrary initial distribution. Let $\Pp_{0, \mu}^{N}$, resp. $\Pp_{0, \mu_{0}}^{N}$, denote the distribution of the process started from $\mu$, resp $\mu_{0}$, and let $\E_{\mu}$, resp. $\E_{\mu_{0}}$, denote the corresponding expectation.

By Chebyshev's inequality we have for any $c > 0$
\begin{equation}\label{eq:chebyshev}
\Pp_{0, \mu_{0}}^{N} \left( \int\limits_{0}^{t} U^{N}_{l}(\eta_s) \, ds > \delta \right) \leq e^{-cN^{\gamma}} \E_{\mu_{0}} \exp\left\{c N^{\gamma} \int\limits_{0}^{t} U^{N}_{l}(\eta_s) \, ds\right\}.
\end{equation}
We also have
\begin{align*}
 \E_{\mu_{0}} \exp\left\{c N^{\gamma} \int\limits_{0}^{t} U^{N}_{l}(\eta_s) \, ds\right\} & = \E_{\mu} \left[ \frac{\dd \Pp_{0,\mu_{0}}^{N}}{\dd \Pp_{0,\mu}^{N}} (t) \exp\left\{c N^{\gamma} \int\limits_{0}^{t} U^{N}_{l}(\eta_s) \, ds\right\} \right] \\
& \leq \norm{\frac{\dd \Pp_{0,\mu_{0}}^{N}}{\dd \Pp_{0,\mu}^{N}}}_{\infty} \E_{\mu} \exp\left\{c N^{\gamma} \int\limits_{0}^{t} U^{N}_{l}(\eta_s) \, ds\right\}.
\end{align*}
Let $M = |I^{N}_{w}|$. Since $M \leq N^2$ and under $\mu$ each initial configuration has probability $(MN)^N$ = $e^{o(N^{\gamma})}$, the supremum norm of the Radon-Nikodym derivative above is $e^{o(N^\gamma)}$ as well, so to prove \eqref{eq:chebyshev} it is in fact enough to show that for any $c > 0$
\begin{equation}\label{eq:exp-moment}
\limsup\limits_{l \to \infty} \limsup\limits_{N \to \infty} N^{-\gamma} \log \exp\left\{c N^{\gamma} \int\limits_{0}^{t} U^{N}_{l}(\eta_s) \, ds\right\} \leq 0
\end{equation}
and then take $c \to \infty$.

An application of Feynman-Kac formula to the semigroup generated by $\Ll$ shows (see e.g., \cite[Theorem 10.3.1 and Section A1.7]{kipnis}) that to obtain \eqref{eq:exp-moment} it is sufficient to prove for any $c > 0$
\[
\limsup\limits_{l \to \infty} \limsup\limits_{N \to \infty} \sup\limits_{f} \left\{ \int c U^{N}_{l}(\eta)f(\eta) \, d\mu(\eta) - N^{-\gamma} D^{N}(f) \right\} \leq 0,
\]
where the supremum is taken over all densities with respect to $\mu$. Since $U^{N}_{l}$ is bounded by a constant $C > 0$ depending only on $\psi$ and $g$, the expression under the supremum becomes negative if $D^{N}(f) > c C N^{\gamma}$. Thus it is enough to show that for any constant $C_0 > 0$ we have
\[
\limsup\limits_{l \to \infty} \limsup\limits_{N \to \infty} \sup\limits_{\substack{f \\  D^{N}(f) \leq C_0 N^\gamma}} \int U^{N}_{l}(\eta)f(\eta) \, d\mu(\eta) \leq 0,
\]
which exactly the statement of Lemma \ref{lm:one-block-superexponential-general}.
\end{proof}

This estimate will be enough for application in the proof of Lemma \ref{lm:upper-bound-one-block}. As for the proof of Lemma \ref{lm:one-block-superexponential-biased}, we will first show that the one block estimate holds for the unbiased process with color evolution, but with all rates equal to $1$, i.e., the process with state space $E'$ and the generator
\begin{align*}
& (\Ll_{0} f)(\eta) = \frac{1}{2} N^{\alpha} \sum\limits_{x=1}^{N-1} (f(\eta^{x, x+1}) - f(\eta)) + \frac{1}{2} N^{\alpha} \sum\limits_{x=1}^{N} \left[ (f(\eta^{x,+}) - f(\eta)) + (f(\eta^{x,-}) - f(\eta)) \right].
\end{align*}
Here as usual $\eta^{x, \pm}$ denotes the configuration obtained from $\eta$ by changing the color $\phi_x$ of the particle at site $x$ to $\phi_x \pm 1$ (note that the colors $a_i$ do not evolve in time here). We will then transfer the result to the biased process by estimating its Radon-Nikodym derivative.

\begin{lemma}\label{lm:one-block-superexponential-unbiased}
Let $\Pp_{0}^{N}$ be the law of the unbiased process with rates $1$ described above (with an arbitrary initial distribution). With the notation from Lemma \ref{lm:one-block-superexponential-probability}, we have for any $t \geq 0$ and $\delta > 0$
\[
\limsup\limits_{l \to \infty} \limsup\limits_{N \to \infty} N^{-\gamma} \log \Pp_{0}^{N} \left( \int\limits_{0}^{t} U^{N}_{l}(\eta_s) \, ds  > \delta \right) = - \infty.
\]
\end{lemma}

\begin{proof}
Let us write $\Ll_{0} = \Ll + \Ll_{c}$, where $\Ll$ is the first term in the definition of $\Ll_{0}$ and $\Ll_{c}$ is the second term. The dynamics induced by $\Ll$ and by $\Ll_{c}$ is reversible with respect to $\mu$, so the Dirichlet forms associated respectively to $\Ll_{c}$ and $\Ll_{0}$ can be written as
\begin{align*}
& D^{N}_{c}(f) = \frac{1}{4} N^{\alpha} \int \sum\limits_{x=1}^{N} \left[ \left( \sqrt{f(\eta^{x,+})} - \sqrt{f(\eta)} \right)^2 + \left( \sqrt{f(\eta^{x,-})} - \sqrt{f(\eta)} \right)^2 \right]\, d\mu(\eta), \\
& D^{N}_{0}(f) = D^{N}(f) + D^{N}_{c}(f). 
\end{align*}
By repeating the argument from the proof of Lemma \ref{lm:one-block-superexponential-probability} with the generator $\Ll_0$ instead of $\Ll$ we obtain that it is enough to prove that for any $c > 0$
\[
\limsup\limits_{l \to \infty} \limsup\limits_{N \to \infty} \sup\limits_{f} \left\{ \int c U^{N}_{l}(\eta)f(\eta) \, d\mu(\eta) - N^{-\gamma} D^{N}_{0}(f) \right\} \leq 0,
\]
where the supremum is taken over all densities with respect to $\mu$.

Now observe that since $D^{N}_{c}(f) \geq 0$ for any nonnegative $f$, it is in fact enough to prove the statement above with $D^{N}_{0}(f)$ replaced by $D^{N}(f)$. Thus we have eliminated color evolution and the conclusion follows as in the proof of Lemma \ref{lm:one-block-superexponential-probability}.
\end{proof}

We can now prove the superexponential estimate for the biased process.

\begin{proof}[Proof of Lemma \ref{lm:one-block-superexponential-biased}]
Recall that $f_x(\eta) = L_{x}(\eta) v_{x-1}(\eta)$. Since we can uniformly approximate $v(0,x,\phi)$ by finite sums of terms which are product in $x$ and $\phi$, by using the triangle inequality we can without loss of generality assume that $v_{x}(\eta) = \psi(x)g(\phi_x)$ for some continuous functions $\psi : [0,1] \to \R, g : [0,1] \to [-1,1]$. Applying Lemma \ref{lm:one-block-superexponential-unbiased} with $w(x, \phi)= v(0,x,\phi)$, $a_i = L_i$ and $h_x = L_x g(\phi_{x-1})$ provides us with the superexponential estimate for the process $\Pp_{0}^{N}$. To transfer the estimate to the biased process $\widetilde{\Pp}^{N}$ we will need to estimate the Radon-Nikodym derivative of the two processes.

If $\Pp$ is a Markov process with jump rates $\lambda(x)p(x,y)$ and $\widetilde{\Pp}$ is another process on the same state space with rates $\widetilde{\lambda}(x) \widetilde{p}(x,y)$, the Radon-Nikodym derivative up to time $t$ is given by (see, e.g., \cite[Proposition A1.2.6]{kipnis})
\begin{equation}\label{eq:radon-nikodym-sum-rates}
\frac{\dd \widetilde{\Pp}}{\dd \Pp} (t)  = \exp \left\{ - \int\limits_{0}^{t} \left( \widetilde{\lambda}(X_{s}) - \lambda(X_{s}) \right) ds + \sum\limits_{s \leq t}\log \frac{\widetilde{\lambda}(X_{s-})\widetilde{p}(X_{s-},X_{s})}{\lambda(X_{s-})p(X_{s-},X_{s})}\right\},
\end{equation}
where the sum is over jump times $s \leq t$.

Let us look at $\frac{\dd \widetilde{\Pp}^{N}}{\dd \Pp_{0}^{N}}$. By the form \eqref{eq:biased-generator} of the generator of $\widetilde{\Pp}^{N}$ the sum of outgoing rates for any $\eta$ is equal to
\[
\frac{1}{2}N^{\alpha} \left(\sum\limits_{x=1}^{N-1} \left[ 1 + \varepsilon \left(v_{x}(\eta) - v_{x+1}(\eta) \right) \right] + \sum\limits_{x=1}^{N} \left[ 1 + \varepsilon r_{x}(\eta) \right]  + \sum\limits_{x=1}^{N} \left[ 1 - \varepsilon r_{x}(\eta) \right]\right).
\]
Since the sum of $\varepsilon(v_{x} - v_{x+1})$ telescopes, the rates $v_x$ are $0$ at the boundaries $x=1, N$ and rates $r_x$ for the color change cancel out, the intensities $\widetilde{\lambda}$ and $\lambda$ cancel out as well. The Radon-Nikodym derivative takes the form
\begin{align}\label{eq:radon-nikodym-sum-jumps}
\frac{\dd \widetilde{\Pp}^{N}}{\dd \Pp_{0}^{N}}(t)  = &  \exp \Big\{ \sum\limits_{s \leq t} \log \left( 1 + \varepsilon \left[v(x_{j_{s}}, \phi_{x_{j_{s}}}(\eta_{s})) - v(x_{j_{s}}+1, \phi_{x_{j_{s}}+1}(\eta_{s}))\right] \right) + \\
& + \sum\limits_{s_{+} \leq t} \log \left( 1 + \varepsilon \left[r(x_{j_{s_{+}}}, \phi_{x_{j_{s_{+}}}}(\eta_{s_{+}}))\right] \right) + \sum\limits_{s_{-} \leq t} \log \left( 1 - \varepsilon \left[r(x_{j_{s_{-}}}, \phi_{x_{j_{s_{-}}}}(\eta_{s_{-}}))\right] \right) \Big\},\nonumber 
\end{align}
where $j_{s}$ is the label of the particle which makes a swap at time $s$ and $j_{s_{\pm}}$ is the label of the particle that changes its color by $\pm 1$ at time $s_{\pm}$.

To simplify this formula we will use the fact that empirical currents across edges can be approximated by their averages, modulo a small martingale. More precisely, let us denote for simplicity
\[
(\nabla_{x} v)(\eta) = v_x (\eta) - v_{x+1} (\eta).
\]
We will sometimes use this notation with $x=N$, in which case we assume $(\nabla_{x} v)(\eta) = 0$. For brevity of notation whenever sums involving both $r_x$ and $-r_x$ appear, we will write them as one term with a $\pm$ sign, that is, with $\sum\limits_{x}(1 \pm \varepsilon r_x)$ serving as a shorthand for $\sum\limits_{x}(1 + \varepsilon r_x) + \sum\limits_{x}(1 - \varepsilon r_x)$ and so on.

We introduce the following extension of the dynamics under $\Pp_{0}^{N}$ -- for any functions $h(x,\eta)$, $h^{\pm}(x,\eta)$, $x \in \{1, \ldots, N\}$ consider the extended state space $E'$, consisting of pairs $(\eta, J)$, $J \in \mathbb{R}$, and the generator $\Ll'$ acting by
\begin{align*}
 (\Ll' f)(\eta, J) = & \frac{1}{2} N^{\alpha} \Big[ \sum\limits_{x=1}^{N-1} \left( f(\eta^{x,x+1}, J + h(x, \eta)) - f(\eta, J) \right) +\\
& + \sum\limits_{x=1}^{N} \left( f(\eta^{x,+}, J + h^{+}(x, \eta)) - f(\eta, J) \right) + \sum\limits_{x=1}^{N} \left( f(\eta^{x,-}, J + h^{-}(x, \eta)) - f(\eta, J) \right) \Big].
\end{align*}
In other words, in the evolution of the extended configuration $(\eta_t, J_t)$ each time the process makes a jump, $J_t$ is increased by $h(x, \eta_t)$, $h^{+}(x, \eta_t)$ or $h^{-}(x, \eta_t)$, depending on the type of the jump (swap or color change). Now if we take
\begin{align*}
& h(x,\eta) = \log \left[ 1 + \varepsilon (\nabla_{x} v)(\eta)\right], \\
& h^{\pm}(x,\eta) = \log \left[ 1 \pm \varepsilon r_{x}(\eta)\right],
\end{align*}
we see that $J_t$ is simply equal to the sum over jumps appearing in the exponent in \eqref{eq:radon-nikodym-sum-jumps}. Thus to bound the Radon-Nikodym derivative we only need to bound $J_t$.

This is done by use of an exponential martingale -- for any $\lambda > 0$ the following process
\[
Z_{t} = \exp\left\{ \lambda J_{t}- \int\limits_{0}^{t} e^{-\lambda J_{s}} \Ll' e^{\lambda J_{s}} \, ds \right\}
\]
is a local martingale with respect to $\Pp_{0}^{N}$. We will actually only need to consider $\lambda = 2$. Writing out the action of $\Ll'$ on the function $g(\eta, J) = e^{2 J}$ we obtain
\[
Z_{t} = \exp\left\{ 2 J_{t}- \frac{1}{2} N^{\alpha} \int\limits_{0}^{t} \sum\limits_{x=1}^{N} \left[ \left( e^{2 \log(1 + \varepsilon (\nabla_{x} v)(\eta_{s}))} - 1\right) +  \left( e^{2 \log(1 \pm \varepsilon r_x (\eta_{s})} - 1 \right) \right]  ds \right\}.
\]
Now we have
\begin{align*}
& e^{2 \log(1 + \varepsilon (\nabla_{x} v)(\eta_{s}))} - 1 = (1 + \varepsilon (\nabla_{x} v)(\eta_{s}))^2 - 1 = 2 \varepsilon (\nabla_{x} v)(\eta_{s}) + \varepsilon^2 \left[(\nabla_{x} v)(\eta_{s})\right]^2, \\
& e^{2 \log(1 \pm \varepsilon r_{x}(\eta_{s}))} - 1 = (1 \pm \varepsilon r_{x}(\eta_{s}))^2 - 1 = \pm 2 \varepsilon r_{x}(\eta_{s}) + \varepsilon^2 r_{x}(\eta_{s})^2.
\end{align*}
The sum of terms linear in $\varepsilon$ vanishes -- the rates $r$ for $\pm 1$ color change have opposite sign and the sum involving $\nabla_{x}v$ telescopes. Recalling that $\varepsilon = N^{1-\alpha}$ and $\gamma = 3 - \alpha$, so $N^{\alpha+1}\varepsilon^2 = N^\gamma$, we can then write
\[
Z_{t} = \exp\left\{ 2 J_{t} - \frac{1}{2} N^{\gamma} \int\limits_{0}^{t} \frac{1}{N}\sum\limits_{x=1}^{N} \left[ (\nabla_{x} v)(\eta_{s})^2  + r_x(\eta_{s})^2\right] \, ds \right\}.
\]
Since the rates $v$ and $r$ are bounded, we have $Z_t = e^{ 2 J_t - N^{\gamma} X_t}$, where $|X_t| \leq C$ for some constant $C > 0$ depending only on $v$, $r$ and $T$. In particular we get
\[
\E e^{2 J_t} = \E \left( e^{2 J_t - N^{\gamma} X_t} e^{N^{\gamma}X_t} \right) = \E \left( Z_t e^{N^{\gamma}X_t} \right) \leq e^{C N^{\gamma}} \E Z_t.
\]
Since $Z_t$ is a local martingale bounded from below, it is a supermartingale, so we have $\E Z_t \leq \E Z_0 = 1$ and thus
\begin{equation}\label{eq:exponential-zt}
\E e^{2 J_t} \leq e^{C N^\gamma}.
\end{equation}
Now we can transfer the superexponential bound of Lemma \ref{lm:one-block-superexponential-unbiased} from $\Pp_{0}^{N}$ to $\widetilde{\Pp}^{N}$. Let $\Oo_{N,l}$ be the event from the statement of the lemma and let us write simply $\frac{\dd \widetilde{\Pp}^{N}}{\dd \Pp_{0}^{N}} = \frac{\dd \widetilde{\Pp}^{N}}{\dd \Pp_{0}^{N}}(T)$. Denoting by $\widetilde{\E}$ the expectation with respect to $\widetilde{\Pp}^{N}$ we have
\[
\widetilde{\Pp}^{N}\left( \Oo_{N,l} \right) = \widetilde{\E} \left( \id_{\Oo_{N,l}} \right) = \E \left( \frac{\dd \widetilde{\Pp}^{N}}{\dd \Pp_{0}^{N}} \id_{\Oo_{N,l}} \right).
\]
Applying the Cauchy-Schwarz inequality gives
\[
\widetilde{\Pp}^{N}\left( \Oo_{N,l} \right) \leq \left[ \E \left(\frac{\dd \widetilde{\Pp}^{N}}{\dd \Pp_{0}^{N}}\right)^2 \right]^{1/2} \cdot \Pp_{0}^{N}\left(\Oo_{N,l}\right)^{1/2}.
\]
Recalling that $\frac{\dd \widetilde{\Pp}^{N}}{\dd \Pp_{0}^{N}} = e^{J_T}$ and applying the bound \eqref{eq:exponential-zt} we obtain
\[
\widetilde{\Pp}^{N}\left( \Oo_{N,l} \right) \leq e^{c N^{\gamma}} \Pp_{0}^{N}\left(\Oo_{N,l}\right)^{1/2}
\]
with $c = \frac{C}{2}$. Thus
\[
\limsup\limits_{N \to \infty} N^{-\gamma} \log \widetilde{\Pp}^{N}\left( \Oo_{N,l} \right) \leq  c + \frac{1}{2}\limsup\limits_{N \to \infty} N^{-\gamma} \log \Pp_{0}^{N}\left( \Oo_{N,l} \right)
\]
and taking $\limsup$ as $l \to \infty$ together with an application of Lemma \ref{lm:one-block-superexponential-unbiased} finishes the proof.
\end{proof}
\end{section}

\begin{section}{Large deviation lower bound}\label{sec:lower-bound}

In this section we prove the large deviation lower bound of Theorem \ref{th:theorem-main-lower}. Let us assume that the permuton process $X$ satisfies equations \ref{eq:ode-general}. Since we already know how to construct a biased interchange process that will typically display the behavior of $X$, to bound the probability that the trajectory of a random particle in the interchange process is close in distribution to $X$ we only need to compare the unbiased process with the biased one by means of calculating their Radon-Nikodym derivative.

Since these two processes have different configuration spaces, for convenience we introduce the \emph{unbiased interchange process with colors}, which has the same configuration space as the biased process associated to \ref{eq:ode-general} and the generator $\Ll^{u}$ obtained by putting all velocities $v$ to $0$
\begin{align}\label{eq:unbiased-with-colors}
& (\Ll^{u}_{t} f)(\eta) = \frac{1}{2} N^{\alpha} \sum\limits_{y=1}^{N-1} (f(\eta^{y, y+1}) - f(\eta))  + \frac{1}{2} N^{\alpha} \sum\limits_{x=1}^{N} \left[ 1 \pm \varepsilon r(t, x, \phi_{x}(\eta)) \right] (f(\eta^{x,\pm}) - f(\eta)).
\end{align}
Since here the colors do not influence the dynamics of swaps, the corresponding permutation process $X^{\eta^{N}}$ will be the same as for the ordinary unbiased interchange process (and we will never be interested in the distribution of $\Phi^{\eta^{N}}$ for the unbiased process with colors).

Let us start by deriving the formula for the Radon-Nikodym derivative of the unbiased process with colors with respect to the biased one. Recall that $v_x (s, \eta_{s}) = v(s,x,\phi_{x}(\eta_s))$ denotes the velocity at time $s$ of the particle at site $x$. Let $\Pp^{N}_{u}$ denote the law of the unbiased process with colors. We will prove the following statement
\begin{lemma}\label{lm:radon-nikodym-lower-bound}
We have
\[
\frac{\dd \Pp_{u}^{N}}{\dd \widetilde{\Pp}^{N}}(T) = \exp \Bigg\{ - \frac{1}{2} N^{\gamma} \Bigg[ \int\limits_{0}^{T} \frac{1}{N} \sum\limits_{x=1}^{N} v_x (s, \eta_{s})^2 \, ds + o(1) \Bigg]\Bigg\},
\]
where the $o(1)$ term goes to $0$ in probability as $N \to \infty$.
\end{lemma}
\begin{proof}
The calculation is similar as in the proof of Lemma \ref{lm:one-block-superexponential-biased}, with the difference that we are using generator $\widetilde{\Ll}$ instead of $\Ll_0$. By the analog of formula \eqref{eq:radon-nikodym-sum-rates} for time-inhomogeneous processes we have
\[
\frac{\dd \Pp_{u}^{N}}{\dd \widetilde{\Pp}^{N}}(t)  = \exp \left\{ -\sum\limits_{s \leq t} \log \left( 1 + \varepsilon \left[v(s,x_{j_{s}}, \phi_{x_{j_{s}}}(\eta_{s})) - v(s,x_{j_{s}}+1, \phi_{x_{j_{s}}+1}(\eta_{s}))\right] \right)  \right\},
\]
where the sum is over jump times $s \leq t$.

Denoting the sum in the exponent by $J_{t}$, we obtain by \eqref{eq:martingale1} (by considering as before the generator $\widetilde{\Ll}$ acting on an extended configuration space) that
\[
J_{t} = M_{t} + \frac{1}{2} N^{\alpha}\sum\limits_{x=1}^{N-1} \int\limits_{0}^{t} \left[ 1 + \varepsilon (\nabla_{x} v)(s, \eta_{s})\right] \log \left[ 1 + \varepsilon (\nabla_{x} v)(s, \eta_{s})\right] ds,
\]
where $M_{t}$ is a local martingale with respect to $\widetilde{\Pp}^{N}$. Expanding all terms up to order $\varepsilon^2$ allows us to write
\[
\frac{\dd \Pp_{u}^{N}}{\dd \widetilde{\Pp}^{N}}(t)  = \exp \left\{- M_{t} -  \frac{1}{2} N^{\alpha} \int\limits_{0}^{t} \sum\limits_{x=1}^{N-1} \left[ \varepsilon (\nabla_{x} v)(s,\eta_{s}) + \frac{\varepsilon^2}{2} \left[ (\nabla_{x} v)(s,\eta_{s}) \right]^2  \right] ds  + O(N^{\alpha+1}\varepsilon^3)\right\}.
\]
As before the term linear in $\varepsilon$ vanishes. Recalling that $\varepsilon = N^{1-\alpha}$ and $\gamma = 3 - \alpha$ we have
\[
\frac{\dd \Pp_{u}^{N}}{\dd \widetilde{\Pp}^{N}}(t)  = \exp \left\{ - M_{t} - \frac{1}{4} N^{\gamma} \int\limits_{0}^{t} \frac{1}{N} \sum\limits_{x=1}^{N-1} \left[ (\nabla_{x} v)(s,\eta_{s}) \right]^2  \, ds  + o(N^{\gamma})\right\}.
\]
Expanding $\left[(\nabla_{x} v)(s,\eta_s)\right]^2 = (v_x(s,\eta_s) - v_{x+1}(s,\eta_s))^2$ leads us to
\begin{align}\label{eq:expansion-of-radon-nikodym}
& \frac{\dd \Pp_{u}^{N}}{\dd \widetilde{\Pp}^{N}}(t)  = \exp \Bigg\{ - M_{t} - \frac{1}{2} N^{\gamma} \int\limits_{0}^{t} \frac{1}{N} \sum\limits_{x=1}^{N} v_x(s,\eta_{s})^2 \, ds  + \frac{1}{2} N^{\gamma} \int\limits_{0}^{t} \frac{1}{N} \sum\limits_{x=1}^{N} v_x(s,\eta_{s})v_{x+1}(s,\eta_{s}) \, ds + o(N^{\gamma})\Bigg\}.
\end{align}
The martingale term will be typically $o(N^{\gamma})$. To see this, we use formula \eqref{eq:martingale2} -- by performing a calculation similar to the one above we get that
\[
N_{t} = M_{t}^2 - \frac{1}{2} N^{\alpha} \int\limits_{0}^{t} \sum\limits_{x=1}^{N} \left[ 1 + \varepsilon (\nabla_{x} v)(s,\eta_{s})\right] \left[\log \left( 1 + \varepsilon (\nabla_{x} v)(s,\eta_{s})\right) \right]^2 \, ds
\]
is a local martingale with respect to $\widetilde{\Pp}^{N}$. By expanding the $\log$ terms up to $\varepsilon^2$ we see that the second term above is bounded by $C N^{\alpha+1} \varepsilon^2 = C N^{\gamma}$ for some $C > 0$. In particular $N_{t}$ is bounded from below, so it is a supermartingale. Thus $\E N_T \leq \E N_0 = 0$ and $\E M_T^2 \leq C N^\gamma$, so Chebyshev's inequality implies that $M_T = o(N^\gamma)$ with high probability.

The second sum in the exponent in \eqref{eq:expansion-of-radon-nikodym} will be small by invariance of the uniform distribution of colors in the biased process. More precisely, at fixed time $s$ for each $x$ the correlation term $v_{x}(s,\eta_{s})v_{x+1}(s,\eta_{s})$ has mean $0$, since $\eta_{s}$ has stationary distribution and by Proposition \ref{prop:speeds} in stationarity velocities at different sites are independent with mean $0$. Moreover, for the same reason these terms are uncorrelated for different $x$, so by the weak law of large numbers we get that for any $s \leq T$ and $\delta > 0$
\[
\widetilde{\Pp}^{N} \left( \left| \frac{1}{N} \sum\limits_{x=1}^{N} v_{x}(s,\eta_{s})v_{x+1}(s,\eta_{s}) \right| > \delta \right) \to 0
\]
as $N \to \infty$.

Since this holds for any fixed $s$ and the random variables are bounded, we also have
\[
\widetilde{\Pp}^{N} \left( \left| \int\limits_{0}^{T} \frac{1}{N} \sum\limits_{x=1}^{N} v_{x}(s,\eta_{s})v_{x+1}(s,\eta_{s}) \, ds \right| > \delta \right) \to 0
\]
as $N \to \infty$, which proves that the correlation term is $o(N^{\gamma})$ with high probability. Together with the bound on $M_t$ this proves the desired formula for the Radon-Nikodym derivative.
\end{proof}

We can now use Lemma \ref{lm:radon-nikodym-lower-bound} and the law of large numbers established in Theorem \ref{th:lln} to prove a large deviation lower bound for the interchange process. As the formula from the lemma suggests, the large deviation rate function will be related to the energy of the process to which the biased interchange process converges.

Recall from \eqref{eq:process-energy} that for any process $\pi \in \PP$ its energy was defined by
\[
I(\pi) = \E_{\gamma \sim \pi} \EE(\gamma),
\]
where $\EE(\gamma)$ is the Dirichlet energy of the path $\gamma$ defined by \eqref{eq:energy-path-full}. We have the following large deviation lower bound
\begin{theorem}\label{th:lower-bound}
Let $\Pp^{N}$ be the law of the unbiased interchange process $\eta^N$ and let $\mu^{\eta^{N}}$ be the (random) distribution of the corresponding permutation process $X^{\eta^N}$. Let $P = (X, \Phi)$ be the colored trajectory process associated to the equation \eqref{eq:ode-general} and let $\mu$ denote the distribution of $X$. For any open set $\Oo \subseteq \PP$ such that $\mu \in \Oo$ we have
\[
\liminf_{N \to \infty} N^{-\gamma} \log \Pp^{N}\left(\mu^{\eta^{N}} \in \Oo\right) \geq - I(\mu).
\]
\end{theorem}

\begin{proof}
It will be enough to show the bound above for $\Oo$ being any open ball $B(\mu, \varepsilon)$ in $\PP$ around $\mu$. Let $\Pp_{u}^{N}$ be the distribution of the unbiased process with colors, $\nu^{\eta^{N}}$ the distribution of the colored permutation process $P^{\eta^N} = (X^{\eta^N}, \Phi^{\eta^N})$ associated to $\eta^N$. Let $\nu$ denote the distribution of $P = (X, \Phi)$ and $\widetilde{B}(\nu, \varepsilon)$ an open ball around $\nu$ in $\MM(\DDD)$. Since the projection $(X, \Phi) \mapsto X$ is continuous as a map from $\DDD$ to $\DD$, the corresponding projection from $\MM(\DDD)$ to $\MM(\DD)$ is also continuous. As $\mu^{\eta^N}$ has the same law under $\Pp^N$ and $\Pp_{u}^N$ (remember that in the latter process the colors do not influence the dynamics of swaps), we have that for any $\varepsilon > 0$ there exists $\varepsilon' > 0$ such that
$\Pp^{N}\left(\mu^{\eta^{N}} \in B(\mu, \varepsilon) \right) \geq \Pp_{u}^{N}\left(\nu^{\eta^{N}} \in \widetilde{B}(\nu, \varepsilon')\right)$. Thus to prove the large deviation bound it is sufficient to prove the local lower bound
\begin{equation}\label{eq:lower-bound-nu-epsilon}
\liminf\limits_{\varepsilon \to 0}\liminf_{N \to \infty} N^{-\gamma} \log \Pp_{u}^{N}\left(\nu^{\eta^{N}} \in \widetilde{B}(\nu, \varepsilon)\right) \geq - I(\mu).
\end{equation}

Recall that $\widetilde{\Pp}^{N}$ denotes the distribution of the biased process associated to \eqref{eq:ode-general} and consider the Radon-Nikodym derivative $\frac{\dd \Pp_{u}^{N}}{\dd \widetilde{\Pp}^{N}}(t)$.
By Lemma \ref{lm:radon-nikodym-lower-bound} we have
\begin{equation}\label{eq:radon-nikodym-in-lower-bound}
\frac{\dd \Pp_{u}^{N}}{\dd \widetilde{\Pp}^{N}}(T) = \exp \Bigg\{ -\frac{1}{2} N^{\gamma} \Bigg[ \int\limits_{0}^{T} \frac{1}{N} \sum\limits_{x=1}^{N} v_{x}(\eta_{s})^2 \, ds  +  Y_N\Bigg]\Bigg\},
\end{equation}
where $Y_N$ goes to $0$ in probability as $N \to \infty$.

Now by the law of large numbers from Theorem \ref{th:lln} and Remark \ref{rm:lln} the distributions $\nu^{\eta^{N}}$ converge in probability in the $d_{\mathcal{W}}^{sup}$ metric to $\nu$ when $\eta^{N}$ is sampled according to $\widetilde{\Pp}^{N}$. Thus for any $\varepsilon > 0$ and an open ball $\widetilde{B}_{\varepsilon} = \{ \zeta \in \MM(\DDD) \, | \, d_{\mathcal{W}}^{sup}(\zeta, \nu) < \varepsilon \}$ around $\nu$ in the $d_{\mathcal{W}}^{sup}$ metric we have $\lim\limits_{N \to \infty} \widetilde{\Pp}^{N}\left(\nu^{\eta^{N}} \in \widetilde{B}_{\varepsilon}\right) = 1$. Since convergence in $d_{\mathcal{W}}^{sup}$ implies convergence in $d_{\mathcal{W}}$, to prove \eqref{eq:lower-bound-nu-epsilon} it is enough to analyze the probability $\Pp^N_{u}\left(\nu^{\eta^{N}} \in \widetilde{B}_{\varepsilon}\right)$.

Fix arbitrary $\delta > 0$ and let $U_{N} = \{ |Y_N| \leq \delta \}$. Let $V_{N,\varepsilon} = U_{N} \cap \{ \nu^{\eta^{N}} \in \widetilde{B}_{\varepsilon}\}$ and $\frac{\dd \Pp_{u}^{N}}{\dd \widetilde{\Pp}^{N}} = \frac{\dd \Pp_{u}^{N}}{\dd \widetilde{\Pp}^{N}}(T)$. With $\E$ denoting the expectation with respect to $\Pp_{u}^{N}$ and $\widetilde{\E}$ with respect to $\widetilde{\Pp}^{N}$ we have for any $\varepsilon > 0$ and sufficiently large $N$
\[
\Pp^{N}_{u}\left(\nu^{\eta^{N}} \in \widetilde{B}_{\varepsilon}\right) = \E \left(\id_{\{ \nu^{\eta^{N}} \in \widetilde{B}_{\varepsilon} \}}\right) \geq \E (\id_{V_{N,\varepsilon}}) = \widetilde{\E} \left( \frac{\dd \Pp_{u}^{N}}{\dd \widetilde{\Pp}^{N}} \id_{V_{N,\varepsilon}} \right) \geq \widetilde{\Pp}^{N}(V_{N,\varepsilon}) \left(\inf_{ V_{N,\varepsilon}}\frac{\dd \Pp_{u}^{N}}{\dd \widetilde{\Pp}^{N}} \right).
\]
We have $\lim\limits_{N \to \infty} \widetilde{\Pp}^{N}(V_{N,\varepsilon}) = 1 $ and on the event $U_N$ we have
\[
\frac{\dd \Pp_{u}^{N}}{\dd \widetilde{\Pp}^{N}}(T) \geq \exp \Bigg\{ -\frac{1}{2} N^{\gamma} \Bigg[ \int\limits_{0}^{T} \frac{1}{N} \sum\limits_{x=1}^{N} v_{x}(\eta_{s})^2 \, ds  +  \delta\Bigg]\Bigg\}.
\]
This implies
\begin{equation}\label{eq:lower-bound-with-in}
N^{-\gamma} \log \Pp_{u}^{N}\left(\nu^{\eta^{N}} \in \widetilde{B}_{\varepsilon}\right) \geq - \inf_{\eta \in V_{N,\varepsilon}} I_N (\eta) - \delta,
\end{equation}
where
\[
I_{N}(\eta) = \frac{1}{2}  \left(\frac{1}{N} \sum\limits_{x=1}^{N} \int\limits_{0}^{T} v_x(s,\eta_{s})^2 \, ds \right).
\]

Now it is not difficult to see that the infimum on the right hand side of \eqref{eq:lower-bound-with-in} converges to $I(\mu)$ as $N \to \infty$ and then $\varepsilon \to 0$. When $(X, \Phi)$ is sampled from $\nu$, $X$ is the solution of \eqref{eq:ode-general} with a uniformly random initial condition, so the energy $I(\mu)$ is simply equal to
\[
\E \int\limits_{0}^{T} V(s, X(s),\Phi(s))^2 \, ds,
\]
where the expectation is with respect to the choice of $(X(0), \Phi(0))$. Recall the notation $X_{i}(\eta^{N}_{t}) = \frac{1}{N} x_{i}(\eta^{N}_{t})$, $\Phi_{i}(\eta^{N}_{t}) = \frac{1}{N} \phi_{i}(\eta^{N}_{t})$. In light of \eqref{eq:s-r-approx} what we need to show is that
\begin{equation}\label{eq:energy}
\inf\limits_{\eta \in V_{N,\varepsilon}} \left( \frac{1}{N} \sum\limits_{i=1}^{N} \int\limits_{0}^{T} V(s, X_{i}(\eta^{N}_{s}),\Phi_{i}(\eta^{N}_{s}))^2 \, ds \right) \to \E \int\limits_{0}^{T} V(s, X(s),\Phi(s))^2 \, ds
\end{equation}
as $N \to \infty$ and then $\varepsilon \to 0$.

Consider the trajectory $\eta^N$ and for any particle $i$ let $(X_{i}(t), \Phi_{i}(t))$ denote the solution of \eqref{eq:ode-general} corresponding to the initial condition $(X_{i}(\eta^{N}_{0}), \Phi_{i}(\eta^{N}_{0}))$. Since the velocities $V$ are bounded, we can write
\begin{align}
& \left| \int\limits_{0}^{T} \left[ V(s,X_{i}(\eta^{N}_{s}),\Phi_{i}(\eta^{N}_{s}))^2 - V(s,X_{i}(s),\Phi_{i}(s))^2 \right] ds \right| \leq \nonumber \\
& \leq C \int\limits_{0}^{T} \big| V(s,X_{i}(\eta^{N}_{s}),\Phi_{i}(\eta^{N}_{s})) - V(s,X_{i}(s),\Phi_{i}(s)) \big| \, ds \label{eq:speeds-squared} \leq \\ & \leq K T \max \left\{\sup_{t \leq T} \left| X_{i}(\eta^{N}_{t}) - X_{i}(t)\right|, \sup_{t \leq T} \left| \Phi_{i}(\eta^{N}_{t}) - \Phi_{i}(t)\right| \right\} \nonumber
\end{align}
for some $C, K > 0$ depending on the bound on $V$ and the Lipschitz constant of $V$.

Now note that if $\nu^{\eta^N} \in \widetilde{B}_{\varepsilon}$, then by considering the same coupling as in the proof of Theorem \ref{th:lln} we have
\begin{equation}
\limsup\limits_{N \to \infty} \frac{1}{N} \sum\limits_{i=1}^{N} \left| \max \left\{\sup_{t \leq T} \left| X_{i}(\eta^{N}_{t}) - X_{i}(t)\right|, \sup_{t \leq T} \left| \Phi_{i}(\eta^{N}_{t}) - \Phi_{i}(t)\right| \right\} \right| \leq \varepsilon'
\end{equation}
for some $\varepsilon' > 0$ satisfying $\varepsilon' \to 0$ as $\varepsilon \to 0$. Since $\{ \nu^{\eta^{N}} \in \widetilde{B}_{\varepsilon}\} \subseteq V_{N,\varepsilon}$, combining this with \eqref{eq:speeds-squared} we obtain that the left hand side of \eqref{eq:energy} converges to
\[ \frac{1}{N} \sum\limits_{i=1}^{N} \int\limits_{0}^{T} V(s,X_{i}(s),\Phi_{i}(s))^2 \, ds
\]
as $N \to \infty$ and then $\varepsilon \to 0$.

Since $(X_{i}(t), \Phi_{i}(t))$ is a solution of \eqref{eq:ode-general} and $V$ is the derivative of $X$, the integral is equal simply to the energy of the path $X_{i}(t)$. Since for each $i$ the initial condition $\Phi_{i}(\eta^{N}_{0})$ has uniform distribution on $\left \{\frac{1}{N}, \ldots, 1  \right\}$, independently for all $i$, it follows easily that this expression converges with high probability to the expected energy on the right hand side of \eqref{eq:energy}. This implies $\inf_{\eta \in V_{N,\varepsilon}} I_N (\eta) \to I(\nu)$ as $N \to \infty$ and then $\varepsilon \to 0$. Since in \eqref{eq:lower-bound-with-in}  we can take $\delta$ to be arbitrarily small, this proves \eqref{eq:lower-bound-nu-epsilon} and finishes the proof of the lower bound.
\end{proof}

With this theorem the large deviation lower bound for generalized solutions to Euler equations, announced as Theorem \ref{th:theorem-main-lower} in the introduction, is now an easy corollary.

\begin{theorem}\label{thm:lower-bound-for-minimizers}
Let $\Pp^{N}$ be the law of the interchange process $\eta^N$ and let $\mu^{\eta^{N}}$ be the (random) distribution of the corresponding permutation process $X^{\eta^N}$. Let $\pi$ be a permuton process which is a generalized solution to Euler equations \eqref{eq:gen-euler}. Provided $\pi$ satisfies Assumptions \eqref{as:main-assumptions}, for any open set $\Oo \subseteq \PP$ such that $\pi \in \Oo$ we have
\[
\liminf_{N \to \infty} N^{-\gamma} \log \Pp^{N}\left(\mu^{\eta^{N}} \in \Oo\right) \geq - I(\pi).
\]
\end{theorem}

\begin{proof}
Let $\pi^{\beta, \delta}$ be the distribution of the process $X^{\beta, \delta}$ defined in Section \ref{sec:ode-part}. By the first part of Proposition \ref{prop:approximation-epsilon-delta} we have $d_{\mathcal{W}}^{sup}(\pi, \pi^{\beta, \delta}) \to 0$ as first $\delta$ and then $\beta \to 0$, in particular for small enough $\delta$ and $\beta$ we have $\pi^{\beta, \delta} \in \Oo$. Then Theorem \ref{th:lower-bound} implies that
\[
\liminf_{N \to \infty} N^{-\gamma} \log \Pp^{N}\left(\mu^{\eta^{N}} \in \Oo\right) \geq - I(\pi^{\beta, \delta}).
\]
Since by the second part of Proposition \ref{prop:approximation-epsilon-delta} we have $\lim\limits_{\beta \to 0} \lim\limits_{\delta \to 0} I(\pi^{\beta, \delta}) = I(\pi)$, the lower bound is proved.
\end{proof}

\end{section}

\begin{section}{Large deviation upper bound}\label{sec:upper-bound}

In this section we prove Theorem \ref{th:theorem-main-upper}, a large deviation upper bound for the distribution of the interchange process (we will drop the term ``unbiased'' from now on). As a first step we will bound the probability that after a (possibly short) time $t > 0$ we see a fixed permutation in the interchange process. This is summarized in the following

\begin{proposition}\label{prop:one-slice-bound}
Let $\Pp^{N}$ be the law of the interchange process, with $\eta = \eta^N$ denoting the trajectory of the process. Let $\sigma^N \in \mathcal{S}_N$ be a sequence of permutations. For any $t > 0$ we have
\[
\limsup\limits_{N \to \infty} N^{-\gamma}\log \Pp^{N}(\eta_{0}^{-1}\eta_{t} = \sigma_N) \leq - \frac{1}{t} \left(\liminf\limits_{N \to \infty} I(\sigma^N) \right),
\]
where $I(\sigma)$ is the energy of the permutation $\sigma$ defined in \eqref{eq:permutation-energy}.
\end{proposition}

In other words, the large deviation rate of seeing a permutation $\sigma$ at time $t$ in the interchange process is asymptotically bounded from above by $\frac{1}{t}$ times the energy of the permutation $\sigma$.

To prove Proposition \ref{prop:one-slice-bound} we will employ exponential martingales. The idea is as follows -- if $M_{S}(\eta)$ is a function of the process (depending on some set of parameters $S$) which is a positive mean one martingale, then for any permutation $\sigma \in \mathcal{S}_N$ we can write
\begin{align}\label{eq:upper-bound-ej}
& \Pp^{N}(\eta_{0}^{-1}\eta_{t} = \sigma) = \E (\id_{\{\eta_{0}^{-1}\eta_{t} = \sigma\}} )= \E \left( M_{S}(\eta) M_{S}(\eta)^{-1} \id_{\{\eta_{0}^{-1}\eta_{t} = \sigma\}}\right) \leq \\
& \sup_{ \{\chi_{0}^{-1}\chi_{t} = \sigma\}} M_{S}(\chi)^{-1} \E \left( M_{S}(\eta) \id_{\{\eta_{0}^{-1}\eta_{t} = \sigma\}}\right) \leq \sup_{\{\chi_{0}^{-1}\chi_{t} = \sigma\}} M_{S}(\chi)^{-1}, \nonumber
\end{align}
where the supremum is over all deterministic permutation-valued paths $\chi = (\chi_s, 0 \leq s \leq T)$ satisfying $\chi_{0}^{-1}\chi_{t} = \sigma$ and the last inequality comes from the fact that $M_{S}(\chi)$ is a positive mean one martingale. If $M_{S}$ depends only on the increment $\chi_{0}^{-1}\chi_{t}$, we obtain a particularly simple expression
\[
\Pp^{N}(\eta_{0}^{-1}\eta_{t} = \sigma) \leq M_{S}(\sigma)^{-1}.
\]
We can then optimize over the set of parameters $S$ to obtain a large deviation upper bound. The family of martingales we will use is similar to the one used in analyzing large deviations for a simple random walk.

Fix $t > 0$ and a sequence $S = (s_{1}, \ldots, s_{N})$, with $s_{i} \in \left\{ \frac{-1 + \frac{1}{N}}{t}, \frac{-1 + \frac{2}{N}}{t}, \ldots, \frac{1 - \frac{2}{N}}{t}, \frac{1 - \frac{1}{N}}{t} \right\}$. We will think of $s_i$ as ``velocity'' assigned to the particle $i$. Consider the function
\[
F_{S}(\eta_{t}) = \varepsilon \sum\limits_{i=1}^{N} s_{i} x_{i}(\eta_{t}),
\]
where $x_{i}(\eta_{t})$ is the position of the particle $i$ in the configuration $\eta_{t}$. If $\Ll$ is the generator of the interchange process, given by \eqref{eq:unbiased-generator}, then by the formula \eqref{eq:martingale-exp} for exponential martingales we obtain that
\[
M_{t}^{S} = \exp \left\{ F_{S}(\eta_{t}) - F_{S}(\eta_{0}) - \int\limits_{0}^{t} e^{-F_{S}(\eta_{s})} \Ll e^{F_{S}(\eta_{s})} \, ds \right\}
\]
is a mean one positive martingale with respect to $\Pp^{N}$.

For simplicity we will use the same notation $s_{x}(\eta) = s_{\eta^{-1}(x)}$ as for velocities $v_x$ of particles in the previous sections (with the convention that $i$ denotes labels of particles and $x$ denotes the positions), although bear in mind that now $s_x$ are just parameters, not related in any way to the the biased interchange process considered in the preceding sections. We have
\[
\Ll e^{F_{S}(\eta)} = \frac{1}{2} N^{\alpha} \sum\limits_{x=1}^{N-1} \left( e^{F_{S}(\eta^{x,x+1})} - e^{F_{S}(\eta)}  \right) = \frac{1}{2} N^{\alpha} \sum\limits_{x=1}^{N-1} \left( e^{F_{S}(\eta) + \varepsilon \left[s_{x}(\eta) - s_{x+1}(\eta)\right]} - e^{F_{S}(\eta)}  \right),
\]
so
\[
M_{t}^{S} = \exp \left\{ \varepsilon \sum\limits_{i=1}^{N} s_{i} \left( x_{i}(\eta_{t}) - x_{i}(\eta_{0})\right) - \frac{1}{2} N^{\alpha} \int\limits_{0}^{t} \sum\limits_{x=1}^{N-1} \left( e^{\varepsilon [s_{x}(\eta_{s}) - s_{x+1}(\eta_{s})]} -1\right) ds \right\}.
\]
Expanding up to order $\varepsilon^2$ we get
\begin{align}\label{eq:exponential}
M_{t}^{S} = & \exp \Bigg\{ \varepsilon \sum\limits_{i=1}^{N} s_{i} \left( x_{i}(\eta_{t}) - x_{i}(\eta_{0})\right) - \frac{1}{2} N^{\alpha} \varepsilon \int\limits_{0}^{t} \sum\limits_{x=1}^{N-1} [s_{x}(\eta_{s}) - s_{x+1}(\eta_{s})]  \, ds - \\
& -  \frac{1}{4} N^{\alpha}\varepsilon^2 \int\limits_{0}^{t} \sum\limits_{x=1}^{N-1} \left(s_{x}(\eta_{s}) - s_{x+1}(\eta_{s})\right)^2  \, ds  + O(N^{\alpha+1} \varepsilon^3 ) \Bigg\} \nonumber,
\end{align}
where the constants in the $O(\cdot)$ notation depend on $t$ (which is fixed). Observe that the sum of $s_{x} - s_{x+1}$ telescopes, leaving only terms with $s_{1}$ and $s_{N}$, which are $O(N^{\alpha} \varepsilon) = o(N^{\gamma})$. Rescaling by appropriate powers of $N$ and expressing the exponents in terms of the large deviation exponent $\gamma$ we get
\begin{equation*}
M_{t}^{S} = \exp \Bigg\{ N^{\gamma} \left[ \frac{1}{N}\sum\limits_{i=1}^{N} s_{i} \left( \frac{x_{i}(\eta_{t}) - x_{i}(\eta_{0})}{N} \right) -  \frac{1}{4} \int\limits_{0}^{t} \frac{1}{N}\sum\limits_{x=1}^{N-1} \left(s_{x}(\eta_{s}) - s_{x+1}(\eta_{s})\right)^2  \, ds  \right]+ o(N^{\gamma}) \Bigg\}.
\end{equation*}
Expanding $(s_x - s_{x+1})^2$ we obtain (after adding and subtracting the boundary terms $s_1^2$, $s_N^2$ which are only $o(1)$ after rescaling) twice the sum of $s_{x}^{2}$ and the sum of mixed terms $s_x s_{x+1}$. Since $\sum\limits_{x=1}^N s_x^2 = \sum\limits_{i=1}^N s_i^2 $ does not depend on time, we can write
\begin{align*}
 M_{t}^{S} = & \exp \Bigg\{ N^{\gamma} \Bigg[ \frac{1}{N}\sum\limits_{i=1}^{N} s_{i} \left( \frac{x_{i}(\eta_{t}) - x_{i}(\eta_{0})}{N} \right) -  \frac{t}{2} \left( \frac{1}{N}\sum\limits_{i=1}^{N} s_{i}^{2} \right) + \\
 & + \frac{1}{2} \int\limits_{0}^{t} \frac{1}{N}\sum\limits_{x=1}^{N-1} s_{x}(\eta_{s}) s_{x+1}(\eta_{s}) \, ds + o(1)\Bigg] \Bigg\}.
\end{align*}
As in the proof of the law of large numbers we want to use the one block estimate to get rid of the sum involving correlations between $s_{x}$ for adjacent $x$. This time the correlation term might not be small, since $s_{i}$ are arbitrary, but typically it will be nonnegative, so we can neglect it for the sake of the upper bound. More precisely, we have the following

\begin{lemma}\label{lm:upper-bound-one-block}
Let $\Pp^{N}$ be the law of the interchange process. Fix $t > 0$ and let \\ $s_{i} \in \left\{ \frac{-1}{t}, \frac{-1 + \frac{1}{N}}{t}, \ldots, \frac{1 - \frac{1}{N}}{t}, \frac{1}{t} \right\}$. Then, with notation as above, we have for any $\delta > 0$
\[
\limsup\limits_{N \to \infty} N^{-\gamma} \log \Pp^{N} \left( \int\limits_{0}^{t} \frac{1}{N}\sum\limits_{x=1}^{N-1} s_{x}(\eta_{s}) s_{x+1}(\eta_{s}) \, ds \leq - \delta \right) = - \infty.
\]
\end{lemma}

\begin{proof}
We employ Lemma \ref{lm:one-block-superexponential-probability} with $w(x,\phi) = \frac{2x-1}{t}$, $a_i = s_i$ and $b_{x}(\eta) = a_{x}(\eta)$, in particular $h_{x}(\eta) = s_x(\eta) s_{x-1}(\eta)$. As in the lemma consider $\E_{\mu_{x,l}^{\eta_s}} \left(s_{x}(\eta) s_{x+1}(\eta)\right)$, where $\mu_{x,l}^{\eta_s}$ is the empirical distribution of $a_i$ in a box $\Lambda_{x,l}$. Let us write
\begin{align}\label{eq:sxsx+1}
& \int\limits_{0}^{t} \frac{1}{N}\sum\limits_{x=1}^{N-1} s_{x}(\eta_s) s_{x+1}(\eta_s) \, ds = \\ 
& = \int\limits_{0}^{t} \frac{1}{N}\sum\limits_{x=1}^{N-1} \left( s_{x}(\eta_s) s_{x+1}(\eta_s) - \E_{\mu_{x,l}^{\eta_s}} \left[s_{x}(\eta) s_{x+1}(\eta)\right] \right)\, ds + \int\limits_{0}^{t} \frac{1}{N}\sum\limits_{x=1}^{N-1} \E_{\mu_{x,l}^{\eta_s}} \left[s_{x}(\eta) s_{x+1}(\eta) \right]\, ds. \nonumber
\end{align}
Since under $\mu_{x,l}^{\eta_s}$ the colors are i.i.d. random variables, we have $\E_{\mu_{x,l}^{\eta_s}} \left[s_{x}(\eta) s_{x+1}(\eta)\right] = \left( \E_{\mu_{x,l}^{\eta_{s}}}  s_{x}(\eta)\right )^2 \geq 0$, so the second term on the right hand side of \eqref{eq:sxsx+1} is nonnegative for every $l$. Lemma \ref{lm:one-block-superexponential-probability} guarantees that for any $\delta > 0$
\[
\limsup\limits_{l \to \infty} \limsup\limits_{N \to \infty} N^{-\gamma} \log \Pp^{N} \left( \int\limits_{0}^{t} \frac{1}{N}\sum\limits_{x=1}^{N-1} \left( s_{x}(\eta_s) s_{x+1}(\eta_s) - \E_{\mu_{x,l}^{\eta_s}} \left[s_{x}(\eta) s_{x+1}(\eta)\right] \right)\, ds  \leq - \delta \right) = - \infty.
\]
Since the left hand side of \eqref{eq:sxsx+1} does not depend on $l$, this finishes the proof.
\end{proof}

With this lemma the proof of Proposition \ref{prop:one-slice-bound} is rather straightforward.

\begin{proof}[Proof of Proposition \ref{prop:one-slice-bound}]
Lemma \ref{lm:upper-bound-one-block} implies that for any $a > 0$ there exist sets $\Oo_{N,a}$ such that on $\Oo_{N,a}$ we have
\begin{equation}\label{eq:upper-bound-mt}
M_{t}^{S}(\eta) \geq \exp \Bigg\{ N^{\gamma} \Bigg[ \frac{1}{N}\sum\limits_{i=1}^{N} s_{i} \left( \frac{x_{i}(\eta_{t}) - x_{i}(\eta_{0})}{N} \right) -  \frac{t}{2} \left( \frac{1}{N}\sum\limits_{i=1}^{N} s_{i}^{2} \right) - a + o(1)\Bigg] \Bigg\},
\end{equation}
with the $o(1)$ term depending on $t$, and $\Pp^{N}(\Oo_{N,a}^{c}) \to 0 $ as $N \to \infty$ superexponentially fast
\[
 \limsup_{N \to \infty} N^{-\gamma} \log \Pp^{N}(\Oo_{N,a}^{c}) = - \infty.
\]
Now we can use the strategy outlined earlier with the positive mean one martingale $M_{t}^{S}(\eta)$. We write
\begin{align*}
& \Pp^{N}(\eta_{0}^{-1}\eta_{t} = \sigma^N) = \E \left(\id_{\{\eta_{0}^{-1}\eta_{t} = \sigma^N\}} \right)= \E \left( M_{t}^{S}(\eta)^{-1} M_{t}^{S}(\eta)  \id_{\{\eta_{0}^{-1}\eta_{t} = \sigma^N\}}\right) = \\
& \E \left( M_{t}^{S}(\eta)^{-1} M_{t}^{S}(\eta)  \id_{\{\eta_{0}^{-1}\eta_{t} = \sigma^N \}}\id_{\Oo_{N,a}} \right) + \E \left(  \id_{\{\eta_{0}^{-1}\eta_{t} = \sigma^N \}}\id_{\Oo_{N,a}^{c}} \right).
\end{align*}
On $\Oo_{N,a}$ we can use the bound \eqref{eq:upper-bound-mt} obtained above. Note also that on the event $\{\eta_{0}^{-1}\eta_{t} = \sigma^N\}$ we have $x_{i}(\eta_{t}) - x_{i}(\eta_{0}) = \sigma^N(i) - i$, which together with \eqref{eq:upper-bound-ej} leads us to
\begin{equation}\label{eq:one-slice-intermediate}
\Pp^{N}(\eta_{0}^{-1}\eta_{t} = \sigma^N) \leq e^{-N^{\gamma} \left(I_{S}(\sigma^N) - a + o(1)\right)} + \Pp^{N}(\Oo_{N,a}^{c}),
\end{equation}
where
\[
I_{S}(\sigma^N) = \frac{1}{N}\sum\limits_{i=1}^{N} s_{i} \left( \frac{\sigma^N(i) - i}{N} \right) -  \frac{t}{2} \left( \frac{1}{N}\sum\limits_{i=1}^{N} s_{i}^{2} \right).
\]
To optimize over the choice of $S = (s_1, \ldots, s_N)$, observe that $I_{S}(\sigma^N)$ is quadratic in $s_{i}$, so an easy calculation shows that the optimal choice is
\[
s_{i} = \frac{\sigma^N(i) - i}{t N},
\]
which is valid, since we assumed $s_{i} \in \left\{ \frac{-1 + \frac{1}{N}}{t}, \frac{-1 + \frac{2}{N}}{t}, \ldots, \frac{1 - \frac{1}{N}}{t}\right\}$. This gives the maximal value of $I_{S}(\sigma^N)$ equal to
\[
\frac{1}{2} \left( \frac{1}{N}\sum\limits_{i=1}^{N} \frac{1}{t}\left( \frac{\sigma^N(i) - i}{N} \right)^2 \right),
\]
which is exactly the energy $I(\sigma^N)$ rescaled by $t$. Inserting this into \eqref{eq:one-slice-intermediate} gives us
\[
\Pp^{N}(\eta_{0}^{-1}\eta_{t} = \sigma^N) \leq e^{-N^{\gamma} \left(I(\sigma^N) - a + o(1)\right)} + \Pp^{N}(\Oo_{N,a}^{c})
\]
Since $\limsup\limits_{n \to \infty} \frac{1}{n} \log(a_n + b_n) = \max\{\limsup\limits_{n \to \infty} \frac{1}{n} \log a_n, \limsup\limits_{n \to \infty} \frac{1}{n} \log b_n\}$, we obtain
\[
\limsup\limits_{N \to \infty} N^{-\gamma} \log \Pp^{N}(\eta_{0}^{-1}\eta_{t} = \sigma^N) \leq \max\left\{ - \liminf\limits_{N \to \infty} I(\sigma^N) + a, \limsup\limits_{N \to \infty} N^{-\gamma} \log \Pp^{N} (\Oo_{N,a}^{c}) \right\}.
\]
The second lim sup is $-\infty$ and by taking $a \to 0$ we arrive at
\[
\limsup\limits_{N \to \infty} N^{-\gamma} \log \Pp^{N}(\eta_{0}^{-1}\eta_{t} = \sigma^N) \leq - \liminf\limits_{N \to \infty} I(\sigma^N)
\]
as desired.
\end{proof}

We can readily extend the bound from Proposition \ref{prop:one-slice-bound} to all finite-dimensional distributions of the interchange process as follows. Fix a finite set of times $0 \leq t_0 < t_1 < \ldots < t_k \leq T$ and for clarity of notation let us write $\vec{\eta}_{t_0, \ldots, t_k} = (\eta_{t_{0}}^{-1}\eta_{t_1}, \ldots, \eta_{t_{k-1}}^{-1}\eta_{t_k})$ for the corresponding sequence of increments of $\eta$. Suppose we want to bound the probability $\Pp^{N}(\vec{\eta}_{t_0, \ldots, t_k} = (\sigma^N_{1}, \ldots, \sigma^N_{k}))$, where $(\sigma^N_{1}, \ldots, \sigma^N_{k})$ is a fixed sequence of permutations for each $N$, $\sigma^N_j \in \mathcal{S}_N$. Recall that the interchange process has independent increments, i.e., the permutations $\left(\eta_{t_{j}}\right)^{-1} \eta_{t_{j+1}}$ for any family non-overlapping intervals $[t_{j}, t_{j+1})$ are independent. Therefore we can write
\[
\Pp^{N}(\vec{\eta}_{t_0, \ldots, t_k} = (\sigma^N_{1}, \ldots, \sigma^N_{k})) = \prod\limits_{j=1}^{k} \Pp^{N}(\eta_{t_{j-1}}^{-1} \eta_{t_j} = \sigma^N_{j}).
\]
As the interchange process is stationary, we have $\Pp^{N}(\eta_{t_{j-1}}^{-1} \eta_{t_j} = \sigma^N_{j}) = \Pp^{N}(\eta_{0}^{-1} \eta_{t_{j} - t_{j-1}} = \sigma^N_{j})$. Thus by applying Proposition \ref{prop:one-slice-bound} we obtain 
\begin{equation}\label{eq:many-slices}
\limsup\limits_{N \to \infty} N^{-\gamma}\log \Pp^{N}(\vec{\eta}_{t_0, \ldots, t_k} = (\sigma^N_{1}, \ldots, \sigma^N_{k})) \leq - \liminf\limits_{N \to \infty} \sum\limits_{j=1}^{k} \frac{1}{t_j - t_{j-1}}I(\sigma^N_j).
\end{equation}

Recall that $\mu^{\eta^N}$ denotes the distribution of the random permutation process associated to $\eta^N$ (defined by \eqref{eq:empirical-eta}) and for a finite partition $\Pi$ by $I^{\Pi}(\mu^{\eta^N})$ we denote the approximation of energy of $\mu^{\eta^N}$ associated to $\Pi$ (defined by \eqref{eq:def-energy-fin-dim}). From equation \eqref{eq:many-slices} we obtain the following corollary which will be useful later

\begin{corollary}\label{cor:many-slices-energy}
For any $C > 0$ and any finite partition $\Pi = \{ 0 = t_0 < t_1 < \ldots < t_k = T\}$ we have
\[
\limsup\limits_{N \to \infty} N^{-\gamma}\log \Pp^{N}(I^{\Pi}(\mu^{\eta^{N}}) \geq C) \leq - C.
\]
\end{corollary}

\begin{proof}
Consider the set $A_{C}^{N}$ of all sequences of permutations $(\sigma^N_1, \ldots, \sigma^N_k)$, $\sigma^N_j \in \mathcal{S}_N$, such that $\sum\limits_{j=1}^{k} \frac{1}{t_j - t_{j-1}}I(\sigma^N_j) \geq C$. By performing a union bound over all such sequences we get
\[
\Pp^{N}(I^{\Pi}(\mu^{\eta^{N}}) \geq C) \leq N!^k \sup\limits_{A_{C}^{N}} \Pp^{N}(\vec{\eta}_{t_0, \ldots, t_k} = (\sigma^N_{1}, \ldots. \sigma^N_{k})),
\]
Now it is enough to observe that for fixed $k$ we have $\log (N!^k) = o(N^\gamma)$ and apply \eqref{eq:many-slices}.
\end{proof}

Now we can proceed to prove a general large deviation upper bound, announced as Theorem \ref{th:theorem-main-upper} in the introduction,. Recall that $\PP \subseteq \MM(\DD)$ denotes the space of all permuton and approximate permuton processes.

\begin{theorem}\label{th:upper-bound}
Let $\Pp^{N}$ be the law of the interchange process $\eta^{N}$ and let $\mu^{\eta^{N}}$ be the (random) distribution of the corresponding random permutation process $X^{\eta^N}$. For any closed set $\mathcal{C} \subseteq \PP$ we have
\[
\limsup_{N \to \infty} N^{-\gamma} \log \Pp^{N}\left(\mu^{\eta^{N}} \in \mathcal{C} \right) \leq - \inf\limits_{\pi \in \mathcal{C}} I(\pi),
\]
where $I(\pi)$ is the energy of the process $\pi$ defined by \eqref{eq:process-energy}.
\end{theorem}

\begin{proof}
It is standard (see, e.g., \cite[Lemma 2.3]{varadhan}) that the large deviation upper bound for closed sets follows from a local upper bound for open balls and exponential tightness of the sequence $\mu^{\eta^{N}}$. The exponential tightness part will be proved in Proposition \ref{prop:exp-tightness} below, so here we focus on the first part, that is, we will prove that for any $\pi \in \PP$ we have
\begin{equation}\label{eq:ld-ball}
\limsup_{\varepsilon \to 0} \limsup_{N \to \infty} N^{-\gamma} \log \Pp^{N}\left(\mu^{\eta^{N}} \in B(\pi, \varepsilon) \right) \leq - I(\pi),
\end{equation}
where $B(\pi, \varepsilon)$ denotes the open $\varepsilon$-ball around $\pi$ in the Wasserstein distance $d_{\mathcal{W}}$ on $\PP$.

Fix a finite set of times $0 = t_{0} < t_{1} < \ldots < t_{k} = T$. Since almost surely the interchange process does not make jumps at any of the prescribed times $t_0, t_1, \ldots, t_k$, by continuity of projections for any $\varepsilon > 0$ there exists $\varepsilon' > 0$ such that
\begin{equation}\label{eq:upper-bound-projection}
\Pp^{N}\left( d_{\mathcal{W}}(\mu^{\eta^{N}}, \pi) < \varepsilon' \right) \leq \Pp^{N} \left(d(\mu^{\eta^{N}}_{t_{0}, t_{1}}, \pi_{t_0, t_1}) < \varepsilon) \wedge \ldots \wedge d(\mu^{\eta^{N}}_{t_{k-1}, t_{k}},\pi_{t_{k-1}, t_{k}}) < \varepsilon \right),
\end{equation}
where $d$ denotes the Wasserstein distance on $\MM([0,1]^2)$. Furthermore, note that the permutation process with distribution $\mu^{\eta^{N}}$ has independent increments, i.e., the permutations $\left(\eta_{t_{j}}^{N}\right)^{-1} \eta_{t_{j+1}}^{N}$ for any family non-overlapping intervals $[t_{j}, t_{j+1})$ are independent. Thus we can write
\begin{align}\label{eq:upper-bound-product}
\Pp^{N} \left(d(\mu^{\eta^{N}}_{t_{0}, t_{1}}, \pi_{t_0, t_1}) < \varepsilon) \wedge \ldots \wedge d(\mu^{\eta^{N}}_{t_{k-1}, t_{k}},\pi_{t_{k-1}, t_{k}}) < \varepsilon \right) = \prod\limits_{i=0}^{k-1} \Pp^{N} \left( d(\mu^{\eta^{N}}_{t_{i}, t_{i+1}}, \pi_{t_{i}, t_{i+1}}) < \varepsilon)\right).
\end{align}

In this way we have reduced the problem to bounding the probability that the random measure $\mu^{\eta^{N}}_{t_{i}, t_{i+1}}$ is close to a fixed permuton $\pi_{t_{i}, t_{i+1}}$.

Fix $i$ and consider all permutations $\sigma \in \mathcal{S}_{N}$ such that the empirical measure $\mu_{\sigma}$ satisfies $d(\mu_{\sigma}, \pi_{t_{i}, t_{i+1}}) < \varepsilon$. As there are at most $N!$ such permutations, by performing a union bound over this set we obtain
\[
\Pp^{N} \left( d(\mu^{\eta^{N}}_{t_{i}, t_{i+1}}, \pi_{t_{i}, t_{i+1}}) < \varepsilon)\right) \leq N! \sup_{ \substack{\sigma \in \mathcal{S}_{N} \\ d(\mu_{\sigma}, \pi_{t_{i}, t_{i+1}}) < \varepsilon}} \Pp^{N} \left(\mu^{\eta^{N}}_{t_{i}, t_{i+1}} = \mu_{\sigma}\right),
\]
where on the right hand side we have the probability that the random measure $\mu_{t_i, t_{i+1}}^{\eta^{N}}$ is equal to $\mu_{\sigma}$. This probability is simply equal to $\Pp^{N} \left( \left(\eta_{t_i}^{N}\right)^{-1} \eta_{t_{i+1}}^{N} = \sigma)\right)$ and by stationarity of the interchange process this is the same as $\Pp^{N} \left( \left(\eta_{0}^{N}\right)^{-1} \eta_{t_{i+1} - t_i}^{N} = \sigma)\right)$.
By employing Proposition \ref{prop:one-slice-bound}, with $\sigma^N \in \mathcal{S}_N$ being any permutation attaining the supremum above, and noticing that $\log N! = o(N^{\gamma})$ we get
\[
\limsup_{N \to \infty} N^{-\gamma} \log \Pp^{N} \left( d(\mu^{\eta^{N}}_{t_i, t_{i+1}}, \pi_{t_{i}, t_{i+1}}) < \varepsilon)\right) \leq \limsup_{N \to \infty} \left(- \frac{1}{t_{i+1} - t_i} I(\sigma^N)\right).
\]
Now observe that for any $\sigma$ such that $d(\mu_{\sigma}, \pi_{t_{i}, t_{i+1}}) < \varepsilon$ the energy $I(\sigma) = I(\mu_{\sigma})$ has to be close to $I(\pi_{t_{i}, t_{i+1}})$, the energy of the permuton $\pi_{t_i, t_{i+1}}$ (recall definition \ref{eq:permuton-energy-def}), since $I$ is continuous in the weak topology on $\MM([0,1]^2)$. Thus upon taking $\varepsilon \to 0$ we obtain
\[
\limsup_{\varepsilon \to 0} \limsup_{N \to \infty} N^{-\gamma} \log \Pp^{N} \left(d(\mu^{\eta^{N}}_{t_i, t_{i+1}}, \pi_{t_{i}, t_{i+1}}) < \varepsilon \right) \leq -\frac{1}{t_{i+1} - t_{i}} I(\pi_{t_{i}, t_{i+1}}) .
\]
Applying this estimate to the product in \eqref{eq:upper-bound-product} and observing that in \eqref{eq:upper-bound-projection} without loss of generality we can assume $\varepsilon' \leq \varepsilon$, we arrive at the following bound
\begin{align*}
& \limsup_{\varepsilon \to 0} \limsup_{N \to \infty} N^{-\gamma} \log \Pp^{N}\left(d_{\mathcal{W}}(\mu^{\eta^{N}}, \pi) < \varepsilon \right) \leq -\sum\limits_{i=0}^{k-1} \frac{1}{t_{i+1} - t_{i}} I(\pi_{t_{i}, t_{i+1}}).
\end{align*}
Since $t_0, t_1, \ldots, t_k$ were arbitrary, by optimizing over all finite partitions $\Pi = \{ 0 = t_{0} < t_{1} < \ldots < t_{k} = T \}$ we obtain
\[
\limsup_{\varepsilon \to 0} \limsup_{N \to \infty} N^{-\gamma} \log \Pp^{N}\left(d_{\mathcal{W}}(\mu^{\eta^{N}}, \pi) < \varepsilon \right) \leq - \sup_{\Pi} \sum\limits_{i=0}^{k-1} \frac{1}{t_{i+1} - t_{i}} I(\pi_{t_{i}, t_{i+1}}).
\]
Recalling the definitions \eqref{eq:energy-path-def}, \eqref{eq:process-energy} and \eqref{eq:def-energy-fin-dim}, to prove \eqref{eq:ld-ball} it remains to show that we have $I(\pi) = \sup\limits_{\Pi} I^{\Pi}(\pi)$, which is exactly the statement of Lemma \ref{lm:approximate-energy}.
\end{proof}

\begin{proposition}\label{prop:exp-tightness}
The family of measures $\mu^{\eta^{N}}$ is exponentially tight, that is, there exists a sequence of compact sets $K_{m} \subseteq \PP$ such that
\[
\limsup_{N \to \infty} N^{-\gamma} \log \Pp^{N} (\mu^{\eta^{N}} \notin K_{m}) \leq - m.
\]
\end{proposition}

\begin{proof}

The idea of the proof is to show that having many particles whose trajectories have poor modulus of continuity (which would spoil compactness) necessarily implies the process having high energy, which by Corollary \ref{cor:many-slices-energy} is unlikely.

Recall that $\DD = \DD([0,T], [0,1])$ is the space of all c\`adl\`ag paths from $[0,T]$ to $[0,1]$. It will be convenient to work with the following notion of c\`adl\`ag modulus of continuity -- for a path $f \in \DD$ we define
\[
w_{\delta}''(f) = \sup_{\substack{t_1 \leq t \leq t_2 \\ t_2 - t_1 \leq \delta}} \left\{ |f(t) - f(t_1)| \wedge |f(t_2) - f(t)| \right\}.
\]
By a characterization of compactness in the Skorokhod space (\cite[Theorem 12.4]{billingsley}) a set $A \subseteq \DD$ has compact closure if and only if the following conditions hold
\[
\begin{cases}
\sup\limits_{f \in A} \sup\limits_{t \in [0,T]} |f(t)| < \infty, \\
\lim\limits_{\delta \to 0} \sup\limits_{f \in A} w_{\delta}''(f) = 0, \\
\lim\limits_{\delta \to 0} \sup\limits_{f \in A} |f(\delta) - f(0)| = 0, \\
\lim\limits_{\delta \to 0} \sup\limits_{f \in A} |f(T-) - f(T-\delta)| = 0.
\end{cases}
\]
In our setting the first condition is trivially satisifed. To exploit the other conditions let us introduce for any $m,r \geq 1$ the following sets
\begin{align*}
& K^{w}_{m,r} = \bigcap\limits_{k \geq 1}\left\{ f \in \DD \, \big\vert \, w_{\delta_{k}(m,r)}''(f) \leq \varepsilon_{k} \right\}, \\
& K^{0}_{m,r} = \bigcap\limits_{k \geq 1}\left\{ f \in \DD \, \big\vert \, |f(\delta_{k}(m,r)) - f(0)| \leq \varepsilon_{k} \right\}, \\
& K^{T}_{m,r} = \bigcap\limits_{k \geq 1}\left\{ f \in \DD \, \big\vert \, |f(T-) - f(T - \delta_{k}(m,r))| \leq \varepsilon_{k} \right\},
\end{align*}
and
\[
K_{m,r} = K^{w}_{m,r} \cap K^{0}_{m,r} \cap K^{T}_{m,r},
\]
where $\varepsilon_{k} = 4^{-k}$ and $\delta_{k}(m,r)$ will be appropriately chosen later. We will assume that for fixed $m,r$ we have $\lim\limits_{k \to \infty} \delta_{k}(m,r) = 0$ and that for any $k \geq 1$ both $\frac{T}{\delta_k}$ and $\frac{\delta_{k}(m,r)}{\delta_{k+1}(m,r)}$ are integer (the latter assumption is for simplicity of notation only). Note that by the aforementioned compactness conditions each set $K_{m,r}$ has compact closure in $\DD$.

Let
\[
K_m = \bigcap\limits_{r \geq 1} \left\{ \mu \in \MM(\DD) \, \bigg\vert \, \mu(K_{m,r}) \geq 1 - \frac{1}{r} \right\}.
\]
We claim that $K_m$ has compact closure in $\MM(\DD)$. Indeed, by Prokhorov's theorem it is enough to prove that $K_m$ is tight. If $\mu \in K_m$, then for any $r \geq 1$ we have $\mu(K_{m,r}^{c}) \leq \frac{1}{r}$, so the sets $K_{m,r}$ form the family of compact sets needed for tightness of $K_m$.

The sets $K_m$ (possibly after taking their closures) will form the family of compact sets needed for exponential tightness. Thus our goal is to bound $\Pp^{N}(\mu^{\eta^{N}} \notin K_{m})$. Let us write
\[
\Pp^{N}(\mu^{\eta^{N}} \notin K_{m}) = \Pp^{N}\left(\exists_{r \geq 1} \, \, \mu^{\eta^{N}} (K_{m,r}^{c}) \geq \frac{1}{r} \right) \leq \sum\limits_{r \geq 1} \Pp^{N}\left(\mu^{\eta^{N}} (K_{m,r}^{c}) \geq \frac{1}{r}\right).
\]
It is enough to show that for any $m,r \geq 1$ and any $N \geq 1$ we have
\begin{equation}\label{eq:exp-tightness-m}
\Pp^{N}\left(\mu^{\eta^{N}} (K_{m,r}^{c}) \geq \frac{1}{r} \right) \leq C e^{-mr N^{\gamma}},
\end{equation}
where $C > 0$ is some global constant.

For any given $m$ and $r$, observe that $\mu^{\eta^{N}} (K_{m,r}^{c}) \geq \frac{1}{r}$ means that in $\eta^N$ we have at least $\frac{N}{r}$ particles with paths $f \notin K_{m,r}$. Since $K_{m,r} = K^{w}_{m,r} \cap K^{0}_{m,r} \cap K^{T}_{m,r}$, clearly it is enough to estimate separately the probabilities that at least $\frac{N}{3r}$ particles have paths respectively not in $K^{w}_{m,r}$, $K^{0}_{m,r}$ or $K^{T}_{m,r}$. The argument for $K^{0}_{m,r}$ and $K^{T}_{m,r}$ is much simpler, so we skip it and concentrate only on the case of $K^{w}_{m,r}$. For simplicity we will write $\alpha(r) = \frac{1}{3r}$

For fixed $m$ and $r$ we will call a path $f$ \emph{bad} if $w_{\delta_{k}(m,r)}''(f) > \varepsilon_{k}(m,r)$ for some $k \geq 1$. We will call $f$ \emph{bad exactly at scale $k$} if $w_{\delta_{k}(m,r)}''(f) > \varepsilon_{k}(m,r)$, but $w_{\delta_{j}(m,r)}''(f) \leq \varepsilon_{j}(m,r)$ for all $j \geq k+1$. Recalling the definition of the set $K^{w}_{m,r}$, the event whose probability we would like to bound is
\[
A_{N}^{m,r} = \left\{ \mbox{there exist} \geq \alpha(r) N \mbox{ particles with bad paths} \right\}.
\]
Consider now the events
\[
B_{N}^{m,r,k} = \left\{ \mbox{there exist} \geq \frac{\alpha(r)}{2^k} N \mbox{ particles whose paths are bad exactly at scale $k$} \right\}.
\]
Note that if $f$ is a bad path with jumps of fixed size $\frac{1}{N}$, then there exists $k \geq 1$ such that $f$ is bad exactly at scale $k$ (since all paths we are considering are c\'{a}dl\`{a}g). Thus we have $A_{N}^{m,r} \subseteq \bigcup\limits_{k \geq 1} B_{N}^{m,r,k}$, so 
\[
\Pp^{N}\left(\mu^{\eta^{N}} ((K_{m,r}^{w})^c) \geq \alpha(r)\right) = \Pp^{N}\left(A_{N}^{m,r}\right) \leq \sum\limits_{k \geq 1} \Pp^{N}(B_{N}^{m,r,k}).
\]
Thus it is enough to show that for any $m,r,k \geq 1$ and any $N \geq 1$ we have
\begin{equation}\label{eq:bound-on-b}
\Pp^{N}(B_{N}^{m,r,k}) \leq e^{-mrk N^\gamma}.
\end{equation}

From now on we fix $m,r,k$ and $N$. All paths we are considering are assumed to come from the interchange process $\eta^N$, in particular they have jumps of fixed size $\frac{1}{N}$. For the sake of brevity we will simply write $\delta_k = \delta_k(m,r)$.

Let us divide the interval $[0,T]$ into $J = \frac{T}{\delta_k}$ intervals of the form $[j \delta_k, (j+1)\delta_k]$, $j=0, \ldots, J-1$. Consider any path $f$ which is bad exactly at scale $k$. The condition $w_{\delta_{k}}''(f) > \varepsilon_{k}$ implies that for some $t, t_1, t_2$ such that $t ,t_2 \in [t_1, t_1 + \delta_k]$ we have $|f(t) - f(t_1)| > \varepsilon_k$ and $|f(t_2) - f(t)| > \varepsilon_k$. A simple application of the triangle inequality implies that there exists $j \in \{0, \ldots, J-1 \}$ and $t' \in [j \delta_k, (j+1)\delta_k]$ such that $|f(j\delta_k) - f(t')| > \frac{\varepsilon_k}{2}$.

Let us consider the interval $[s, s'] = [j \delta_k, (j+1)\delta_k]$ obtained above. Subdivide it into $L = \frac{\delta_k}{\delta_{k+1}}$ intervals of the form $[s_{\ell}, s_{\ell+1}]$, $\ell = 0, \ldots, L-1$, where $s_{\ell} = s + \ell \delta_{k+1}$. For $\ell=0, \ldots, L-1$ let $\Delta_{\ell}(f) = |f(s_{\ell}) - f(s_{\ell+1})|$.

The crucial observation is that $\sum\limits_{\ell=0}^{L-1} \Delta_{\ell}(f) > \frac{\varepsilon_k}{4}$. To see this, consider $t'$ such that $|f(s) - f(t')| > \frac{\varepsilon_k}{2}$, obtained above, and let $\tilde{\ell}$ be such that $t' \in (s_{\tilde{\ell}}, s_{\tilde{\ell} + 1}]$. By the triangle inequality we have
\[
\frac{\varepsilon_{k}}{2} < |f(s) - f(t')| \leq \sum\limits_{\ell=0}^{\tilde{\ell} - 2} \Delta_{\ell}(f) + |f(s_{\tilde{\ell}}) - f(t')|.
\]
Since $f$ is bad exactly at scale $k$, we have $w_{\delta_{k+1}}''(f) \leq \varepsilon_{k+1}$, which together with $|s_{\tilde{\ell}} - t'| \leq \delta_{k+1}$ implies $|f(s_{\tilde{\ell}}) - f(t')| \leq \varepsilon_{k+1} = 4^{-(k+1)} = \frac{\varepsilon_{k}}{4}$. Thus necessarily $\sum\limits_{\ell=0}^{\tilde{\ell} - 2} \Delta_{\ell}(f) > \frac{\varepsilon_k}{4}$. From this we obtain
\begin{equation}\label{eq:sum-of-delta}
\frac{\varepsilon_k^2}{16} < \left( \sum\limits_{\ell=0}^{L-1} \Delta_{\ell}(f) \right)^2 \leq L \sum\limits_{\ell=0}^{L-1} \Delta_{\ell}(f)^2,
\end{equation}
where the right-hand side estimate follows from the Cauchy-Schwarz inequality.

Now let us suppose that the event $B_{N}^{m,r,k}$ holds. Then there exist at least $\frac{\alpha(r)}{2^k}N$ paths $f_i$ for which the estimate \eqref{eq:sum-of-delta} holds. Consider the partition $\Pi = \{0 = t_0 < t_1 < \ldots < t_n = T \}$ where $n = \frac{T}{\delta_{k+1}}$, $t_j = j \delta_{k+1}$ for $j=0, \ldots, n$. Recalling that $f_i = \frac{1}{N}\eta^{N}(i)$, the definition of $\Delta_{\ell}(f)$ and the definition \eqref{eq:def-energy-fin-dim} of the energy $I^{\Pi}(\mu^{\eta^{N}})$ we obtain that on $B_{N}^{m,r,k}$ we have
\begin{align*}
I^{\Pi}(\mu^{\eta^{N}}) =  &\frac{1}{N} \sum\limits_{i=1}^{N} \left( \frac{1}{2} \sum\limits_{j=1}^{n} \frac{| f_{i}(t_{j}) - f_{i}(t_{j-1}) |^2}{t_{j} - t_{j-1}} \right) = \\
& \frac{1}{N} \sum\limits_{i=1}^{N} \left(\frac{1}{2\delta_{k+1}} \sum\limits_{j=1}^{n} | f_{i}(t_{j}) - f_{i}(t_{j-1}) |^2 \right) >
\frac{1}{N} \cdot \frac{\alpha(r)}{2^k}N \cdot \left( \frac{1}{2\delta_{k+1}} \frac{\varepsilon_k^2}{16 L} \right) = \\
& \frac{\alpha(r)}{2^{k+5}} \frac{1}{\delta_{k+1}} \frac{\varepsilon_k^2}{ \frac{\delta_k}{\delta_{k+1}}} = \frac{\varepsilon_k^2}{\delta_k} \frac{\alpha(r)}{2^{k+5}}.
\end{align*}
Writing again $\delta_k = \delta_k(m,r)$, we have thus obtained the bound
\[
\Pp^{N}(B_{N}^{m,r,k}) \leq \Pp^{N}\left( I^{\Pi}(\mu^{\eta^{N}}) \geq \frac{\varepsilon_k^2}{\delta_k(m,r)} \frac{\alpha(r)}{2^{k+5}}\right).
\]
Recalling $\varepsilon_k = 4^{-k}$, $\alpha(r) = \frac{1}{3r}$, we see that to prove \eqref{eq:bound-on-b} it is sufficient to take $\delta_k(m,r)$ small enough so that
\[
\frac{4^{-2k}}{\delta_k(m,r)} \frac{1}{3r 2^{k+5}} \geq 2 mrk.
\]
By applying Corollary \ref{cor:many-slices-energy} we obtain that
\[
\Pp^{N}(B_{N}^{m,r,k}) \leq \Pp^{N}\left( I^{\Pi}(\mu^{\eta^{N}}) \geq 2mrk \right)\leq e^{-mrkN^{\gamma}}
\]
for $N$ large enough. By taking $\delta_k(m,r)$ even smaller if necessary we can make this estimate true for all values of $N \geq 1$, which proves \eqref{eq:bound-on-b} and finishes the proof of exponential tightness.
\end{proof}

\end{section}

\begin{section}{Asymptotics of relaxed sorting networks}\label{sec:asymptotics}

In this section we prove the limiting behavior of random relaxed sorting networks, given by Theorem \ref{th:lln-for-relaxed-networks}, and the asymptotic counting formula of Theorem \ref{th:asymptotics-relaxed}. With the large deviation bounds obtained in the preceding sections both of the proofs are now rather straightforward.

\begin{proof}[Proof of Theorem \ref{th:lln-for-relaxed-networks}]
Let $\mathcal{R} \subseteq \PP$ be the set of permuton processes $X$ reaching exactly the reverse permuton at time $1$, i.e., such that $(X_0, X_1) \sim (X, 1 - X)$, and likewise let $\mathcal{R}_N$ be the set of permutation processes on $N$ elements reaching exactly the reverse permutation $\mathrm{rev}_N = (N \, \ldots \, 2 \, 1)$ at time $1$. Let $\mathcal{R}_{\delta}$ denote the $\delta$-neighborhood in the Wasserstein distance on $\PP$ of the set $\mathcal{R} \cup \bigcup\limits_{N \geq 1} \mathcal{R}_N$.

Let $\eta^N$ be the interchange process with $\alpha = 1 + \kappa \in (1,2)$ and let $\mu^{\eta^N}$ be the distribution of the corresponding permutation process. By definition of a random relaxed sorting network, for any given $\delta > 0$ we have for sufficiently large $N$
\[
\Pp^N \left( \pi^{N}_{\delta} \in B(\pi_{\mathcal{A}}, \varepsilon) \right) = \Pp^N \left( \mu^{\eta^N} \in B(\pi_{\mathcal{A}}, \varepsilon) \big\vert \mu^{\eta^N} \in \mathcal{R}_{\delta} \right).
\]
Now, we have
\[
\Pp^N \left( \mu^{\eta^N} \notin B(\pi_{\mathcal{A}}, \varepsilon) \big\vert \mu^{\eta^N} \in \mathcal{R}_{\delta} \right) = \frac{1}{\Pp^N\left( \mu^{\eta^N} \in \mathcal{R}_{\delta} \right)}\Pp^N \left( \left\{ \mu^{\eta^N} \notin B(\pi_{\mathcal{A}}, \varepsilon) \right\} \cap \left\{ \mu^{\eta^N} \in \mathcal{R}_{\delta} \right\} \right).
\]
By the large deviation lower bound of Theorem \ref{thm:lower-bound-for-minimizers} we have
\[
\Pp^N\left( \mu^{\eta^N} \in \mathcal{R}_{\delta} \right) \geq \exp \left\{-N^\gamma \left( I(\pi_{\mathcal{A}}) + o(1) \right) \right\},
\]
where $\gamma = 3 - \alpha$.

Let $\mathcal{C}_{\varepsilon, \delta} = B(\pi_{\mathcal{A}}, \varepsilon)^{c} \cap \overline{\mathcal{R}_{\delta}}$. By the large deviation upper bound of Theorem \ref{th:upper-bound} we have
\[
\Pp^N \left( \mu^{\eta^N} \in \mathcal{C}_{\varepsilon, \delta} \right) \leq \exp \left\{ -N^\gamma \left( \inf\limits_{\mu \in \mathcal{C}_{\varepsilon, \delta}} I(\mu) + o(1)\right) \right\}.
\]
Since $\pi_{\mathcal{A}}$ is the unique minimizer of energy on $\mathcal{R}$ (\cite{brenier}, \cite{mustazee}), given $\varepsilon > 0$ there exists $\beta = \beta(\varepsilon) > 0$ such that
\begin{equation}\label{eq:infimum-sine}
\inf\limits_{\mu \in B(\pi_{\mathcal{A}}, \varepsilon)^c \cap \mathcal{R}} I(\mu) \geq I(\pi_{\mathcal{A}}) + \beta.
\end{equation}

Since $I(\cdot)$ is lower semi-continuous, by \eqref{eq:infimum-sine} we obtain (possibly after adjusting $\delta$ to replace $\overline{\mathcal{R}_{\delta}}$ with $\mathcal{R}_{\delta}$) that for all sufficiently small $\delta$ we have
\[
\inf\limits_{\mu \in B(\pi_{\mathcal{A}}, \varepsilon)^c \cap \mathcal{R}_{\delta}} I(\mu) \geq I(\pi_{\mathcal{A}}) + \frac{\beta}{2}.
\]
Altogether we obtain that for all sufficiently small $\delta$ (depending on $\varepsilon$ only)
\[
\Pp^N \left( \mu^{\eta^N} \notin B(\pi_{\mathcal{A}}, \varepsilon) \big\vert \mu^{\eta^N} \in \mathcal{R}_{\delta} \right) \leq e^{N^{\gamma} \left( I(\pi_{\mathcal{A}}) + o(1) \right)} e^{- N^{\gamma} \left( I(\pi_{\mathcal{A}}) + \frac{\beta}{2} + o(1) \right)} = e^{-N^{\gamma} \left( \frac{\beta}{2} + o(1) \right)}
\]
and the right-hand side goes to $0$ as $N \to \infty$
\end{proof}

\begin{proof}[Proof of Theorem \ref{th:asymptotics-relaxed}]
Let $\Pp^N$ denote the law of the interchange process with $\alpha = 1 + \kappa$. Let $J$ be the number of all particle swaps in the process and let $M = \left\lfloor \frac{1}{2} N^{\alpha}(N-1) \right\rfloor$.

Let $\mathcal{S}_{\delta}$ denote the $\delta$-neighborhood of the reverse permuton in the Wasserstein distance on $\MM([0,1]^2)$. Observe that for sufficiently large $N$ we have for any $k \leq M$
\begin{equation}\label{eq:comparison-with-k}
\Pp^N \left( \mu^{\eta^N}_{0,T} \in \mathcal{S}_{\delta} \big\vert J = k\right) \leq \Pp^N \left( \mu^{\eta^N}_{0,T} \in \mathcal{S}_{\delta} \big\vert J = M\right)(1 + o(1)).
\end{equation}
This is because if the process has done $k$ swaps up to time $T_k$ and $\mu^{\eta^N}_{0,T_k} \in \mathcal{S}_{\delta}$, then with high probability $\mu^{\eta^N}_{0,T_M} \in \mathcal{S}_{\delta}$ as well, since $\mathcal{S}_{\delta}$ is an open set in $\MM([0,1]^2)$ and the additional number of steps done between $T_k$ and $T_M$ is $\leq \frac{1}{2} N^{\alpha}(N-1) = o(N^3)$, so typically almost all particles have negligible displacement.

On the other hand, since in the interchange process each sequence of swaps of given length is equally likely, we have
\[
\Pp^N \left( \mu^{\eta^N}_{0,T} \in \mathcal{S}_{\delta} \big\vert J = M\right) = \frac{|\mathcal{S}^{N}_{\kappa, \delta}|}{|\PP^{N}_{M}|},
\]
where $\PP^{N}_{M}$ is the set of all sequences of adjacent transpositions of length $M$. Summing \eqref{eq:comparison-with-k} over $k \leq M$ we obtain
\[
\Pp^N \left( \mu^{\eta^N}_{0,T} \in \mathcal{S}_{\delta} \cap \{ J \leq M \}\right) \leq \Pp^N \left( J \leq M \right) \frac{|\mathcal{S}^{N}_{\kappa, \delta}|}{|\PP^{N}_{M}|}(1 + o(1)).
\]
Since under $\Pp^N$ $J$ has Poisson distribution with mean $\frac{1}{2}N^{\alpha}(N-1)$, we have $\Pp^N \left( J \leq M \right) \to 1/2$ as $N \to \infty$.

To estimate the left-hand side, let $\widetilde{\Pp}^N$ be the law of the biased interchange process corresponding to the sine curve process $\pi_{\mathcal{A}}$. Recall Lemma \ref{lm:radon-nikodym-lower-bound} and for fixed $\varepsilon > 0$ let $A$ be the event that the $o(1)$ term in the formula for $\frac{\dd \Pp_{u}^{N}}{\dd \widetilde{\Pp}^{N}}(T)$ is at most $\varepsilon$. 
Let us write
\[
\Pp^N \left( \mu^{\eta^N}_{0,T} \in \mathcal{S}_{\delta} \cap \{ J \leq M \} \cap A\right) \leq \Pp^N \left( \mu^{\eta^N}_{0,T} \in \mathcal{S}_{\delta} \cap \{ J \leq M \}\right)
\]
and
\begin{align*}
& \Pp^N \left( \mu^{\eta^N}_{0,T} \in \mathcal{S}_{\delta} \cap \{ J \leq M \} \cap A \right) = \\
& = \widetilde{\Pp}^N \left( \mu^{\eta^N}_{0,T} \in \mathcal{S}_{\delta} \cap \{ J \leq M \} \cap A \right) \frac{\Pp^N \left( \mu^{\eta^N}_{0,T} \in \mathcal{S}_{\delta} \cap \{ J \leq M \}\cap A \right)}{\widetilde{\Pp}^N \left( \mu^{\eta^N}_{0,T} \in \mathcal{S}_{\delta} \cap \{ J \leq M \} \cap A \right)}.
\end{align*}
By Theorem \ref{th:lln} $\mu^{\eta^N}_{0,T} \in \mathcal{S}_{\delta}$ has high probability under $\widetilde{\Pp}^N$ and, since the particle swap rates for the biased process sum up to $\frac{1}{2}N^{\alpha}(N-1)$ (recall \eqref{eq:biased-generator}), we have similarly as for the unbiased process $\widetilde{\Pp}^N \left( J \leq M \right) \to 1/2$ as $N \to \infty$. By Lemma \ref{lm:radon-nikodym-lower-bound} $A$ is a high probability event under $\widetilde{\Pp}^N$ as well.

To estimate the remaining probabilities, we employ the formula for the Radon-Nikodym derivative from Lemma \ref{lm:radon-nikodym-lower-bound}. Since in the biased process with high probability the energy term in the derivative is close to $I(\pi_{\mathcal{A}}) = \frac{\pi^2}{6}$, we obtain
\[
\frac{\Pp^N \left( \mu^{\eta^N}_{0,T} \in \mathcal{S}_{\delta} \cap \{ J \leq M \}\cap A \right)}{\widetilde{\Pp}^N \left( \mu^{\eta^N}_{0,T} \in \mathcal{S}_{\delta} \cap \{ J \leq M \} \cap A \right)} \geq e^{-N^{\gamma} \left( \frac{\pi^2}{6} + \varepsilon \right) + o(N^\gamma)},
\]
where $\gamma = 3 - \alpha$.

Altogether we obtain
\[
|\mathcal{S}^{N}_{\kappa, \delta}| \geq |\PP^{N}_{M}| e^{-N^{\gamma} \left( \frac{\pi^2}{6} + \varepsilon \right) + o(N^\gamma)}.
\]
Since $|\PP^{N}_{M}| = (N-1)^M = e^{\lfloor \frac{1}{2} N^{\alpha}(N-1) \rfloor \log(N-1)}$ and $\varepsilon$ was arbitrary, we obtain the asymptotic lower bound on $|\mathcal{S}^{N}_{\kappa, \delta}|$ as claimed.

For the upper bound, let $\mathcal{R}_{\delta}$ be as in the previous theorem. By the large deviation upper bound of Theorem \ref{th:upper-bound} we have
\[
\Pp^N\left( \mu^{\eta^N}_{0,T} \in \mathcal{R}_{\delta} \right) \leq \exp \left\{ -N^{\gamma} \left( \inf\limits_{\mu \in \overline{\mathcal{R}_{\delta}}} I(\mu) + o(1) \right) \right\},
\]
Since $I$ is lower semi-continuous, given any $\varepsilon > 0$ we have for all sufficiently small $\delta > 0$
\[
\inf\limits_{\mu \in \overline{\mathcal{R}_{\delta}}} I(\mu) \geq \inf\limits_{\mu \in \mathcal{R}} I(\mu) - \varepsilon = I(\pi_{\mathcal{A}}) - \varepsilon,
\]
where again we have used the energy minimization property of $\pi_{\mathcal{A}}$. Since $I(\pi_{\mathcal{A}}) = \frac{\pi^2}{6}$, this implies that for any $\varepsilon > 0$ and sufficiently small $\delta > 0$
\[
\Pp^N\left( \mu^{\eta^N} \in \mathcal{R}_{\delta} \right) \leq e^{-N^{\gamma} \left( I(\pi_{\mathcal{A}}) - \varepsilon + o(1) \right)}.
\]
Now we estimate
\[
\Pp^N\left( \mu^{\eta^N} \in \mathcal{R}_{\delta} \right) \geq \Pp^N\left( \mu^{\eta^N} \in \mathcal{R}_{\delta} \big\vert J = M\right) \Pp^N \left( J = M \right) = \frac{|\mathcal{S}^{N}_{\kappa, \delta}|}{|\PP^{N}_{M}|}\Pp^N \left( J = M \right)
\]
and use the same asymptotic estimate for $|\PP^{N}_{M}|$ as in the lower bound. Since $J$ is Poisson with mean $\frac{1}{2}N^\alpha (N-1)$ under $\Pp^N$, the second term on the right-hand side is $e^{O(\log N)}$. Altogether we obtain
\[
|\mathcal{S}^{N}_{\kappa, \delta}| \leq \exp\left\{ \frac{1}{2}N^{1 + \kappa} (N-1) \log(N-1)- N^{2 - \kappa} \left( I(\pi_{\mathcal{A}}) - \varepsilon \right) + o(N^{2 - \kappa})\right\},
\]
which proves the desired asymptotic upper bound on $|\mathcal{S}^{N}_{\kappa, \delta}|$.
\end{proof}

\end{section}

\bibliography{bibliography}{}

\newcommand{\etalchar}[1]{$^{#1}$}
\providecommand{\MR}[1]{}
\providecommand{\bysame}{\leavevmode\hbox to3em{\hrulefill}\thinspace}
\providecommand{\MR}{\relax\ifhmode\unskip\space\fi MR }
\providecommand{\MRhref}[2]{%
  \href{http://www.ams.org/mathscinet-getitem?mr=#1}{#2}
}
\providecommand{\href}[2]{#2}
\begin{thebibliography}{HKM{\etalchar{+}}13}

\bibitem[AF09]{ambrosio-figalli}
Luigi Ambrosio and Alessio Figalli, \emph{Geodesics in the space of
  measure-preserving maps and plans}, Archive for Rational Mechanics and
  Analysis \textbf{194} (2009), no.~2, 421--462.

\bibitem[AHRV07]{sorting}
Omer Angel, Alexander~E. Holroyd, Dan Romik, and B{\'a}lint Vir{\'a}g,
  \emph{Random sorting networks}, Adv. Math. \textbf{215} (2007), no.~2,
  839--868. \MR{2355610}

\bibitem[Ald83]{aldous}
David Aldous, \emph{Random walks on finite groups and rapidly mixing {M}arkov
  chains}, S{\'e}minaire de Probabilit{\'e}s XVII 1981/82 (Berlin, Heidelberg)
  (Jacques Az{\'e}ma and Marc Yor, eds.), Springer Berlin Heidelberg, 1983,
  pp.~243--297.

\bibitem[BFS09]{bernot-figalli-santambrogio}
Marc Bernot, Alessio Figalli, and Filippo Santambrogio, \emph{Generalized
  solutions for the {E}uler equations in one and two dimensions}, Journal de
  Mathématiques Pures et Appliquées \textbf{91} (2009).

\bibitem[Bil13]{billingsley}
P.~Billingsley, \emph{Convergence of probability measures}, Wiley Series in
  Probability and Statistics, Wiley, 2013.

\bibitem[Bre89]{brenier}
Yann Brenier, \emph{The least action principle and the related concept of
  generalized flows for incompressible perfect fluids}, Journal of the American
  Mathematical Society \textbf{2} (1989), no.~2, 225--255.

\bibitem[Bre99]{brenier-distribution}
\bysame, \emph{Minimal geodesics on groups of volume-preserving maps and
  generalized solutions of the {E}uler equations}, Communications on Pure and
  Applied Mathematics \textbf{52} (1999), no.~4, 411--452.

\bibitem[Bre08]{brenier-physica}
\bysame, \emph{Generalized solutions and hydrostatic approximation of the
  {E}uler equations}, Physica D \textbf{237} (2008), 1982--1988.

\bibitem[CL55]{ode-book}
A.~Coddington and N.~Levinson, \emph{Theory of ordinary differential
  equations}, International series in pure and applied mathematics, McGraw-Hill
  Companies, 1955.

\bibitem[Dau22]{duncan}
Duncan Dauvergne, \emph{The {A}rchimedean limit of random sorting networks},
  Journal of the American Mathematical Society \textbf{35} (2022), no.~4,
  1215--1267.

\bibitem[FT04]{Fritz2004}
J{\'o}zsef Fritz and B{\'a}lint T{\'o}th, \emph{Derivation of the {L}eroux
  system as the hydrodynamic limit of a two-component lattice gas},
  Communications in Mathematical Physics \textbf{249} (2004), no.~1, 1--27.

\bibitem[HKM{\etalchar{+}}13]{limits}
Carlos Hoppen, Yoshiharu Kohayakawa, Carlos~Gustavo Moreira, Bal{\'a}zs
  R{\'a}th, and Rudini Menezes~Sampaio, \emph{Limits of permutation sequences},
  J. Combin. Theory Ser. B \textbf{103} (2013), no.~1, 93--113. \MR{2995721}

\bibitem[Kal21]{kallenberg}
Olav Kallenberg, \emph{Foundations of modern probability}, third ed.,
  Probability Theory and Stochastic Modelling, Springer International
  Publishing, 2021.

\bibitem[KL99]{kipnis}
Claude Kipnis and Claudio Landim, \emph{Scaling limits of interacting particle
  systems}, Grundlehren der Mathematischen Wissenschaften [Fundamental
  Principles of Mathematical Sciences], vol. 320, Springer-Verlag, Berlin,
  1999. \MR{1707314}

\bibitem[Lac16]{lacoin}
Hubert Lacoin, \emph{{Mixing time and cutoff for the adjacent transposition
  shuffle and the simple exclusion}}, The Annals of Probability \textbf{44}
  (2016), no.~2, 1426 -- 1487.

\bibitem[Lov12]{lovasz}
L{\'a}szl{\'o} Lov{\'a}sz, \emph{Large networks and graph limits}, American
  Mathematical Society Colloquium Publications, vol.~60, American Mathematical
  Society, Providence, RI, 2012. \MR{3012035}

\bibitem[RV17]{mustazee2}
Mustazee Rahman and B{\'a}lint Vir{\'a}g, \emph{Brownian motion as limit of the
  interchange process}, arXiv:1609.07745 (2017).

\bibitem[RVV19]{mustazee}
Mustazee Rahman, B{\'a}lint Vir{\'a}g, and M{\'a}t{\'e} Vizer, \emph{Geometry
  of permutation limits}, Combinatorica \textbf{39} (2019), no.~4, 933--960.

\bibitem[Shn87]{shnirelman}
A.~Shnirelman, \emph{On the geometry of the group of diffeomorphisms and the
  dynamics of an ideal incompressible fluid}, Math. Sbornik USSR \textbf{56}
  (1987), 79--105.

\bibitem[Sta84]{stanley}
Richard~P. Stanley, \emph{On the number of reduced decompositions of elements
  of {C}oxeter groups}, European J. Combin. \textbf{5} (1984), no.~4, 359--372.
  \MR{782057}

\bibitem[Var16]{varadhan}
S.R.S. Varadhan, \emph{Large deviations}, Courant Lecture Notes, Courant
  Institute of Mathematical Sciences, 2016.

\end{thebibliography}
\bibliographystyle{amsalpha}

\bigskip
\noindent
Michał Kotowski \\
Institute of Mathematics of the Polish Academy of Sciences \\
ul. Śniadeckich 8 00-656 Warszawa, Poland\\
\noindent
{E-mail:} {\tt michal.kotowski@mimuw.edu.pl} \\

\bigskip
\noindent
 B\'{a}lint Vir\'{a}g \\
 Department of Mathematics, University of Toronto \\
 40 St George St.
 Toronto, ON, M5S 2E4, Canada
 \\
\noindent
{E-mail:} {\tt balint@math.toronto.edu} \\
\end{document}